\documentclass[11pt]{article}
\usepackage{amsmath}
\usepackage{lineno,hyperref}
\usepackage{color}
\usepackage{empheq}
\usepackage{amsfonts}
\usepackage{amssymb,amsmath,amsthm, amsfonts}
\usepackage{lipsum}
\usepackage{upgreek}
\usepackage{lipsum}
\allowdisplaybreaks[4]

\DeclareMathOperator{\esssup}{\ensuremath{ess\,sup}}

\usepackage{float}

\def\sqr#1#2{{\vcenter{\vbox{\hrule height.#2pt
				\hbox{\vrule width.#2pt height#1pt \kern#1pt \vrule width.#2pt}
				\hrule height.#2pt}}}}

\newcommand{\norm}[1]{\left\Vert#1\right\Vert}

\def\sqr#1#2{{\vcenter{\vbox{\hrule height.#2pt
				\hbox{\vrule width.#2pt height#1pt \kern#1pt \vrule width.#2pt}
				\hrule height.#2pt}}}}

\begin{document}

\newtheorem{thm}{Theorem}[section]
\newtheorem{lem}[thm]{Lemma}
%\newdefinition{rmk}{Remark}
\newtheorem{remark}[thm]{Remark}
\newtheorem{pf}{proof}
\newtheorem{assumption}[thm]{Assumption}
\newtheorem{prop}[thm]{Proposition}
\newtheorem{defn}[thm]{Definition}
\newtheorem{cro}[thm]{Corollary}
\newtheorem{example}[thm]{Example}
\newcommand{\tabincell}[2]{\begin{tabular}{@{}#1@{}}#2\end{tabular}}
\newcommand{\mathbbm}[1]{\text{\usefont{U}{bbm}{m}{n}#1}}
\let\oldnorm\norm   % <-- Store original \norm as \oldnorm
\let\norm\undefined % <-- "Undefine" \norm
\DeclarePairedDelimiter\norm{\lVert}{\rVert}
\hypersetup{colorlinks,%
	citecolor=blue,%
	filecolor=blue,%
	linkcolor=blue,%
	urlcolor=blue
}

\newcommand{\E}{\mathbb{E}}
\newcommand{\EE}{\mathbb{E}}
\newcommand{\W}{\dot{W}}
\newcommand{\ud}{\ensuremath{\mathrm{d}}}
\newcommand{\Ceil}[1]{\left\lceil #1 \right\rceil}
\newcommand{\Floor}[1]{\left\lfloor #1 \right\rfloor}

\newcommand{\Lad}{\text{L}_{\text{ad}}^2}
\newcommand{\SI}[1]{\mathcal{I}\left[#1 \right]}
\newcommand{\SIB}[2]{\mathcal{I}_{#2}\left[#1 \right]}
\newcommand{\Indt}[1]{1_{\left\{#1 \right\}}}
\newcommand{\LadInPrd}[1]{\left\langle #1 \right\rangle_{\text{L}_\text{ad}^2}}
\newcommand{\LadNorm}[1]{\left|\left|  #1 \right|\right|_{\text{L}_\text{ad}^2}}
\newcommand{\Norm}[1]{\left|\left|  #1   \right|\right|}
\newcommand{\Ito}{It\^{o} }
\newcommand{\Itos}{It\^{o}'s }
\newcommand{\spt}[1]{\text{supp}\left(#1\right)}
\newcommand{\InPrd}[1]{\left\langle #1 \right\rangle}
\newcommand{\mr}{\textbf{r}}
\newcommand{\Ei}{\text{Ei}}
\newcommand{\arctanh}{\operatorname{arctanh}}
\newcommand{\ind}[1]{\mathbb{I}_{\left\{ {#1} \right\} }}

\newcommand{\steps}[1]{\vskip 0.3cm \textbf{#1}}
\newcommand{\calB}{\mathcal{B}}
\newcommand{\calD}{\mathcal{D}}
\newcommand{\calE}{\mathcal{E}}
\newcommand{\calF}{\mathcal{F}}
\newcommand{\calG}{\mathcal{G}}
\newcommand{\calK}{\mathcal{K}}
\newcommand{\calH}{\mathcal{H}}
\newcommand{\calI}{\mathcal{I}}
\newcommand{\calL}{\mathcal{L}}
\newcommand{\calM}{\mathcal{M}}
\newcommand{\calN}{\mathcal{N}}
\newcommand{\calO}{\mathcal{O}}
\newcommand{\calT}{\mathcal{T}}
\newcommand{\calP}{\mathcal{P}}
\newcommand{\calR}{\mathcal{R}}
\newcommand{\calS}{\mathcal{S}}
\newcommand{\bbN}{\mathbb{N}}
\newcommand{\bbZ}{\mathbb{Z}}
\newcommand{\myVec}[1]{\overrightarrow{#1}}
\newcommand{\sincos}{\begin{array}{c} \cos \\ \sin \end{array}\!\!}
\newcommand{\CvBc}[1]{\left\{\:#1\:\right\}}
\newcommand*{\one}{{{\rm 1\mkern-1.5mu}\!{\rm I}}}
\newcommand{\R}{\mathbb{R}}
\newcommand{\myEnd}{\hfill$\square$}
\newcommand{\ds}{\displaystyle}
\newcommand{\Shi}{\text{Shi}}
\newcommand{\Chi}{\text{Chi}}
\newcommand{\Daw}{\ensuremath{\mathrm{daw}}}
\newcommand{\Erf}{\ensuremath{\mathrm{erf}}}
\newcommand{\Erfi}{\ensuremath{\mathrm{erfi}}}
\newcommand{\Erfc}{\ensuremath{\mathrm{erfc}}}
\newcommand{\He}{\ensuremath{\mathrm{He}}}
\newcommand{\RR}{\mathbb{R} }
\def\crtL{\theta}
\def\sign{\hbox{sign}}
\newcommand{\conRef}[1]{(#1)}

\newcommand{\cov}{\text{\bf{Cov}}}
\newcommand{\var}{\text{\bf{Var}}}
\newcommand{\crb}{\color{blue}}
\newcommand{\di}{\diamond}
\newcommand{\dom}{\mbox{Dom}}
\newcommand{\hb}[1]{\textcolor{blue}{#1}}
\newcommand{\id}{\mbox{Id}}
\newcommand{\iou}{\int_{0}^{1}}
\newcommand{\iot}{\int_{0}^{t}}
\newcommand{\iott}{\int_0^T}
\newcommand{\ist}{\int_{s}^{t}}
\newcommand{\ot}{[0,t]}
\newcommand{\ott}{[0,T]}
\newcommand{\ou}{[0,1]}
\newcommand{\1}{{\bf 1}}
\newcommand{\2}{{\bf 2}}
\newcommand{\tss}{\textsuperscript}
\newcommand{\txx}{\textcolor}

\newcommand{\D}{\mathbb D}
\newcommand{\N}{\mathbb N}
\newcommand{\Q}{\mathbb Q}
\newcommand{\Z}{\mathbb Z}

\newcommand{\bb}{\mathbf{B}}
\newcommand{\be}{\beta}
\newcommand{\bg}{\mathbf{G}}
\newcommand{\bp}{\mathbf{P}}

\newcommand{\ca}{\mathcal A}
\newcommand{\cb}{\mathcal B}
\newcommand{\cac}{\mathcal C}
\newcommand{\hcac}{\hat {\mathcal C}}
\newcommand{\cd}{\mathcal D}
\newcommand{\ce}{\mathcal E}
\newcommand{\cf}{\mathcal F}
\newcommand{\cg}{\mathcal G}
\newcommand{\ch}{\mathcal H}
\newcommand{\ci}{\mathcal I}
\newcommand{\cj}{\mathcal J}
\newcommand{\ck}{\mathcal K}
\newcommand{\cl}{\mathcal L}
\newcommand{\cm}{\mathcal M}
\newcommand{\cn}{\mathcal N}
\newcommand{\cq}{\mathcal Q}
\newcommand{\cs}{\mathcal S}
\newcommand{\ct}{\mathcal T}
\newcommand{\cu}{\mathcal U}
\newcommand{\cv}{\mathcal V}
\newcommand{\cw}{\mathcal W}
\newcommand{\cz}{\mathcal Z}
\newcommand{\crr}{\mathcal R}
\newcommand{\XX}{\mathfrak X}
\newcommand{\HH}{\mathfrak H}
\newcommand{\al}{\alpha}
\newcommand{\dde}{\theta}
\newcommand{\ep}{\varepsilon}
\newcommand{\io}{\iota}
\newcommand{\ga}{\gamma}
\newcommand{\gga}{\Gamma}
\newcommand{\ka}{\kappa}
\newcommand{\la}{\lambda}
\newcommand{\laa}{\Lambda}
\newcommand{\om}{\omega}
\newcommand{\oom}{\Omega}
\newcommand{\ssi}{\Sigma}
\newcommand{\si}{\sigma}
\newcommand{\te}{\theta}
\newcommand{\tte}{\Theta}
\newcommand{\vp}{\varphi}
\newcommand{\vt}{\vartheta}
\newcommand{\ze}{\zeta}
\newcommand{\lp}{\left(}
\newcommand{\lt}{\left }
\newcommand{\rt}{\right}
\newcommand{\rp}{\right)}
\newcommand{\lc}{\left[}
\newcommand{\rc}{\right]}
\newcommand{\lcl}{\left\{}
\newcommand{\rcl}{\right\}}
\newcommand{\lln}{\left|}
\newcommand{\rrn}{\right|}
\newcommand{\lla}{\left\langle}
\newcommand{\rra}{\right\rangle}
\newcommand{\lan}{\left\langle}
\let\Section=\section
\def\section{\setcounter{equation}{0}\Section}
\newcommand{\Hh}{\mathcal{H}}
\newcommand{\e}{\varepsilon}
\def\theequation{\thesection.\arabic{equation}}
\def\RR{\mathbb{R} }
\def\NN{\mathbb{N}}
\def\EE{\mathbb{E}}
\def\Tr{{\rm Tr}}
\def\barh{B^H }

\def\problem{{\bf Problem\ \  \ }}
\def\wt{\widetilde}
\def\bfa{{\bf a}}
\def\bfb{{\bf b}}
\def\bfc{{\bf c}}
\def\bfd{{\bf d}}
\def\bfe{{\bf
e}}
\def\bff{{\bf f}}
\def\bfg{{\bf g}}
\def\bfh{{\bf h}}
\def\bfi{{\bf
i}}
\def\bfj{{\bf j}}
\def\bfk{{\bf k}}
\def\bfl{{\bf l}}
\def\bfm{{\bf
m}}
\def\bfn{{\bf n}}
\def\bfo{{\bf o}}
\def\bfp{{\bf p}}
\def\bfq{{\bf
q}}
\def\bfr{{\bf r}}
\def\bfs{{\bf s}}
\def\bft{{\bf t}}
\def\bfu{{\bf
u}}
\def\bfv{{\bf v}}
\def\bfw{{\bf w}}
\def\bfx{{\bf x}}
\def\bfy{{\bf y}}
\def\bfz{{\bf z}}
\def\cA{{\cal A}}
\def\cB{{\cal B}}
\def\cD{{\cal D}}
\def\cC{{\cal C}}
\def\cE{{\cal E}}
\def\cF{{\cal F}}
\def\cG{{\cal G}}
\def\cH{{\cal H}}
\def\fin{{\hfill $\Box$}}
\def\iint{{\int_0^t\!\!\!\int_0^t }}
\def\de{{\delta}}
\def\la{{\lambda}}
\def\si{{\sigma}}
\def\bbE{{\bb E}}
\def\rF{{\cal F}}
\def\De{{\Delta}}
\def\et{{\eta}}
\def\cL{{\cal L}}
\def\Om{{\Omega}}
\def\Al{{\Xi}}
\def\Ga{{\Gamma}}
\def\th{{\theta}}
\def\Exp{{\hbox{Exp}}}

\def\Si{{\Sigma}}
\def\tr{{ \hbox{ Tr} }}
\def\esssup {{ \hbox{ ess\ sup} }}
\def\La{{\Lambda}}
\def\vare{{\varepsilon}}
\def \eref#1{\hbox{(\ref{#1})}}
\def\bart{{\underline  t}}
\def\barf{{\underline f}}
\def\barg{{\underline  g}}
\def\barh{{\underline h}}
\def\th{{\theta}}
\def\Th{{\Theta}}
\def\Om{{\Omega}}
\def\cS{{\cal S}}
\def\cP{{\cal P}}
\def\bfone{{\bf 1}}
\def\var{{\rm Var}}
\def\diag{{\rm Diag}}
\newcommand{\red}{\textcolor{red}}
\def\blue{\color{blue}}

\makeatletter
\renewcommand{\theequation}{%
	\thesection.\arabic{equation}}
\@addtoreset{equation}{section}
\makeatother

\title{ BSDEs  generated by fractional space-time noise and related SPDEs}
\author{Yaozhong Hu$^{1,*}$,\,\,Juan Li$^{2,\dag}$,\,\, Chao Mi$^{2,\ddag}$ \\
 {\small Department of Math  and Stat  Sciences, University of Alberta, Edmonton, AB T6G 2G1, Canada.$^1$}\\
 {\small School of Mathematics and Statistics, Shandong University, Weihai, Weihai 264209, P. R. China.$^2$}\\
 {\small{\it E-mails: yaozhong@ualberta.ca,\,\ juanli@sdu.edu.cn,\,\ michao94@mail.sdu.edu.cn.}}
 \date{\today}
	}
\renewcommand{\thefootnote}{\fnsymbol{footnote}}
\footnotetext[1]{Y.Hu is supported by an NSERC discovery grant and a centennial fund of University of Alberta.}
\footnotetext[2]{J.Li  is supported by the NSF of P.R. China (NOs. 12031009, 11871037), National Key R and D Program
of China (NO. 2018YFA0703900), and NSFC-RS (No. 11661130148; NA150344).}
\footnotetext[3]{Corresponding authors. C.Mi is supported by China Scholarship Council. }
\maketitle

\textbf{Abstract}. This paper  is concerned with    the backward stochastic differential equations  whose generator is a weighted fractional Brownian field:
$Y_t=\xi+\int_t^T Y_s    W  (ds,B_s)  -\int_t^T Z_sdB_s$, $0\le t\le T$,     where $W$ is a $(d+1)$-parameter weighted fractional Brownian field of Hurst parameter  $H=(H_0, H_1, \cdots, H_d)$, which   provide probabilistic interpretations
(Feynman-Kac formulas) for certain linear stochastic partial differential
equations with colored space-time noise. Conditions on the Hurst parameter  $H$ and on the decay rate of the weight are given to ensure the existence and uniqueness of the solution
pair. Moreover, the explicit expression for both components $Y$ and $Z$ of the solution pair  are given.


\section{Introduction  and main result}
Let
$\RR^d$  be the $d$-dimensional Euclidean space.
Let $W=(W(t,x)\,, t\ge 0\,, x\in \RR^d)$ be a weighted
fractional Brownian field. Namely, $W$ is a mean-zero Gaussian random field with the following
covariance structure:
\begin{equation}
\EE \left[   W(t,x)   W(s,y)\right]
= R_{H_0}(s,t)  \rho(x)\rho(y)   \prod_{i=1}^dR_{H_i}(x_i, y_i)   \,,\label{e.1.1}
\end{equation}
where and throughout the paper, we assume $H_i\in (1/2, 1)$ for all $i=0, 1, \cdots, d$, and $R_H(\xi,\eta)=
\left[ |\xi|^{2H }+|\eta |^{2H }-|\xi-\eta|^{2H }\right]/2$, for all $\xi, \eta\in \RR$ and  $\rho(x)$ is a continuous function from $\RR^d$ to $\RR$ satisfying some properties which will be specified later.
We  consider the following (one dimensional)  linear backward  stochastic differential equation
(BSDE for short) with fractional noise generator:
\begin{equation}\label{e1}
  \begin{split}
  Y_t=\xi+\int_t^T Y_s    W  (ds,B_s)  -\int_t^T Z_sdB_s,\quad  t\in [0,T]\,,
  \end{split}
\end{equation}
where
%$\tilde{\dot{W}}(t,x) = \dot{W}(t,x)\rho(x)=\frac{\partial W}{\partial t}(t,x)(1+|x|)^{-\beta}$, $W(t,x)$ %\frac{\partial^{d+1} W^H}{\partial t \partial x_1,...,\partial x_d}$, $W^H$
% is fractional Brownian sheet with Hurst indexes $H_i\in(1/2,1),\ i=0,1,\ldots ,d$ and
$B$ is a $d$-dimensional
standard Brownian motion.  Our interest in this equation is motivated  from the following three aspects.
\begin{enumerate}
\item[(a)] The  first aspect  is the nonlinear Feynman-Kac formula (in our special case) which   relates the following two stochastic differential equations: the first  one is  the backward doubly stochastic differential equation (BDSDE for short)
\begin{eqnarray*}
Y^{t,x}_s
&=& \phi(X^{t,x}_T)+\int_s^T f(r, X^{t,x}_r, Y^{t,x}_r, Z^{t,x}_r)dr\\
&&\qquad\quad  +\int_s^T g(r, X^{t,x}_r, Y^{t,x}_r, Z^{t,x}_r)W(dr, X^{t,x}_r)-\int_s^t
Z^{t,x}_r dB_r\,,
\end{eqnarray*}
where $X_s^{t,x}$ is the solution to the following stochastic differential equation
\[
dX_s^{t,x}=b(X_s^{t,x}) ds+\si(X_s^{t,x}) dB_s\,,\quad s\in [t, T]\,, \quad X_t^{t,x} =x\in \RR^d\,.
\]
The second one is the stochastic partial differential equation (SPDE for short)
\begin{eqnarray}\label{spde}
\begin{cases}-d u(t,x)
 =  \left[ \cL u(t,x) +f(t, x, u(t,x), \nabla u(t,x) \si(x) )\right] dt \\
  \qquad\quad  +g (t, x, u(t,x), \nabla u(t,x) \si(x) )W(dt, x)\,,
\quad (t,x)\in [0, T]\times \RR^d\,, \\
u(T, x) =  \phi(x)\,,
\end{cases}
\end{eqnarray}
where $\cL$ is the generator associated with the Markov process $X_s^{t,x}$.
There are many articles  along this direction since the work of
\cite{PP1994}. Most scholars studied the BDSDEs under various conditions, whose solution can be used as the nonlinear Feynman--Kac formula to represent the solution to  the correlated semi-linear SPDEs driven by white noise.
We refer   to   \cite[Theorem 5.1]{SSZ2019}  and the references therein for the exact relation between the solutions of these two equations. It is worth noting that, BDSDEs and probabilistic interpretation (nonlinear Feynman-KAC formula) of SPDEs driven only by temporal white noise have been studied extensively in several directions, see e.g.  \cite{BM2001}, \cite{BM2001-1}, \cite{BM2001-2}  and \cite{JL2011}.  Although  Feynman-Kac formulas of (linear or non-linear) SPDE with spatial-temporal noise is obtained in \cite{HNS2011}, \cite{HHLNT2017}and \cite{HNS2013} for example, there are limited works to characterize the solution of SPDEs by using the solution of BSDEs. To the best of our knowledge  only    \cite{SSZ2019} and \cite{J2012} dealt with such problems.

%Let us emphasize that the noise $W$ is usually assumed to be independent  of $x$ and/or $W$ is Brownian motion with respect to time  (see e.g. \cite{J2012}).
\item[(b)] If  $b=0$ and $\si=1$, then  $\cL=\frac12 \Delta=\frac12\sum_{i=1}^d \frac{\partial ^2}{\partial x_i^2} $
is the half of the Laplacian. If further $g(r,x, u, p)=u$,  then
the  above SPDE \eqref{spde} becomes
\begin{equation}
-d u(t,x)=\frac12 \Delta u dt+u(t,x) W(dt, x),\ u(T, x) =  \phi(x)\,.\label{anderson}
\end{equation}
This equation has  enjoyed  a great attention  in recent decade (when the terminal condition is replaced by the initial condition and the noise $W(dt,x)$ is replaced by   more singular one
$\frac{\partial ^d}{\partial x_1\cdots \partial x_d}W(dt,x)$),
 often in the name of parabolic Anderson model.  We refer to a survey work \cite{H2019} and references therein for further study. Let us only point out that  many works  do not require that the noise is white in time in their study: For the SPDE in   the above case (b), the associated BDSDEs becomes
\begin{equation}\label{BSDE_Mar}
Y_s^{t,x}= \phi(B_T^{t,x})+\int_s^T
Y_r^{t,x} W(dr, B_{r-t}^{t,x})  -\int_s^t
Z^{t,x}_r dB_r\,,
\end{equation}
where $B_s^{t,x} = x+ (B_s-B_t) $ is a $d$-dimensional Brownian motion starting at time $t$ from the point $x$. This  equation is of the form \eqref{e1}. Its   probabilistic interpretation, the explicit form  and some sharp properties
of solution will be the main focus  of  this paper.
\end{enumerate}

To  illustrate our main results of finding  the explicit representation of the solution pair using     partial Malliavin  derivatives we shall follow the idea of \cite{HNS2011}.
Define (we shall justify it in the next section)
\begin{equation}
\al_s^t=\exp\left[\int_s^t {W} (dr,B_r)\right]
\label{e.1.5}
\end{equation}
and denote  by
%$\cF_t^B=\si(B_s, 0\le s\le t)$ the $\si$-algebra generated by
% Brownian motion up to time instant $t$  and by
 $\cF_t =\si(B_s, 0\le s\le t\,; W(t,x), t\ge 0, x\in \RR^d )$ the $\si$-algebra generated by
 Brownian motion up to time instant $t\ge 0$ and $W(t,x)$ for all $t$ and $x\in \RR^d$.
Then   we  have formally   the  following  candidate for the solution pair
\begin{equation}
\begin{cases}
Y_t=(\al_0^t)^{-1}\EE  \left[\xi \al_0^T  \,\big|\,\mathcal{F}_t\right]
=\EE  \left[\xi   \exp\big(\int_t^T   {W}(dr,B_r)\big)\,\big|\,\mathcal{F}_t \right]
 &  \hbox{  by \cite[Equation (2.11)]{hu2011malliavin} },\\ \\
Z_t=D_t^BY_t=D_t^B\EE \left[\xi   \exp\big(\int_t^T  {W}(dr,B_r)\big)\,\big|\,\mathcal{F}_t \right]
&  \hbox{  by \cite[Equation (2.23)]{hu2011malliavin} }\,,
\end{cases}\label{e.1.6}
\end{equation}
where $D_t^B$ is the Malliavin gradient with respect to the Brownian motion $B$ (see next section for the  definition and properties), and $\EE^B$ is the expectation with respect to $B$ (explained in detail in the proof of Proposition \ref{3.1}).
Here is the main result of this paper.
\begin{thm}\label{th1}
  Suppose $\sum_{i=1}^d (2H_i - \beta_i)<2$ and $\xi\in \mathbb{D}^{1,q}_B$ is measurable w.r.t. $\sigma$-field $\cF_T^B$, for $\displaystyle q>\frac{2}{2\underline{H} -1}$, where $\underline{H}=\min\{H_0,\ldots,H_d\}$.
%  and $\xi$ is Malliavin differentiable with
%respect to $B$
Then we have the following results:
\begin{enumerate}
\item[(1)]   The processes $\{(Y_t, Z_t), 0\le t\le T\}$ formally
 defined  by \eqref{e.1.6}  are well-defined and square integrable, and they are the solution pair to  the BSDE \eqref{e1}. Moreover, $Z$ has the following alternative expression:
 \begin{equation}
 Z_t=\EE \bigg[   e^{\int_t^T  {W} (d\tau, B_\tau) } D_t^B\xi
 +\int_t^T e^{\int_t^s { {W}} (d\tau, B_\tau)
      }  Y_s  (\nabla_x  W) (ds, B_s)
 \Big|\cF_t\bigg] \, .
 \end{equation}
\item[(2)] If for all  $\displaystyle q >2$, $\EE |D_t^B\xi-D_s^B\xi|^q\le
 C|t-s|^{\kappa q/2}$ for some $\kappa \in (0, 2) $, then
 for any $a>1$ and for any $\vare>0$,  we have
 the following H\"older continuity for $Y$ and $Z$:
 \begin{equation}
 \EE  |Y_t-Y_s|^a  \leq C_a \,|t-s|^{a/2},\ \quad  \EE  \big| Z_t  - Z_s \big|^2  \leq C_\vare |t-s|^{
 (2H_0+ \underline{H} -1 -\vare)\wedge \kappa      }\,, \ \
 \forall \  s, t\in[0,T].
 \end{equation}
\item[(3)] If a  pair $(Y, Z)$ satisfies   (2) for some $a, \kappa >0$,  then $(Y,Z)$ is  represented by \eqref{e.1.6} and hence the BSDEs \eqref{e1} has a unique solution.
\item[(4)] If $(Y,Z)\in\mathcal{S}^2_{\mathbb{F} }(0,T;\RR) \times\mathcal{M} ^2_{\mathbb{F} }(0,T;\RR^d)$ is  the solution pair of BSDEs \eqref{e1} so that  $Y, D^B Y$ are  $ \mathbb{D}^{1,2}$ then the solution  also has the  explicit expression \eqref{e.1.6} and hence the BSDEs \eqref{e1} has a unique solution.
\end{enumerate}
 \end{thm}

\begin{remark} Since we assume $H_0>1/2$,
we see $2H_0+ \underline{H} -1>0$.
We
%can obtain the H\"older continuity in the mean sense of $Y$ for any $a\in(1,q)$, which implies the almost sure H\"older continuity for $Y$. But we
can only obtain the H\"{o}lder continuity of $Z$ in the mean square sense. We encounter the difficulty to deal with high moments for $Z$.
\end{remark}

%In this paper we consider a linear SPDE generated by fractional space-time noise and give a probability interpretation (Feynman-Kac formula) through studying the solution of BSDE \eqref{BSDE_Mar}.
Now let us  point out the novelty compared to two relevant works.      In the  work \cite{J2012},  the generator
$W$ is a fractional Brownian motion (the generator
$W$ does not depend on $x$).
%(it is not a random field) and under some smooth conditions, they use the Doss-Sussmann transformation to transform the associated BDSDEs to BSDE without stochastic integral with respect to fBm.
In the work    \cite{SSZ2019}  $W$ can depend on the space $x$,   but
it is assumed that    that it  is   a backward martingale with to time variable $t$ so that the
backward martingale technology   can be used. In our above theorem,
neither the assumption   that $W$ is independent of $x$,
nor the assumption   that $W$ is  a backward martingale is assumed.
In particular, we can obtain the explicit solution (for linear equation)
and use this expression to obtain the some kind sharp  H\"older continuity
for the solution pair which, to our best knowledge, are new.

Here is the organization of this work.  In next section, we shall
show that the quantity $\int_s^t {W} (dr,B_r)$ in \eqref{e.1.6} is
well-defined and is exponentially integrable so that $Y_t$ is well-defined.
In Section  \ref{s.3}, we obtain some properties of the process $Y_t$ and show that it is Malliavin differentiable and $Z_t$ is well-defined.
We show that the pair  $(Y_t, Z_t)$ is the solution to the linear BSDE
\eqref{e1}.  A great difficulty
 is that we need to show that the process $Y$
is in $\mathcal{S}_{\mathbb{F}}^{p}(0, T ; \mathbb{R})$ and $ Z$ is in
 $\mathcal{M}_{\mathbb{F}}^{p}(0, T ; \mathbb{R})$ due to the singularity of the noise $W$ in the generator.  We overcome this difficulty by
 Talagrand theorem \ref{majorizing}, Borell theorem \ref{Borell},   and a new Lemma \ref{x1x2y}.     In Section \ref{s.4}, we use the explicit expression to obtain   H\"older continuity of the solution pair. The H\"older continuity of the process $Z_t$ is always a difficult problem (see e.g. \cite{hu2011malliavin,
zhang2004, mazhang2002}) however plays a critical role in numerical
method. In Section \ref{s.5} we discuss the relation between the linear BSDE \eqref{BSDE_Mar}  and the stochastic PDE
\eqref{anderson}.

%So our work can be considered the extension of
%\cite{J2012} to infinite dimensional
%(random field), and of that of \cite{SSZ2019} from white to colored noise in time, so that to accommodate
%the study of the parabolic Anderson model.
%\footnote{Find more references about backward doubly
%sde with  generator  determined by fractional noise. Jing Shuai,  Nie Tieyang, ... }

%Namely, (\ref{e1}) has the solution
%\begin{equation}\label{so namely}
  %\begin{split}
  %  Y_t = E^B\Big[\xi \exp\big(\int_t^T \tilde{\dot{W}}(s,B_s) ds\big)\,\big|\,\mathcal{F}_t\Big].
 % \end{split}
%\end{equation}
%In fact, $\dot{W}^H(t,x)$, (\ref{e1}) and solution (\ref{so namely}) don't make sense, but \\

\section{Exponential integrability of
$\int_t^T {{W}}(ds,B_s)$ }
Let $T>0$ be a fixed time horizon and let $(\Omega, \mathcal{F}, P)$ be a   complete probability space, on which the expectation is denoted by $\EE$.
Let $\left\{B_{t}, 0 \leq t \leq T\right\}$
 be a $d$-dimensional standard Brownian motion defined on $(\Omega, \mathcal{F}, P)$.
Suppose $   W=\{  W(t,x),t\geq 0,\ x\in\RR^d\}$ is a weighted fractional   Brownian space-time  field
whose covariance is given by \eqref{e.1.1}.
The stochastic integral with respect to $W$ is
well-defined in many references,   and we refer to \cite{H2019} and references therein for more details.
We shall use this concept  freely.
For example, we denote  $W(\phi)=\int_{\RR_+\times \RR^d} \phi(t, x) W(dt, x)dx$ for any $\phi\in \cD=\cD( \RR_+\times \RR^d,\RR)$, where $\cD$ is the set of all smooth functions with compact support from  $\RR_+\times \RR^d$ to $\RR$.
We denote the spatial covariance as
\begin{equation}
q(x,y)=\rho(x)\rho(y)\prod_{i=1}^d R_{H_i}(x_i,y_i)\,,\quad \forall \ x=(x_1, \cdots, x_d)^T\,,\quad
y=(y _1, \cdots, y_d)^T\in\RR^d\,,
\label{e.1.7}
\end{equation}
where $\rho:\RR^d\to \RR $ is  a  continuous    function of power decay, and we will specify the conditions that $\rho$ are satisfied  later.   It is known that
\begin{equation}\label{inprod func true}
  \begin{split}
    \EE\big[ {W}(h) {W}(g) \big ]
%    &= \big\langle h,g \big\rangle _\cH =    \int_{\RR_+^2 \times \RR^{2d} } h(t,x)g(s,y)|t-s|^{2H_0 -2} q(x,y) dsdtdxdy  \\
    &=   \int_{\RR_+^2 \times \RR^{2d} } h(t,x)g(s,y)|s-t|^{2H_0-2}
    \rho( x)\rho(y )\prod_{i=1}^d R_{H_i}(x_i,y_i) dsdtdxdy
  \end{split}
\end{equation}
for all $ h, g\in \cD$. It is clear that $\big\langle h,g \big\rangle _\cH$ is a scalar product on $\cD$. We denote $\cH$ the Hilbert space by completing $\cD$ with respect to  this  scalar product.
  %where for all $x\in \RR$, $\rho(x)\in \mathcal{C}^1$ is a well-defined function and we will introduce it in the next section.

%Denote by $D$ the derivative operator in the sense of Malliavin calculus.
Let  $F$ be  a  cylindrical random variable of the form
$$F=f\left(W\left(\phi^{1}\right), \ldots, W\left(\phi^{n}\right)\right),$$
where $\phi^i \in \cD\,, i=1, \cdots, n$  and  $f\in C^{\infty}_p (\RR^n)$, i.e., $f$ and all its partial derivatives have polynomial growth.   The set of all such cylindrical random variables is denoted by
$\cP$.   If $F\in \cP$ has the above form, then  $D^WF$ is the $\mathcal{H}$-valued
random variable defined by
$$D^W F=\sum_{j=1}^{n} \frac{\partial f}{\partial x_{j}}\left( {W}\left(\phi_{1}\right), \ldots,  {W}\left(\phi_{n}\right)\right) \phi_{j}.$$
The operator $D^W$ is closable from $L^{2}(\Omega)$ into $L^{2}(\Omega, \mathcal{H})$, namely $D^W$ is the Malliavin derivative operator with respect to the fractional Brownian motion $W$. We define the Sobolev space $\mathbb{D}^{1,p}_W$ as the closure of $\cP$  under the following norm :
$$\|D^W F\|_{1,p}=\left(\EE\left|F\right|^{p}+\EE\|D^W F\|_{\mathcal{H}}^{p}\right)^{1/p}.$$
Let us denote by $\delta$ the adjoint of the derivative operator given by duality formula
$$\EE(\delta(u) F)=\EE\left(\langle D^W F, u\rangle_{\mathcal{H}}\right) \quad \text { for any } F \in \mathbb{D}^{1,2}_W,$$
where $\delta(u)$ is also called the Skorohod integral of $u$. We refer to   \cite{H2016} and \cite{N2006} for a detailed account on the Malliavin calculus. For any random
variable $F \in \mathbb{D}^{1,2}_W$ and $\phi \in \mathcal{H}$, we will often use the following formula in the text:
$$F  {W}(\phi)=\delta(F\phi)+\langle D^W F, \phi\rangle_{\mathcal{H}}.$$
Accordingly, we can define $D^B$ the Malliavin derivative operator with respect to the standard Brownian motion $B$ and $\mathbb{D}^{1,p}_B$ the Sobolev space in the same way. We say a random field $F\in\mathbb{D}^{1,p}$ if $F$ is an element both in $\mathbb{D}^{1,p}_B$ and $\mathbb{D}^{1,p}_W$.

%%if we define a random variable $G$ of the form
%%$$G=f\left(B\left(\xi^{1}\right), \ldots, B\left(\xi^{n}\right)\right),$$
%%where $B\left(\xi^{i}\right) = \int_{\RR_+} \xi_s^i\, dB_s$ and $\xi^i:\RR_+\rightarrow \RR$ are smooth functions, $i=1,\ldots,n$. We denote by $D^B$ the Malliavin derivative operator with respect to the standard Brownian motion $B$, and by $\mathbb{D}^{1,p}_B$ the Sobolev space under the following norm :
%%$$\|D^B G\|_{1,p}=\left(\EE\left|G\right|^{p}+\EE\|D^B G\|_{L_p}^{p}\right)^{1/p}.$$
%
%\section{ The approximation and the exponential integrability of stochastic integral}
The stochastic integral studied earlier is
useful in this paper but is  not sufficient for our purpose.
We also need to introduce a new kind of nonlinear stochastic integral
similar to that of Kunita (\cite{kunita1997}).
%First, we would like to  use approximation to study the existence of \eqref{e.1.5}.
 To this end,  we
introduce the  approximation of ${\dot{W}}$
as follows.
%, denoted by ${\dot{W}}_{\varepsilon,\eta}$:
%\footnote{use a simple notation $ \dot{W} $  and  $ \dot{W} _{\varepsilon,\eta}$
%for $\tilde{\dot{W}}$ and  $ \tilde {\dot{W}} _{\varepsilon,\eta}$.
%\newline Replace  $\dot W(s,x)ds$ by
%$W(ds, x)$}.
\begin{equation}\label{nsp}
\begin{split}
  {\dot{W}}_{\varepsilon,\eta}(s,B_s) =& \int_0^s \int_{R^d} \varphi_\eta (s-r) p_\varepsilon (B_s -y) W(dr,y)dy\,,
%  \\
%  =&\int_0^s \int_{R^d} \varphi_\eta (s-r) p_\varepsilon (B_s -y) \rho(y) \widetilde{W}(dr,y)dy ,
\end{split}
\end{equation}
where $\varphi_\eta$ and $p_\varepsilon$ are the approximation of the Dirac delta functions:
$$\varphi_\eta(t) = \frac{1}{\eta}\mathbf{1}_{[0,\eta]} (t),\ p_\varepsilon(x) = (2\pi\epsilon)^{-d/2} e^{-|x|^2/2\varepsilon},  \quad \mbox{for all}\ \eta,\ \varepsilon>0.$$

\begin{prop}\label{3.1}
Let $\rho:\RR^d\to \RR $ be  a  continuous  function of power decay, i.e., $\rho$ satisfying $ 0\le \rho(x )\le C\prod_{i=1}^d (1+|x_i|)^{-\beta_i}$, where $\beta_i \in (0,2) $ and $2H_i>\beta_i$ for all $i=1,2,\ldots d$ and suppose
\begin{equation}\label{e.2.7}
\al:=\sum_{i=1}^d (2H_i - \beta_i)<2\,.
% \quad \mbox{ and}\quad 2H_0+\sum_i^d( H_i-\beta_i/2) >1\,.
\end{equation}
%satisfy $ 0\le \rho(x )\le C\prod_{i=1}^d (1+|x_i|)^{-\beta_i}$, where $\beta_i \in (0,2) $ for all $i=1,2,\ldots d$.
%Suppose
%\[
%\sum_{i=1}^d (2H_i - \beta_i)\in (-1,2) \quad \mbox{ and}\quad 2H_0+\sum_i^d( H_i-\beta_i/2) >1\,.
%\]
 Then
% for any $\varepsilon >0$ and $\eta >0$,
 the stochastic integral $V_t^{\ep,\eta}:=\int_t^T
  {\dot{W}}_{\varepsilon,\eta}(s,B_s)ds$ converges in $L^2(\Omega)$ to a limit denoted by
\begin{eqnarray}\label{vt}
  V_t = \int_t^T {{W}}(ds,B_s).
\end{eqnarray}
Moreover, conditioning  on $\cF^B$, $V_t$ is a mean-zero Gaussian random variable with variance
\begin{eqnarray}\label{vtv}
  \var^W (V_t) =   \int_t^T\int_t^T |s-s'|^{2H_0-2} \rho(B_s )\rho(B_{s'} )\prod^d_{i=1} R_{H_i}(B_s^i,B_{s'}^i) dsds'\,.
\end{eqnarray}
\end{prop}

\begin{proof}
Suppose $\ep,\ep',\eta,\eta'\in(0,1)$. % are fixed.
%Let us first consider the square integrability of $\int_t^T \tilde{\dot{W}}^H_{\varepsilon,\eta}(s,B_s)ds$.
Throughout the paper denote   by $\EE^W$  the expectation with respect to the random field $W$  which considers other random elements as    fixed ``constant".  For example, if $ F(W, B)$ is a functional of $W$ and $B$, then $\EE^W(F(W, B) )=\EE(F(W, B)|\cF^B)$.
By Fubini's theorem and \eqref{nsp}, we have
\begin{align}
    &\EE^W  \Bigg[\int_t^T{\dot{W}}_{\varepsilon,\eta}(s,B_s)ds \int_t^T {\dot{W}}_{\varepsilon',\eta'}(s,B_s)ds\Bigg]\nonumber\\
%    =&\EE^W [\int_t^T\int_t^T \int_t^s \int_t^{s'}\int_{\RR^{2d}} \varphi_\eta (s-r) p_\varepsilon (B_s -y)\varphi_{\eta'} (s'-r') p_{\varepsilon'} (B_{s'} -y')\nonumber \\
%    &\qquad\qquad\qquad\qquad\qquad\qquad\qquad\qquad W(dr,y)W(dr',y')\rho(y)\rho(y')dydy'drdr'dsds']\\
    =\,& \alpha_{H_0}  \int_t^T\int_t^T\int_t^s \int_t^{s'}\int_{\RR^{2d}}\varphi_\eta (s-r)\varphi_{\eta'} (s'-r')p_\varepsilon (B_s -y) p_{\varepsilon'} (B_{s'} -y')\nonumber \\
    &\qquad\qquad\qquad\qquad\qquad |r-r'|^{2H_0-2}\rho(y )\rho(y' ) \prod^d_{i=1} R_{H_i}(y_i,y'_i) dydy'drdr'dsds'\nonumber  \\
    =\,& \alpha_{H_0}\int_t^T\int_t^T\int_t^s \int_t^{s'} \varphi_\eta (s-r)\varphi_{\eta'} (s'-r')|r-r'|^{2H_0-2}
    \\
%    &\qquad \prod^d_{i=1} \Big[\EE\big[\big(|\sqrt\varepsilon X_i + B_s^i|^{2H_i}  + |\sqrt\varepsilon' X'_i + B_{s'}^i|^{2H_i} - |\sqrt\varepsilon' X'_i + B_{s'}^i
%        -\sqrt\varepsilon X_i - B_{s}^i|^{2H_i}%%%\big)\\
        &\qquad
     \EE^{X, X'}  \left\{ \rho(\sqrt\varepsilon X  + B_s  )\rho(\sqrt{\varepsilon'} X'  +B_{s'} ) \prod^d_{i=1} \Big[  R_{H_i}(
     \sqrt\varepsilon X_i + B_s^i, \sqrt\varepsilon X'_i + B_{s'}^i)  \Big] \right\} %\Big|%_{x=B_s^i,x'=B_{s'}^i}
        drdr'dsds'\nonumber\\
=:&I(\varepsilon, \varepsilon', \eta, \eta')\,,\nonumber
\end{align}
where $X=(X_1, \cdots, X_d),X'=(X_1', \cdots, X_d')$ are  independent  standard  random variables,
which are also independent of $\cF^B$.

% Now we have to obtain the integrability of the last term of \eqref{eq.1}.

%Firstly, let us deal with $\prod^d_{i=1} R_{H_i}(y,y')\rho(y)\rho(y')$. We note that for $d\geq 2,\ H_i\in(1/2,1),$  $0<\beta_i<2$, we have
%the condition $-1< 2H_i -  \beta_i<2$ be satisfied for all $i=1,2,...,d$. Thus, we have
To study the  limit of  above $I(\varepsilon, \varepsilon', \eta, \eta')$  as $\varepsilon, \varepsilon', \eta, \eta'\to 0$, we observe that, firstly \cite[Lemma A.3]{HNS2011} directly yields
\begin{equation}\label{estphieta}
\int_t^s \int_t^{s'}\varphi_\eta (s-r)\varphi_{\eta'} (s'-r')|r-r'|^{2H_0-2}drdr' \leq |s-s'|^{2H_0-2}.
\end{equation}
Moreover,
  \begin{align}
  \left| q(y, y') \right|&
%  = \left| \rho(y )\rho(y' ) \prod^d_{i=1} R_{H_i}(y_i,y'_i)\ \right|
   = \frac{1}{2}\rho(y )\rho(y' ) \prod^d_{i=1} \left| |y_i|^{2H_i} + |y'_i|^{2H_i}-|y_i-y_i'|^{2H_i}\right|
  \nonumber \\
    &\leq C \rho(y )\rho(y' )\prod^d_{i=1}(|y_i|^{2H_i} + |y'_i| ^{2H_i}) \nonumber\\
&\leq C \prod^d_{i=1} (1+|y_i| ^{2H_i})
     (1+|y'_i|^{2H_i})(1+|y_i|)^{-\beta_i}
     (1+|y_i'|)^{-\beta_i}\label{R split}\\
     &\leq C \prod^d_{i=1} (1+|y_i|)^{2H_i-\beta_i}
     (1+|y'_i|)^{2H_i-\beta_i} \,,\nonumber
  \end{align}
where and throughout this paper  $C$  is a generic constant
  depending  only on $H_i,\ i=1,\ldots,d.$\\
This can be used to show that
\begin{align}\label{Eep}
I_1(\varepsilon, \varepsilon', s, s')
:=\EE^{X, X'}  \Big[ \rho(\sqrt\varepsilon X  + B_s  )\rho(\sqrt{\varepsilon'} X' +B_{s'} )\prod^d_{i=1} R_{H_i}(
     \sqrt\varepsilon X_i + B_s^i, \sqrt\varepsilon X'_i + B_{s'}^i)\big]
\end{align}
is a pathwise bounded continuous function of
$\varepsilon, \varepsilon', s, s'$ in the concerned
domain (almost surely with respect to
$B$).   Thus, we have
%\begin{equation*}
%  \begin{split}
% &\lim_{ \varepsilon, \varepsilon', \eta, \eta'\to 0}
%I(\varepsilon, \varepsilon', \eta, \eta')\\
%    &\qquad\qquad\qquad =\lim_{ \varepsilon, \varepsilon', \eta, \eta'\to 0}\int_t^T\int_t^T\int_t^s \int_t^{s'} \varphi_\eta (s-r)\varphi_{\eta'} (s'-r')|r-r'|^{2H_0-2}
%I_1(\varepsilon, \varepsilon', s, s')
%        drdr'dsds'\\
%         &\qquad\qquad\qquad =  \int_t^T\int_t^T |s-s'|^{2H_0-2} \rho(B_s)\rho(B_{s'})\prod^d_{i=1} R_{H_i}(B_s^i,B_{s'}^i)dsds' \,.
%  \end{split}
%\end{equation*}
\begin{equation}\label{EIinfyty}
  \begin{split}
\EE [I&(\varepsilon, \varepsilon', \eta, \eta')]\\
   & =\alpha_{H_0}\EE\int_t^T\int_t^T\int_t^s \int_t^{s'} \varphi_\eta (s-r)\varphi_{\eta'} (s'-r')|r-r'|^{2H_0-2}
I_1(\varepsilon, \varepsilon', s, s')
        drdr'dsds'\\
         &\leq  \alpha_{H_0}\int_t^T\int_t^T |s-s'|^{2H_0-2} \prod^d_{i=1}\EE\, (1+|B_s^i|)^{2H_i-\beta_i}(1+|B_{s'}^i|)^{2H_i-\beta_i}dsds'\\
         &\leq C |T-t|^{2H_0} <\infty\,.
  \end{split}
\end{equation}
Moreover, for $s\neq s'$, as $\ep,\ep',\eta,\eta'$ tend to zero we have
\begin{equation*}
  \begin{split}
 &\lim_{ \varepsilon, \varepsilon', \eta, \eta'\to 0}
I(\varepsilon, \varepsilon', \eta, \eta')\\
    &\qquad =\alpha_{H_0}\lim_{ \varepsilon, \varepsilon', \eta, \eta'\to 0}\int_t^T\int_t^T\int_t^s \int_t^{s'} \varphi_\eta (s-r)\varphi_{\eta'} (s'-r')|r-r'|^{2H_0-2}
I_1(\varepsilon, \varepsilon', s, s')
        drdr'dsds'\\
         &\qquad = \alpha_{H_0} \int_t^T\int_t^T |s-s'|^{2H_0-2} \rho(B_s)\rho(B_{s'})\prod^d_{i=1} R_{H_i}(B_s^i,B_{s'}^i)dsds' \,.
  \end{split}
\end{equation*}
Therefore, if we put $\ep=\ep',\ \eta=\eta'$ and use the estimates \eqref{R split} and \eqref{Eep}, and with the help of Lebesgue's convergence theorem we have
$$
\EE\left(V_{t}^{\varepsilon, \eta}-V_{t}^{\varepsilon^{\prime}, \eta^{\prime}}\right)^{2}=\EE\left(V_{t}^{\varepsilon, \eta}\right)^{2}-2 \EE\left(V_{t}^{\varepsilon,\eta} V_{t}^{\varepsilon^{\prime}, \eta^{\prime}}\right)+\EE\left(V_{t}^{\varepsilon^{\prime}, \eta^{\prime}}\right)^{2} \rightarrow 0,\ \mbox{as}\ \ep,\ep',\eta,\eta'\rightarrow 0.
$$
As a consequence we have $V_{t}^{\varepsilon_n, \eta_n}$ is a Cauchy sequence in $L^2(\Om)$. It has then a limit denoted by $V_t$,
proving the proposition.
\end{proof}
%
%We want to prove the integrability of the following:
%\begin{lem}\label{p order int}
%   Assume $p \geq 2$. Then we have
%\begin{align}\label{ordin inequ}
%  \int_{[t,T]^2} |s-r|^{2H_0-2} \prod^d_{i=1}R_{H_i}(B^i_s,B^i_r)\rho(B^i_s)\rho(B^i_r) dsdr
%\end{align}
%   is $L^p(\Omega)$-integrable.
%\end{lem}
%\begin{pf}
%Denote
%\[
%\rho_H(s,t)  =\int_s^t\int_s^t |r-r'|^{2H-2} drdr'=\frac1{2H(2H-1)}  |t-s|^{2H}\,.
%\]
%Then first by the Jessen's inequality and then by  the bound \eqref{R split}, we see that for any $p\ge 2$
%\begin{eqnarray*}
%\EE^B\left|  \var^W (V_t)
%\right|^p
%&=&  \rho_H^p (t, T)\EE \left|
%\frac{1}{\rho_H(t, T)}
%\int_t^T\int_t^T |s-s'|^{2H_0-2}    \prod^d_{i=1} R_{H_i}(B_s^i,B_{s'}^i)\rho(B_s^i)\rho(B_{s'}^i)dsds'\right|^p \\
%&\le &  \rho_H^p (t, T)
%\frac{1}{\rho_H(t, T)}
%\int_t^T\int_t^T |s-s'|^{2H_0-2} \EE \left|\prod^d_{i=1} R_{H_i}(B_s^i,B_{s'}^i)\rho(B_s^i)\rho(B_{s'}^i)\right|^p dsds' \\
%&=&  C_{p, t, T}
%\int_t^T\int_t^T |s-s'|^{2H_0-2} \EE \left |
%\prod^d_{i=1} \big(1+|B_s^i|\big)^{p(2H_i-\beta)}
%     \big(1+|B_{s'} ^i|\big)^{p(2H_i-\beta)}
%\right|  dsds'
%\end{eqnarray*}
% which is easy to get that this is finite.
%\end{pf}

% Lemma \ref{p order int} can be used to show that
% $\int_t^T {W}(dr,B_r)$ has   finite moments
%  of
%  any order as a random variable (functional) of
%  $B, W$.
%In fact, we can show the exponential integrability of $\int_t^T {W}(dr,B_r)$
%and we also need this in the future.
%Now we prove the integrability of $ V_t$ defined in \eqref{vt}.
\begin{prop}\label{exp int V}
 Let $\rho:\RR^d\to \RR $ be  a  continuous    function satisfying \eqref{e.2.7}.
%  Suppose $\alpha=\sum_{i=1}^d   2 H_i -    \beta_i  \in (-1,2)$ and $2H_0+\sum_i^d H_i-\beta_i/2>1$.
Then, for all $\lambda\in\RR$,
$$\displaystyle \EE \left[\exp\big(\lambda\int_t^T  {W}(dr,B_r)\big)\right]<\infty.$$
\end{prop}
\begin{proof}
%Note that $\int_t^T{W}(dr,x)=\int_t^T \rho(x) \tilde{W}(dr,x) $.
% is a mean-zero Gaussian stochastic field for all $x\in \RR^d$ respect to fractional Brownian motion $W$. Thus,
From \eqref{vtv}  and  the first inequality in   \eqref{R split}, it follows
  \begin{align}\label{3.4star}
  I:&= \notag\EE \left[\EE^W\exp\big(\lambda\int_t^T {W}(dr,B_r)\big)\right]\\
    &= \EE \left[\exp\Big((\alpha_{H_0}\lambda^2) /2 \int_t^T \int_t^T \Big|s-r\Big|^{2H_0-2} \rho( B_s  )\rho( B_r  )\prod_{i=1}^d R_{H_i}(B_s^i,B_r^i)dsdr\Big)\right]\\
    &\notag \leq  \EE \left[\exp\Big(C(\alpha_{H_0}\lambda^2) /2\int_t^T \int_t^T \Big|s-r\Big|^{2H_0-2} \rho( B_s  )
    \rho( B_r  )\prod_{i=1}^d \Big(\big|B_s^i\big|^{2H_i}+\big|B_r^i\big|^{2H_i}\Big) dsdr\Big)\right]\,.
  \end{align}
Note that
  \begin{eqnarray}\label{estrhoR}
   &&\notag \left|\rho( B_s  )\rho( B_r )\prod_{i=1}^d \Big(\big|B_s^i\big|^{2H_i}+\big|B_r^i\big|^{2H_i}\Big) \right|  \le
    2^d\rho( B_s  )\prod_{i=1}^d\sup_{s\in[t,T]}\big|B_s^i\big|^{2H_i} \,\\
    &&   \leq 2^d \prod_{i=1}^d \big(1+ \sup_{s\in[t,T]} |B_s^i |\big)^{ 2H_i-\beta_i} \le 2^d
     \big(1+ \sup_{s\in[t,T]} \sum_{i=1}^d |B_s^i |\big)^{\sum_{i=1}^d 2H_i-\beta_i}\\
    &&\notag  =:  C_d\big(1+ \sup_{s\in[t,T]} \sum_{i=1}^d |B_s^i |\big)^{\alpha}.
  \end{eqnarray}
We have
\begin{align*}
I \leq \EE\left[\exp\left(C \int_t^T \int_t^T \Big|s-r\Big|^{2H_0-2} dsdr \cdot \Big(1+   \sup_{s\in[t,T]} \sum_{i=1}^d |B_s^i |\Big)^{\alpha} \right)  \right]\,,
\end{align*}
which is finite thanks to Fernique's theorem (e.g. \cite[Theorem 4.14]{H2016}) since $\alpha<2$,
%   by choosing $p = \frac{2}{\alpha} $ and $a$ sufficiently small, we see that $I<\infty$,
completing the proof of the proposition.
\end{proof}
 % But we can not use Fernique's Theorem there.
 %   \red \sum_{i=1}^d H_i<1\,.
  %\]
  %\[
  %  \EE(\tilde{\dot {W}}(t,x) \tilde{ \dot{ W}}(s,y))
  %  =|t-s|^{2H_0-2} \prod_{i=1}^d R_{H_i}(x_i, y_i) \rho(x_i)\rho (y_i)\,,
  %\]
  %\[
  %  0<\rho(x)\le (1+|x|)^{-\beta}\,.\,.
%  \]
 % \[
  %  |x|^{2\sum_{i=1}^d H_i} (1+|x|)^{-\beta}
   % \le (1+|x|)^{2\sum_{i=1}^d H_i  -\beta}\,,.
%  \]
%  \[
%    \beta= \sum_{i=1}^d \beta_i
%  \]
%  \[
%    2\sum_{i=1}^d H_i  -\beta<2\,.
%  \]
%  \[
%    \tilde{\dot{ W}}(t,x)=\dot W(t,x)\prod_{i=1}^d  \rho(x_i) \,.
%  \]

\section{Linear backward stochastic differential equation}
\label{s.3}
Now we   consider  the backward
stochastic differential equation \eqref{e1}. In order to study the regularity of $(Y,Z)$,  we approximate
 it by \eqref{nsp} and obtain    the following approximation of   \eqref{e1}:
\begin{equation}\label{e appro1}
  \begin{split}
    Y_t^{\varepsilon,\eta}=\xi+\int_t^T Y_s^{\varepsilon,\eta} {\dot{W}}_{\varepsilon,\eta}(s,B_s)ds -\int_t^T Z_s^{\varepsilon,\eta}dB_s,\quad  t\in [0,T]\,.
  \end{split}
\end{equation}
  Due to the regularity
of the approximated noise $ \dot{W}  _{\varepsilon,\eta}$ and Proposition  \ref{exp int V}, we can explicitly  express  its  solution    as follows  (see e.g.
\cite{HNS2011} and references therein):
%\begin{equation}\label{so appro}
%  \begin{split}
%    Y_t^{\varepsilon,\eta} = E^B\Big[\xi \exp\big(\int_t^T \tilde{\dot{W}}_{\varepsilon,\eta}(s,B_s)ds\big)\,\big|\,\mathcal{F}_t\Big].
%  \end{split}
%\end{equation}
\begin{equation}
\begin{cases}
Y_t^{\varepsilon, \eta}  =\EE \left[\xi   \exp\big(\int_t^T   \dot{W}_{\varepsilon, \eta} (r,B_r)dr\big)\,\big|\,\mathcal{F}_t\right]
 &  \hbox{  by \cite[Equation (2.11)]{hu2011malliavin} }\,,
 \\
Z_t^{\varepsilon, \eta}   =D_t^B
\EE \left[\xi   \exp\big(\int_t^T  \dot {W}_{\varepsilon, \eta} (r,B_r)dr\big)\,\big|\,\mathcal{F}_t\right]
&  \hbox{  by \cite[Equation (2.23)]{hu2011malliavin} }\,,
\end{cases}  \label{so appro}
\end{equation}
where  $
D^B  =(D^{B^1} , \cdots,  D^{B^d})^T$ is the Malliavin gradient
 operator   with respect to the Brownian motion $B $, so that $Z^{\varepsilon, \eta}_t$ is a $d$-dimensional vector.

We have proved
% that  $\displaystyle \int_t^T \dot{W}_{\varepsilon,\eta}(s,B_s)ds \rightarrow \int_t^T {{W}}(ds,B_s)$ in $L^2(\Omega)$ in Proposition \ref{3.1}
   that $\int_t^T {{W}}(ds,B_s)$ is  exponentially  integrable  in Proposition \ref{exp int V}.  Then we can define
\begin{equation}\label{so}
Y_t := \EE\Big[\xi \exp\big(\int_t^T {W}(dr,B_r)%\int_{R^d} \delta(B_r -y)W(drdy)%%\int_t^T \int_{R^d} \delta (B_s -y)  W(dsdy)$
    \big)\,\big|\,\mathcal{F}_t\Big].
\end{equation}
%Moreover, from the defination of $\varphi_\eta(\cdot) = \mathbf{1}_{[0,\eta]}(\cdot) $, we can deduce that
%\begin{equation}\label{exchange}
%  \begin{split}
% \int_t^T {\dot{W}}_{\varepsilon,\eta}(s,B_s)ds &= \int_t^T \int_t^s \int_{R^d} \varphi_\eta (s-r) p_\varepsilon (B_s -y) \rho(y)\tilde W(dr,y)dyds\\
% & = \int_t^T \int_t^T \int_{R^d}\int_{R^d} \varphi_\eta (s-r) \mathsf{1}_{[0,y]} (z) p_\varepsilon (B_s -y)\rho(y)\tilde W (dr,dz)dyds\\
% & = \int_t^T\int_t^T  \int_{R^d}\int_{R^d} \varphi_\eta (s-r) \mathsf{1}_{[0,y]} (z) p_\varepsilon (B_s -y)\rho(y)dyds \tilde W (dr,dz)\,.
%\end{split}
%\end{equation}

\begin{lem}\label{Y converge}
Assume  $\xi\in L^q(\Om)$ for some $q>2$.
  %  and $2H_0+\sum_{i=1}^d H_i-\beta_i/2 >1$.
  Then for any $t\in [0,T]$, we have $Y_t^{\ep,\eta}$ converges to $Y_t$ in $L^p(\Omega)$
  for all $p\in [1, q)$.
\end{lem}
\begin{proof}
%In fact, $q$ have to satisfies $\displaystyle q>\max\big\{\frac{2}{2H_i-1},\,i=0,1,\ldots,d.\big\}\,$, but we don't need to discuss it now.
Denote $V_t^{\varepsilon,\eta}=\int_t^T {\dot{W}}_{\varepsilon,\eta}(s,B_s)ds$.
%%Fix $\ep,\eta\in(0,1)$.
Let $q'p=q$ and $1/p'+1/q'=1$.
From \eqref{so appro}, \eqref{so}, Jensen's  inequality and H\"older's  inequality it follows
\begin{eqnarray}\label{expVt}
  \EE  \left[\Big|Y_t^{\ep,\eta}-Y_t\Big|^p \right] &=& \EE  \bigg| \EE \Big[\xi   \Big(\exp\big(V_t^{\varepsilon,\eta} \big)
  -  \exp\big( V_t  \big)\Big)  \Big|\mathcal{F}_t  \Big]  \bigg|^p \notag\\
&\le & \EE \left[   |\xi |^p \left|  \exp\big(V_t^{\varepsilon,\eta} \big)
  -  \exp\big( V_t  \big)  \right|^p\right]  \\
%  &\leq& \| \xi \Vert_{q}^{1/q'}  \left(\EE  \Big| \EE^B \Big[\exp\big(\int_t^T {\dot{W}}_{\varepsilon,\eta}(s,B_s)ds \big)
%  -  \exp\big(\int_t^T {W}(ds,B_s) \big)  \Big|\mathcal{F}_t  \Big]  \Big|^{pp'}\right)^{1/p'} \\
&\le & \| \xi \Vert_{q}^{1/q'}
\left[ \EE    \Big|\exp\big(V_t^{\varepsilon,\eta} \big)-  \exp\big(V_t\big)  \Big| ^{pp'} \right]^{1/p'}.\notag
\end{eqnarray}
Similar to  Proposition \ref{exp int V} we can prove
%we know $V_t^{\ep,\eta}$ converges to $V_t$ in probability, and thus it is sufficient to prove for any $\lambda\in \RR$,
\[
\sup_{\varepsilon,\eta\in (0, 1] }\EE  \Big|\exp(\lambda V_t^{\varepsilon,\eta})\Big|<\infty,
\quad \forall \ \la\in \RR\,.
\]
%This estimate can be derivate similarly from \eqref{3.4star}:
%\begin{equation*}
%  \begin{split}
%\EE^B \EE^W \Big|\exp\left(\lambda V_t^{\varepsilon,\eta}\right)\Big|& = \EE^B  \Big| \exp\left(\frac{\lambda^2}{2}\, \EE^W |V_t^{\varepsilon,\eta}|^2\right)\Big|\\
%&\leq \EE^B  \Big| \exp\Big(\lambda^2/2\cdot C \int_{[t,T]^2} |u-v|^{2H_0-2} \prod_{i=1}^d \big(|1+B_v^i|^{2H_i-\beta_i} +|1+B_u^i|^{2H_i-\beta_i}\big) dudv \Big)\Big|\\
%&\leq  \EE^B  \Big| \exp\Big(C\lambda^2  |T-t|^{2H_0}/2 \prod_{i=1}^d \sup_{s\in[t,T]}  \big(|1+B_v^i|^{2H_i-\beta_i}\big)\Big)  \Big|.
%\end{split}
%\end{equation*}
Proposition \ref{3.1}  implies that
$V_t^{\varepsilon,\eta}\to V_t$ in probability. Thus, we prove the
proposition by  Lebesgue's convergence theorem.
\end{proof}

%
%From this theorem we can get a very important result, and will be very important to the latter proof.
Let us denote
\begin{equation*}
  \begin{split}
  \mathcal{S}_{\mathbb{F}}^{p}(0, T ; \mathbb{R}):= \Big\{\psi=\left(\psi_{s}\right)_{s \in[0, T]}:&\text {$\psi$ is a real-valued}\  \cF   \text {-adapted} \\
 & \text { continuous process; }
  \EE\left[\sup_{0 \leq s \leq T}\left|\psi_{s}\right|^{p}\right]<\infty \Big\}.
  \end{split}
\end{equation*}
To prove $Y=\{Y_t, t\in[0,T]\}\in \mathcal{S}^p_{\mathbb{F} }(0,T;\RR)$ for all $p\in [1, q)$, we shall first recall  Talagrand's majorizing measure theorem.
\begin{lem}\label{majorizing}
(Majorizing Measure Theorem, see e.g.  [\cite[Theorem 2.4.2]{talagrand2014upper}]. Let $T$ be a given set and let $\left\{X_{t}, t \in T\right\}$ be a centred Gaussian process indexed by $T$. Denote by $d(t, s)=\left(\mathbb{E}\left|X_{t}-X_{s}\right|^{2}\right)^{\frac{1}{2}}$  the associated natural metric on $T$. Then
$$
\mathbb{E}\left[\sup _{t \in T} X_{t}\right] \asymp \gamma_{2}(T, d):=\inf _{\mathcal{A}} \sup _{t \in T} \sum_{n \geq 0} 2^{n / 2} \operatorname{diam}\left(A_{n}(t)\right),
$$
where ''$\asymp$'' indicates the asymptotic notation. Note that the infinimum is taken over all increasing sequence $\mathcal{A}:=\left\{\mathcal{A}_{n}, n=1,2, \cdots\right\}$ of partitions of $T$ such that \#$ \mathcal{A}_{n} \leq 2^{2^{n}}$ (\#A denotes the number of elements in the set $A$ ), $A_{n}(t)$ denotes the unique element of $\mathcal{A}_{n}$ that contains $t$, and $\operatorname{diam}\left(A_{n}(t)\right)$ is the diameter (with respect to the natural distance $d(\cdot,\cdot)$ ) of $A_{n}(t)$.
\end{lem}
We shall apply the above majorizing measure theorem to  $V_t = \int_t^T {{W}}(ds,B_s)$ as a random variable of $W$ (which is Gaussian under the conditional law knowing $B$).  The associated natural metric  (which is a random variable of $B$) is
%\begin{equation}\label{dsit}
%d(t,s) = \sqrt {(\EE^W |V_t-V_s|^2)}.
%\end{equation}
%Without loss of generality, let us
(assuming  $t > s$)
\begin{equation}\label{dst}
\begin{split}
d(t,s) &:= \sqrt {(\EE^W |V_t-V_s|^2)} = \sqrt {(\EE^W |\int_s^t W(dr,B_r)|^2)}\\
&= \sqrt { \int_s^t\int_s^t \alpha_{H_0} |u-v|^{2H_0 -2} \rho (B_u)\rho(B_r) \prod_{i=1}^d R_{H_i} (B_u^i, B_r^i) dudv}\\
&\leq C_H \upnu(B) |t-s|^{H_0},
\end{split}
\end{equation}
where $C_H$ is a constant depending only $H_i,\ i=1,\ldots, d$ and
\begin{equation}\label{upB}
\upnu(B) := C_d  (1\!+\! \sup_{u\in [0,T]}\sum_{i=1}^d |B_u^i|)^{\alpha}.
\end{equation}
Next, we   choose the admissible sequences $(\mathcal{A}_n)$ as uniform partition of $[0,T]$ such that  $\#( \mathcal{A}_n)\leq 2^{2^n} $:  %i.e., $[0,T]$ shall be partitioned as
$$[0, T]= \bigcup_{j=0}^{2^{2^{n-1}}-1}\left[j \cdot 2^{-2^{n-1}} T,(j+1) \cdot 2^{-2^{n-1}} T\right).$$
Thus, we can deduce that, by Lemma \ref{majorizing},
\begin{equation}\label{supV}
\EE^W \big[\sup_{t\in[0,T]} V_t \big] \leq C \sup_{t\in[0,T]} \sum_{n\geq 0} 2^{n/2} \operatorname{diam}\left(A_{n}(t)\right)\\
\end{equation}
where $A_{n}(t)$ is the element of uniform partition $\mathcal{A}_{n}$ that contains $(t)$, i.e.,
$$
A_{n}(t)=\left[j \cdot 2^{-2^{n-1}} T,(j+1) \cdot 2^{-2^{n-1}} T\right)
$$
such that $j \cdot 2^{-2^{n-1}} T \leq t<(j+1) \cdot 2^{-2^{n-1}} T$.\\
%The estimate of each $A_n(t)$ is easy to obtain.
Since $(\mathcal{A}_n)$ is a  uniform partition, and  by using the bound    \eqref{dst} we see the  diameter of $A_n(t)$ with respect to $d(t,s)$  can be estimated by
$$\mbox{diam}\big(A_n(t)\big)\leq C_H \upnu (B)2^{-H_0 2^{n-1}} T^{H_0}.$$
Inserting  this result  into   \eqref{supV}, we have
\begin{equation}\label{est supV}
\begin{split}
\EE^W \big[\sup_{t\in[0,T]} V_t\big] &\leq C \sup_{t\in[0,T]} \sum_{n\geq 0} 2^{n/2} \operatorname{diam}\left(A_{n}(t)\right)\\
&\leq C_H \upnu (B) T^{H_0} \sum_{n\geq 0} 2^{n/2} 2^{-H_0 2^{n-1}}  \leq C_H \upnu (B) T^{H_0}.
\end{split}
\end{equation}

We also need  the following two results to show $Y\in  \mathcal{S}_{\mathbb{F}}^{p}(0, T ; \mathbb{R})$.
\begin{lem}\label{Borell}
(Borell-TIS inequality, see e.g. \cite[Theorem 2.1]{adler1990introduction}). Let $\left\{X_{t}, t \in T\right\}$ be a centered separable Gaussian process on some topological index set $T$ with almost surely bounded sample paths. Then $\mathbb{E}\left(\sup _{t \in T} X_{t}\right)<\infty$, and for all $\lambda>0$,
$$
\mathbf{P}^W\left\{\left|\sup _{t \in T} X_{t}-\mathbb{E}\left(\sup _{t \in T} X_{t}\right)\right|>\lambda\right\} \leq 2 \exp \left(-\frac{\lambda^{2}}{2 \sigma_{T}^{2}}\right),
$$
where $\sigma_{T}^{2}:=\sup _{t \in T} \mathbb{E}\left(X_{t}^{2}\right).$
\end{lem}
%
%The following observation can be deduced directly from $[23$, Lemma 2.2.1]. This simple fact tells us $\mathbb{E}\left[\sup _{t \in T}\left|X_{t}\right|\right] \simeq \mathbb{E}\left[\sup _{t \in T} X_{t}\right]$. So, we just need to consider $\mathbb{E}\left[\sup _{t \in T} X_{t}\right]$.
\begin{lem}\label{abssupV}
If the process $\left\{X_{t}, t \in T\right\}$ is symmetric, then we have
\begin{equation}\label{absVeqV}
\mathbb{E}\left[\sup _{t \in T}\left|X_{t}\right|\right] \leqslant 2 \mathbb{E}\left[\sup _{t \in T} X_{t}\right]+\inf _{t_{0} \in T} \mathbb{E}\left[\left|X_{t_{0}}\right|\right].
\end{equation}
\end{lem}

Now we can state and prove  one of the main results of this  work.
\begin{thm}\label{sup Y}
  Suppose $\xi\in L^q(\Om)$ for some $q>2$  and
 suppose that  \eqref{e.2.7}  holds.
%  , $\sum_{i=1}^d (2H_i - \beta_i)\in (-1,2)$ and $2H_0+\sum_i^d H_i-\beta_i/2>1$.
Then we have $Y^{\ep,\eta}$ converges to
   $Y=\{Y_t, t\in[0,T]\}\in \mathcal{S}^p_{\mathbb{F} }(0,T;\RR)$ for all $p\in [1, q)$.
\end{thm}
\begin{proof}
We just need to verify $Y_t \in \mathcal{S}^p_{\mathbb{F} }(0,T;\RR)$. Let $q'p=q$ and $1/p'+1/q'=1$. By  \eqref{so appro} and  Jessen's inequality and Doob's martingale inequality we see
  \begin{eqnarray*}
    \EE    \big| \sup_{t\in[0,T]} Y_t \big|^p &=& \EE  \Big| \sup_{t\in[0,T]} \EE^B\Big[\xi \exp\big(\int_t^T
      {{W}}(ds ,B_s)\big)\,\big|\,\mathcal{F}_t^B\Big] \Big|^p \\
      &\leq& \EE\left|\sup_{t\in[0,T]}\EE^B\left[|\xi|^p \exp\{ p \sup_{t\in[0,T]}|V_t|\}\,\big|\,\cF_t^B \right] \right|\\
    &\leq&  \left(\frac{p}{p-1} \right)^p \| \xi\|_{q}^{1/q'} \Big[\EE\exp\big( pp'\,\sup_{t\in[0,T]}|V_t|\big)\Big]^{1/p'}.
  \end{eqnarray*}
Denote by $\|V\|_T :=  \sup_{t\in[0,T]}V_t  $. From Lemma \ref{Borell} and Lemma \ref{abssupV} it follows  for all $\lambda >0$,
\begin{equation}
\mathbf{P}^W\left\{\left|\|V\|_T-\EE^W\|V\|_T\right| \geq \lambda\right\} \leq 2 \exp\left(-\frac{\lambda^2}{2\sigma_T^2}\right), \label{e.3.8}
\end{equation}
The above term  $\sigma_T^2  $  is  defined and bounded by
\begin{equation}\label{sigmaT}
\begin{split}
\sigma_T^2= \sup_{t\in[0,T]} \EE^W [|V_t|^2]  \leq C_{T, H_0}  \sup_{u\in[0,T]}\rho(B_u)\prod_{i=1}^d |B_u^i|^{2H_i} \leq C_{T, H_0,d}\big(1+ \sup_{s\in[t,T]} \sum_{i=1}^d |B_s^i |\big)^{\alpha},
\end{split}
\end{equation}
by \eqref{vtv} and \eqref{estrhoR}.  From \eqref{e.3.8}  we have for any $m>0$,
\begin{align}\label{estYinS2}
\EE^W\Big[\exp\big( m\|V\|_T\big)\Big] &= \EE^W\left[\exp\left\{ m(\|V\|_T-\EE^W[\|V\|_T])\right\}\right]\cdot \exp\left\{m\EE^W[\|V\|_T]\right\}\notag\\
&\leq m\exp\left\{m\EE^W[\|V\|_T]\right\} \int^\infty_0 e^{m\lambda}\mathbf{P}\left(\left|\|V\|_T-\EE^W[\|V\|_T]\right|\geq \lambda\right)d\lambda\notag\\
&\leq 2m\exp\left\{m\EE^W[\|V\|_T]\right\}  \int^\infty_0 e^{m\lambda}\cdot e^{-\frac{\lambda^2}{2\sigma_T^2}} d\lambda\\
&\leq 2\sqrt{2\pi} m\cdot \sigma_T  \exp\left\{mC_H T^{H_0}\upnu(B)+\frac{m^2}{2}\sigma_T^2\right\}.\notag
%&\leq m\int_0^\infty \exp\big(m(\lambda+k)-\frac{\lambda^2}{2\sigma_T^2}\big)d\lambda
%&\leq C_{T, H_0, m} \cdot\sigma_T \exp\left\{C_{T,H_0,m}\upnu(B)+\frac{m^2}{2}\sigma_T^2\right\}\\
%&\leq C_{T, H_0, m} \exp\left\{C_{T, H_0, m} \sup_{u\in[0,T]}\rho(B_u)\prod_{i=1}^d |B_u^i|^{2H_i} \right\}\notag\\
%&\leq C_{T, H_0, m} \exp\left\{C_{T, H_0, m} (1+\sup_{u\in [0,T]} \sum_{i=1}^d|B_u^i|)^{\alpha } \right\}.\notag
\end{align}
Since for all $x>0$, we have $x\cdot e^{\frac{x^2}{2}}\leq 2e^{x^2}$. Therefore, taking account \eqref{upB} and \eqref{sigmaT} it yields that, there is a constant $C_{T,H,m,d}$ which only depends on $T,m,d,H_i,\ i=0,1,\ldots,d$ such that:
\begin{equation*}
\begin{split}
\EE^W\Big[\exp\big( m\|V\|_T\big)\Big]\leq 4\sqrt{2\pi}\exp\left\{C_{T, H, m,d} (1+\sup_{u\in [0,T]} \sum_{i=1}^d|B_u^i|)^{\alpha } \right\}.
\end{split}
\end{equation*}
%\begin{equation*}
%\begin{split}
%\EE^W\Big[\exp\big( m\|V\|_T\big)\Big] &= \int^\infty_0 e^{m\lambda}\Big(-\frac{d}{d\lambda} \mathbf{P}\left(\|V\|_T\geq \lambda\right)\Big)d\lambda=  m \int^\infty_0 e^{m\lambda}  \mathbf{P}\left(\|V\|_T\geq \lambda\right)d\lambda\\
%&\leq m\int_0^\infty \exp\big(m(\lambda+k)-\frac{\lambda^2}{2\sigma_T^2}\big)d\lambda\le C_{T, H_0, m} \exp(c_1m\upnu(B)+\frac{m^2}{2}\sigma_T^2)\\
%&\le C_{T, H_0, m} \exp(c_1 \sup_{u\in [0,T]} \prod_{i=1}^d (1+ |B_u^i|)^{2H_i-\beta_i } )\,.
%\end{split}
%\end{equation*}
%Thus, with the help of fundamental inequality we have
%\begin{equation}\label{supYineq}
%\EE^B\big| \sup_{t\in[0,T]} Y_t \big|^p \leq C \| \xi\|_{q}^{1/q'} \Big(\EE^B\big[\sigma_T^2 +\exp(2mk) +\exp(m^2\sigma^2_T)\big]\Big)^{1/p'}.
%\end{equation}
%
%Moreover,
%\begin{equation}\label{pk}
%\EE^B\left[\exp\left(2mk\right)\right]= \EE^B\left[\exp\left(2mC_H f(B)T^{H_0}\right)\right]\leq C\,\EE^B\left[\exp\left(2m\sup_{u\in[0,T]}\rho(B_u)\prod_{i=1}^d |B_u^i|^{2H_i}\right)\right].
%\end{equation}
By the Fernique's theorem we obtain $\EE
\left[ \EE^W\big[\exp\big( m\|V\|_T\big)\big]\right] < \infty$,
%Finally, taking \eqref{sigmaT} and \eqref{pk} into \eqref{supYineq} and Proposition \ref{exp int V} make sure that
which implies $\EE \big| \sup_{t\in[0,T]} Y_t \big|^p <\infty$. That is to say    $Y=\{Y_t, t\in[0,T]\}\in \mathcal{S}^p_{\mathbb{F} }(0,T;\RR)$ for all $p\in[1,q)$. The convergence of $Y^{\ep,\eta}$   to
   $Y=\{Y_t, t\in[0,T]\}\in \mathcal{S}^p_{\mathbb{F} }(0,T;\RR)$ for all $p\in [1, q)$
   is routine and a little bit more complicated. But the essential
    estimates are the  same as above.
%(Parhdoux, Prop 1.56) tells us that, with the help of Proposition \ref{3.1}, we deduce
%  \begin{eqnarray*}
%  \EE    \big| \sup_{t\in[0,T]} Y_t^{\ep,\eta} \big|^p
%    &{\red \leq}& {\red \| \xi\|_{\infty}^p \left(\frac{p}{p-1}\right)^p \EE^B  \bigg[ \exp\Big(Cp^2/2  \int_{[0,T]^2} |u-v|^{2H_0-2}}\\
%      &&\qquad {\red \prod_{i=1}^d \big(|1+B_v^i|^{2H_i-\beta_i} +|1+B_u^i|^{2H_i-\beta_i}\big) dudv \Big)\Big| \bigg],}
%  \end{eqnarray*}
%  where we still use the fact that   {\red  $\sup_{t\in[0,T]}\int_t^T \tilde{\dot{W}}_{\varepsilon,\eta}(s,x)ds$ is a Gaussian r.v.}.
%    Therefore, with the
 % help of Proposition \ref{exp int V}, we get $Y^{\ep,\eta}\in \mathcal{S}^p_{\mathbb{F} }(0,T;\RR^d)$. Finally,
 % Proposition \ref{3.1} and Lemma \ref{Y converge} tell us that $Y\in \mathcal{S}^p_{\mathbb{F} }(0,T;\RR^d)$.
\end{proof}

Now we want to study the second component of the solution pair of \eqref{e appro1}, i.e.   $Z^{\ep,\eta}=\{Z_s^{\ep,\eta},\,s\in[0,T]\}$ defined by
\eqref{so appro}. Introduce the space
\[\mathcal{M} ^2_{\mathbb{F} }(0,T;\RR^d):=\Big\{ \phi=(\phi_s)_{s\in[0,T]}:\RR^d\mbox{-valued}\ \mathbb{F}\mbox{-progressively measurable and  }
  \ \EE\big[\int_0^T |\phi_s|^2 ds\big]<\infty\Big\}.\]
\begin{thm}\label{z}
Denote
\[
\bar H=\max\{H_0,H_1, \cdots, H_d\}\,\quad \text{and} \quad \underbar
H =\min \{H_0, H_1, \cdots, H_d\}\,.
\]
Suppose $\sum_{i=1}^d (2H_i - \beta_i)<2$, terminal condition $\xi\in \mathbb{D}^{1,q}_B$ is measurable w.r.t. $\sigma$-field $\cF_T^B$, for $\displaystyle q>\frac{2}{2\underline{H} -1}$.
Then   $Z^{\ep,\eta}\in \mathcal{M} ^2_{\mathbb{F} }(0,T;\RR^d) $   and  $Z^{\ep,\eta}$ has a limit $Z=\{Z_s,\,s\in[0,T]\}$ in $\mathcal{M} ^2_{\mathbb{F} }(0,T;\RR^d)$.  This limit   can be written as
\begin{eqnarray}
  Z_t
  &=& D_t^BY_t  =D_t^B\EE \left[\xi   \exp\big(\int_t^T   {W}(dr,B_r)\big)\,\big|\,\mathcal{F}_t\right]\label{e.3.9} \\
&=&\EE \bigg[   e^{\int_t^T {W} (d\tau, B_\tau) } D_t^B\xi
  +\int_t^T e^{\int_t^s { {W}} (d\tau, B_\tau) }  Y_s  (\nabla_x  W)
       (ds, B_s)
  \big|\cF_t\bigg] \, .\label{e.3.10}
\end{eqnarray}
% The mean square of $Z_t$ is given by
%\begin{equation}
%\begin{split}
%\EE(Z_t^2)
%=&\EE^{B^1\!, B^2} \Big\{ e^{\int_t^T W(ds,B_s^1)+\int_t^T W(ds,B_s^2)}\bigg[ D_t^{B^1}(\xi(B^1))D_t^{B^1}(\xi(B^2))\\
%&\quad +\big(D_t^{B^1}(\xi(B^1))\,\xi(B^2)\big)\int_t^T\nabla_x {W}(dr,B_r^2) \\
%&\quad +\big(D_t^{B^2}(\xi(B^2))\,\xi(B^1)\big)\int_t^T\nabla_x {W}(dr,B_r^1) \bigg]  \bigg|  \cF_t^{B^1, B^2}  \Big\} \bigg |_{B^1=B^2=B}\\
%&+ \al_{H_0} \int_t^T \int_t^T \EE^{B^1\!, B^2}
%\Big\{  \xi(B^1) \xi(B^2) |s_2-s_1|^{2H_0-2}
%  \Big[(\partial _{x_i,y_i}
%q)( B_{s_1}^1,B^2_{s_2})\\
%& \quad+\int_t^T |{s_1}-u|^{2H_0-2}
% \left[  \partial _{x_i} q(B_{s_1}^1, B_u^2) + \partial _{x_i} q(B_{s_1}^1, B_u^1)    \right]  du\\
%   &\quad + \int_t^T |{s_1}-u|^{2H_0-2}
% \left[  \partial _{x_i} q(B_{s_1}^2, B_u^2) + \partial _{x_i} q(B_{s_1}^2, B_u^1) \right]  du\Big] \\
%&\qquad    \times \Upupsilon(t, T, B^1, B^2)
%  \Big|  \cF_t^{B^1, B^2}  \Big\} \bigg |_{B^1=B^2=B} ds_1ds_2\,,
%\end{split}\label{e.3.11}
%\end{equation}
%where $\partial _{x_i,y_i}q(x,y)=\frac{\partial^2}{
%\partial x_i\partial y_i} q(x,y)$, $B^1, B^2 $ are two independent standard Brownian motions
%which are identical copies of  $B$, $\mathcal{F}_t^{B^1,B^2} := \sigma\{B_s^1,B_r^2,\ 0\leq s,r\leq t\}$ is the $\sigma$-algebra generated
% by $B^1,B^2$ and
%\begin{equation}
%\Upsilon:=\exp\bigg\{ \int_t^T\int_t^T \alpha_{H_0} |u-v|^{2H_0-2}
%  \big[ q(B^1_u, B^1_v)+2 q(B^1_u, B^2_v)
%  +q(B^2_u, B^2_v)\big] dudv \bigg\} \,.
%\end{equation}
\end{thm}
%\begin{remark}
%Although it is a generalized function, in the sense of the expected square, this function still has good definitions and properties.
%\end{remark}
\begin{proof}   Presumably we may apply $D_r^B$ to $Y_t^{\varepsilon,\eta}$ given by \eqref{so appro}.  But it is inconvenient to  deal with the Malliavin derivative of the conditional expectation. We find that it is more convenient to find $D_r^B Y_t^{\varepsilon,\eta}$ by working on
\eqref{e appro1} directly.  In fact applying $D_r^B$ to \eqref{e appro1}   yields
\begin{eqnarray*}
  D_r ^B Y_t^{\varepsilon,\eta}
  &=& D_r^B \xi + \int_t^T {\dot{W}}_{\varepsilon,\eta}(s,B_s)   D_r ^B Y_s^{\varepsilon,\eta} ds \\
  &&\qquad  +\int_t^T Y_s^{\varepsilon, \eta} \nabla_x {\dot{W}}_{\varepsilon, \eta}(s, B_s) I_{[0, s]}(r) ds
  -\int_t^T D_r^B Z_s^{\varepsilon,\eta} dB_s.
\end{eqnarray*}
%\footnote{be careful that $D_r ^B Y_t^{\varepsilon,\eta}$ is a vector}
Denote  $\tilde Y_t=D_r^BY_t^{\varepsilon, \eta},\ \tilde Z_t=D_r^B Z_t^{\varepsilon, \eta}$ (we fix $r$)   and  we can rewrite the above   equation
as
\[
\begin{cases}
  d\tilde Y_t=-
  {\dot{W}}_{\varepsilon, \eta}(t, B_t) \tilde Y_t dt-Y_t^{\varepsilon, \eta} \nabla_x {\dot{W}}_{\varepsilon, \eta}(t, B_t) I_{[0, t]}(r)dt
  +\tilde Z_t dB_t,\  \ r\leq t\le T\\
  \tilde Y_T=D_r^B\xi\,.
  \end{cases}
\]
This is another linear backward stochastic differential equation, whose solution   has the following  explicit form.
\begin{eqnarray*}
  D_r^BY_t^{\varepsilon, \eta}
  &=&\EE \bigg[   e^{\int_t^T {\dot{W}}_{\varepsilon, \eta}(\tau, B_\tau)
      d \tau} D_r^B\xi \\
  &&\qquad
  +\int_r^T e^{\int_t^s {\dot{W}}_{\varepsilon, \eta}(\tau, B_\tau)
      d\tau} Y_s^{\varepsilon, \eta} \nabla_x {\dot{W}}_{\varepsilon, \eta}(s, B_s)  ds
  \big|\cF_t\bigg]\ ,\quad t\ge r.
\end{eqnarray*}
%\footnote{be careful that $\nabla_x \tilde{\dot{W}}_{\varepsilon, \eta}(s, B_s)$ is a vector in later computations and expressions}
By   \cite[Equation (2.11)]{HNS2011} and  \cite[Equation (2.23)]{HNS2011}
  we have
%\footnote{We need to extend the result in \cite{HNS2011} to multiple Brownian motion cases?}
\begin{eqnarray*}
  Z_t^{\varepsilon, \eta}
  &=& D_t^BY_t^{\varepsilon, \eta}
  =\EE \bigg[   e^{\int_t^T {\dot{W}}_{\varepsilon, \eta}(\tau, B_\tau)
      d \tau} D_t^B\xi \\
  &&\qquad
  +\int_t^T e^{\int_t^s {\dot{W}}_{\varepsilon, \eta}(\tau, B_\tau)
      d\tau} Y_s^{\varepsilon, \eta} \nabla_x {\dot{W}}_{\varepsilon, \eta}(s, B_s)  ds
  \big|\cF_t\bigg] \,\\
  &:=& Z_t^{0, \varepsilon, \eta}+Z_t^{1, \varepsilon, \eta}.
\end{eqnarray*}
Assuming $D_r^B\xi$ is nice, we   $Z_t^{0, \varepsilon, \eta}$ can be treated in exactly the same way as $Y_s^{\varepsilon, \eta}$.

We shall focus our effort  on showing  $Z_t^{1, \varepsilon, \eta}\in \mathcal{M} ^2_{\mathbb{F} }(0,T;\RR^d)$.
 Substituting $Y_t
^{\varepsilon, \eta}$  given by
\eqref{so appro}  into the above expression, we have
%and noticing
%\begin{equation}
%\int_t^s {\dot{W}}_{\varepsilon, \eta}(\tau, B_\tau)
%      d\tau \,, \quad \nabla_x {\dot{W}}_{\varepsilon, \eta}(s, B_s)\in \cF_s^B\,,
%      \label{e.4.5}
%\end{equation}
\begin{eqnarray*}
  Z_t^{1, \varepsilon, \eta}
  &=& \EE  \bigg[  \int_t^T e^{\int_t^s {\dot{W}}_{\varepsilon, \eta}(\tau, B_\tau)
      d\tau} \EE \Big[\xi \exp\big(\int_s^T {\dot{W}}_{\varepsilon,\eta}(u,B_u)du\big)\,\big|\,\mathcal{F}_s \Big] \nabla_x {\dot{W}}_{\varepsilon, \eta}(s, B_s)  ds
  \big|\cF_t \bigg] \\
%  &=& \EE^B \bigg[    \int_t^T   \EE^B\Big[\xi\, \nabla_x {\dot{W}}_{\varepsilon, \eta}(s, B_s) \exp\big(\int_t^T
%  {\dot{W}}_{\varepsilon,\eta}(u,B_u)du\big)\,\big|\,\mathcal{F}_s^B\Big]    ds
%  \big|\cF_t^B\bigg]\\
  &=&     \int_t^T   \EE \Big[\xi\, \nabla_x {\dot{W}}_{\varepsilon, \eta}(s, B_s)
   \exp\big(\int_t^T {\dot{W}}_{\varepsilon,\eta}(u,B_u)du\big)\,\big|\,\mathcal{F}_t \Big]    ds \,.
\end{eqnarray*}
Since it involves  the term  $\nabla_x {\dot{W}}_{\varepsilon, \eta}(s, B_s)$, this term is much more difficult to deal with.
% bound $Z_t^{1, \varepsilon, \eta}$ by using the conventional way.
We shall fully explore the normality of the Gaussian
field $W$.  Moreover,  there is a conditional expectation in the expression of $Z_t^{1, \varepsilon, \eta}$ which  seems to stop us carrying out any meaningful computations.  We shall get around this difficulty by   introducing  two independent standard Brownian motions $B^1, B^2 $
which are identical copies of the
Brownian motion $B$.   Denote $\mathcal{F}_t^{B^1,B^2} = \sigma\{B_s^1,B_r^2,\ 0\leq s,r\leq t;W(t,x),\ t\geq0,x\in\RR^d\}$ by the $\sigma$-algebra generated
 by sBm $B^1,B^2$ up to time instant $t$ and $W(t,x)$ for all $t\geq0$ and $x\in\RR^d$. \\
Note that, $\EE^W$  only denotes the expectation with respect to $W$, which consider other random variables as  "fixed constant".
%we can regard it as $\EE^W(\,\cdot\,)=\EE(\,\cdot\,|\cF^B\vee\cF^{B^1,B^2})$.
% Similarly, the expectation only with respect to $B$, $\EE^B$, can be thought of as  $\EE^B(\,\cdot\,)=\EE(\,\cdot\,|\cF^W\vee\cF^{B^1,B^2})$ and $\EE^{B^1,B^2} (\,\cdot\,)=\EE(\,\cdot\,|\cF^W\vee\cF^{B})$.
Then, we have
\begin{eqnarray}
  \EE^W \left[ Z_t^{1, \varepsilon, \eta}\right]^2
  &=&  \EE^W \int_t^T\int_t^T    \EE \Big[\xi(B^1)\xi(B^2)\big(\nabla_x {\dot{W}}_{\varepsilon, \eta}({s_1}, B_{s_1}^1)\big)^T
  \nabla_x {\dot{W}}_{\varepsilon, \eta}({s_2}, B_{s_2}^2)\nonumber  \\
  &&\qquad \exp\big(\int_t^T \left[
  \dot{W}^H_{\varepsilon,\eta}(u,B_u ^1 )
  +\dot{W}^H_{\varepsilon,\eta}(u,B_u^2  )\right] du\big)\big|\,\mathcal{F}_t^{B^1,B^2}\Big] \Big|_{B^1 = B^2=B}   ds_1ds_2\nonumber \\
%  &=&   \int_t^T\int_t^T    \EE^{B^1, B^2} \Big[\xi(B^1)\xi(B^2)  \EE^W\bigg\{ \big(\nabla_x{\dot{W}}_{\varepsilon, \eta}({s_1}, B_{s_1}^1)\big)^T
%  \nabla_x {\dot{W}}_{\varepsilon, \eta}({s_2}, B_{s_2}^2) \\
%  &&\qquad \exp\big(\int_t^T \left[
%    {\dot{W}}_{\varepsilon,\eta}(u,B_u ^1 )
%  +{\dot{W}}_{\varepsilon,\eta}(u,B_u^2  )\right] du
%  \big) \bigg\} \,\big|\,\mathcal{F}_t^{B^1,B^2}\Big] \Big|_{B^1 = B^2=B}   ds_1ds_2\\
  &=&   \int_t^T\int_t^T    \EE  \Big[\xi(B^1)\xi(B^2)   I^{\vare, \eta}(s_1, s_2) \,\big|\,\mathcal{F}_t^{B^1,B^2}\Big] \Big|_{B^1 = B^2=B}   ds_1ds_2\, ,\label{eqn.z_compute}
\end{eqnarray}
where  $I^{\vare, \eta}(s_1, s_2)$   is defined by
\begin{equation}\label{beforeA_1}
  \begin{split}
   I^{\vare, \eta}(s_1, s_2 )  &:=
   \sum_{i=1}^d \EE^W\bigg\{ \nabla_{x_i} {\dot{W}}_{\varepsilon, \eta}({s_1}, B_{s_1}^1)
  \nabla_{x_i} {\dot{W}}_{\varepsilon, \eta}({s_2}, B_{s_2}^2)  \exp\left(\int_t^T \left[
    {\dot{W}}_{\varepsilon,\eta}(u,B_u ^1 )
  +{\dot{W}}_{\varepsilon,\eta}(u,B_u^2  )\right] du
  \right) \bigg\}\,.
  \end{split}
\end{equation}
Denote
\[
\begin{cases}
Z_{1,i}^{\vare, \eta}=\nabla_{x_i} {\dot{W}}_{\varepsilon, \eta}({s_1}, B_{s_1}^1)\,; \quad
Z_{2,i}^{\vare, \eta}=\nabla_{x_i} {\dot{W}}_{\varepsilon, \eta}({s_2}, B_{s_2}^2)\,;  \\
Y^{\vare, \eta}=\displaystyle\int_t^T \left[
    {\dot{W}}_{\varepsilon,\eta}(u,B_u ^1 )
  +{\dot{W}}_{\varepsilon,\eta}(u,B_u^2  )\right] du\,.
\end{cases}
\]
%For the latter simplification, we omit the dependence of $Z_{1,i}^{\vare, \eta}, Z_{2,i}^{\vare, \eta}$ on $i$.
Then
\begin{equation*}
  \begin{split}
   I^{\vare, \eta}(s_1, s_2 )  =\sum_{i=1}^d \EE^W\bigg\{Z_{1,i}^{\vare, \eta} Z_{2,i}^{\vare, \eta} \exp{(Y^{\vare, \eta})}\bigg\}\,.
  \end{split}
\end{equation*}
As  random variables of $W$ (namely for fixed $B^1, B^2$), $Z_{1,i}^{\vare, \eta}, Z_{2,i}^{\vare, \eta}, Y^{\vare, \eta}$ are jointly   Gaussians, we shall use the following lemma to compute the above expectations.
\begin{lem}\label{x1x2y}
  Assume that $X_1, X_2, Y $ are  jointly  mean zero  Gaussians. Then
\begin{equation}
\EE \left[ X_1X_2\exp(Y )\right] = \left(\EE(X_1Y )+\EE(X_2Y )+\EE(X_1X_2)\right)\exp  \Big[  \frac12 \EE(Y ^2)  \Big]\,. \label{e.3.5}
  \end{equation}
 \end{lem}
 \begin{proof} For any constants $s,  t\in \RR$ we have
   \begin{eqnarray*}
  &&   \EE \exp(Y +sX_1+tX_2)
      = \exp \left\{\frac12 \EE(Y +sX_1+tX_2)^2\right\}\\
     & &
     \qquad \quad =\exp  \Bigg[ \bigg\{ \EE(Y ^2)
       +s^2\EE(X_1^2)+t^2\EE(X_2^2)+2s\EE(X_1Y )  +2t\EE(X_2Y )
       + 2st\EE(X_1X_2) \bigg\}\Big/2\Bigg]\,.
   \end{eqnarray*}
   Thus
   \begin{eqnarray*}
     \EE \left[X_1X_2\exp(Y )\right]
     &=&\frac{\partial ^2}{\partial s\partial t}\big|_{s=t=0} \exp \left\{\frac12 \EE(Y +sX_1+tX_2)^2\right\}\\
     &=& \left(\EE(X_1Y )+\EE(X_2Y )+\EE(X_1X_2)\right)
     \exp  \Bigg[  \frac12 \EE(Y ^2)
       \Bigg]  \,.
   \end{eqnarray*}
This is \eqref{e.3.5}.
 \end{proof}

% Thus, we can use Lemma \ref{x1x2y} to evaluate the above  expectation.
%Recall that
%$$\displaystyle{\dot{W}}_{\varepsilon,\eta}(s,B_s) = \int_0^s \int_{R^d} \varphi_\eta (s-r) p_\varepsilon (B_s -y) W(dr,y)dy\,.  $$
%Thus, we  have
%$$\displaystyle \nabla_{x_i}\dot{W}_{\varepsilon,\eta}(s,B_s) = \int_0^s \int_{R^d} \varphi_\eta (s-r)  \nabla_{x_i}p_\varepsilon  (B_s -y) W(dr,y)dy\,. $$
Applying   the above  Lemma \ref{x1x2y}  to evaluate
$I^{\vare, \eta}(s_1, s_2)$  yields
\begin{align*}
  \EE^W \left[ Z_t^{1, \varepsilon, \eta}\right]^2
%  =&\int_t^T\int_t^T    \EE^{B^1, B^2} \Big[\xi(B^1)\xi(B^2)   I^{\vare, \eta}(s_1, s_2) \,\big|\,\mathcal{F}_t^{B^1,B^2}\Big] \Big|_{B^1 = B^2=B}   ds_1ds_2\\
%  =& \int_t^T\int_t^T    \EE^{B^1, B^2} \Big[\xi(B^1)\xi(B^2)\sum_{i=1}^d \left(\EE^W(Z_{1,i}^{\vare, \eta}Y^{\vare, \eta})+\EE^W(Z_{2,i}^{\vare, \eta}Y^{\vare, \eta})+\EE^W(Z_{1,i}^{\vare, \eta}Z_{2,i}^{\vare, \eta})\right)\\
%  &\qquad\qquad\times\exp  \Big[  \frac12 \EE^W \left(Y^{\vare, \eta}\right)^2   \Big]\,\big|\,\mathcal{F}_t^{B^1,B^2}\Big]\Big|_{B^1=B^2=B}      ds_1ds_2\\
  =& \int_t^T\int_t^T    \EE \bigg[\xi(B^1)\xi(B^2)\left(\sum_{i=1}^d \left(A_{1,i}^{\vare, \eta}  +
  A_{2,i}^{\vare, \eta}+ A_{3,i}^{\vare, \eta}\right)\right)\\
  &\qquad\qquad\qquad \exp\left(\frac{A_{4}^{\vare, \eta}}{2}\right)\,\big|\,\mathcal{F}_t^{B^1,B^2}\bigg]\bigg|_{B^1=B^2=B}      ds_1ds_2\\
  =&\sum_{j=1}^3 \sum_{i=1}^d I_{j,i, t}^{\vare, \eta}\,,
\end{align*}
%%$I^{\vare, \eta}(s_1, s_2)$,
%To this end
where
\begin{equation}
\left\{\begin{split}
A_{1,i}^{\vare, \eta}
:=&\EE^W (Z_{1,i}^{\vare, \eta}Z_{2,i}^{\vare, \eta}), \quad A_{2,i}^{\vare, \eta}
:=\EE^W (Z_{1,i}^{\vare, \eta}Y^{\vare, \eta}),
\\
 A_{3,i}^{\vare, \eta}
:=&\EE^W (Z_{2,i}^{\vare, \eta}Y^{\vare, \eta}),
\quad   A_{4}^{\vare, \eta}
:=\EE^W (\left(Y^{\vare, \eta}\right)^2)
\end{split} \right.
\end{equation}
and
\begin{equation}
I_{j,i, t}^{\vare, \eta}
:=\int_t^T\int_t^T    \EE  \bigg[\xi(B^1)\xi(B^2) A_{j,i}^{\vare, \eta}    \exp\left(\frac{A_{4}^{\vare, \eta}}{2}\right)\,\big|\,\mathcal{F}_t^{B^1,B^2}\bigg]\bigg|_{B^1=B^2=B}      ds_1ds_2\,.
\end{equation}
Let us consider $I_{1,i,t}^{\vare, \eta}$ in details.  The other terms can be treated in similar way. First, let us compute
\begin{equation}\label{3.10star}
  \begin{split}
% \EE^W (Z_{1,i}^{\vare, \eta}Z_{2,i}^{\vare, \eta})
A_{1,i}^{\vare, \eta} =& \EE^W \bigg[  \nabla_{x_i}{\dot{W}}_{\varepsilon, \eta}(s_1, B_{s_1}^1)   \nabla_{x_i} {\dot{W}}_{\varepsilon, \eta}(s_2, B_{s_2}^2)\bigg]  \\
  =& \alpha_{H_0}  \int_0^{s_1}\int_0^{s_2}   \varphi_\eta (s_1-r_1)    \varphi_\eta (s_2-r_2)
  |r_2-r_1|^{2H_0 -2}dr_1dr_2\\
  &\qquad\qquad\qquad \int_{\RR^{2d}}\nabla_{x_i} p_\varepsilon  (B_{s_1}^1  -w)  \nabla_{x_i}p_\varepsilon  (B_{s_2}^2  -z)\rho(w)\rho(z)\prod_{i = 1}^{d} R_{H_i}(w_i,z_i)dwdz\\
  =& J_1^{  \eta}(s_1, s_2)
 J_2^{\vare }(s_1, s_2)    \,,
  \end{split}
\end{equation}
where $ J_1^{  \eta}(s_1, s_2) $  and  $J_2^{\vare }(s_1, s_2)$ are   defined   as follows.
\begin{equation*}
\left\{\begin{split}
 J_1^{\eta}(s_1, s_2)
  &:=
 \int_0^{s_1}\int_0^{s_2}     \varphi_\eta (s_1-r_1)    \varphi_\eta (s_2-r_2)
  |r_2-r_1|^{2H_0 -2}dr_1dr_2  \,, \\
  J_2^{  \vare }(s_1, s_2)
  &:=\int_{\RR^{2d}} \nabla_{w_i}p_\varepsilon (B_{s_1}^{1}  -w)\nabla_{z_i}p_\varepsilon (B_{s_2}^{2}  -z)\rho(w)\rho(z)\prod_{i = 1}^{d}R_{H_i}(w_i,z_i)dwdz \\
  &=\int_{\RR^{2d}} \nabla_{w_i}p_\varepsilon (B_{s_1}^{1}  -w)\nabla_{z_i}p_\varepsilon (B_{s_2}^{2}  -z)q(w, z)dwdz\,,
  \end{split}  \right.
  \end{equation*}
  where we recall that $q(x,y)$ is the
  spatial covariance of noise given by
   \eqref{e.1.7}.
Notice that $J_2^{\vare}(s_1, s_2)$ is independent of $\vare$.
It is  elementary to see that
\begin{equation}\label{jeta1}
J_1^\eta (s_1, s_2):=
 \int_0^{s_1}\int_0^{s_2}     \varphi_\eta (s_1-r_1)    \varphi_\eta (s_2-r_2)
  |r_2-r_1|^{2H_0 -2}dr_1dr_2 \to
  |s_2-s_1|^{2H_0 -2} \quad \hbox{ as $\vare, \eta\to 0$} \,.
\end{equation}
Moreover, for any $p<1/(2-2H_0)$ and $1/p+1/q=1$, by H\"older's inequality we have
\begin{equation*}
\begin{split}
|J_1^\eta(s_1, s_2)| ^p
\le & \int_0^{s_1}\int_0^{s_2}
  |r_2-r_1|^{(2H_0 -2)p }dr_1dr_2   \nonumber\\
&\qquad\qquad  \times \left(\int_0^{s_1}\int_0^{s_2}     \varphi_\eta (s_1-r_1)    \varphi_\eta (s_2-r_2)
   dr_1dr_2 \right)^{p/q}\,.
   \end{split}
\end{equation*}
The above second factor is less than or equal to $1$.
Making substitutions $ s_1-r_1\to r_1'\eta $ and
$ s_2-r_2\to r_2'\eta $ we  have
\begin{equation}\label{e.3.15}
\begin{split}
\sup_{\eta\in (0, 1]} |J_1^\eta  (s_1, s_2)| ^p
\le &\sup_{\eta\in (0, 1]} \int_0^{1 }\int_0^{1 }
  |s_2-s_1+\eta (r_1'-r_2')|^{(2H_0 -2)p }dr_1'dr_2'<\infty\,.
   \end{split}
\end{equation}
Now we consider $J_2^\vare$.
Integration by parts yields
\begin{align*}
  J_2^{\vare}(s_1, s_2)
  &= \int_{\RR^{2d}} p_\varepsilon (B_{s_1}^{1}  -w) p_\varepsilon (B_{s_2}^{2}  -z)\nabla_{w_i} \nabla_{z_i} q(w,z) dwdz\\
  &= \int_{\RR^{2d}} p_\varepsilon (B_{s_1}^{1}  -w) p_\varepsilon (B_{s_2}^{2}  -z)\prod_{j\neq i}^d R_{H_j} (w_j,z_j)\Big[ \nabla_{w_i}\rho(w)\nabla_{z_i} \rho(z)
  R_{H_i}(w_i,z_i)\\
  &\qquad + \nabla_{w_i}\rho(w)\rho({z})\Big(
   H_i|z_i|^{2H_i-1} \sign(z_i)-H_i|w_i-z_i|^{2H_i-1}\sign(w_i-z_i)  \Big)  \\
  &\qquad +    \rho(w) \nabla_{z_i}\rho(z) \Big(
   H_i|w_i|^{2H_i-1} \sign(z_i)-H_i|w_i-z_i|^{2H_i-1}\sign(w_i-z_i)  \Big)  \\
  &\qquad  + \rho(z)
   \rho(w)\alpha_{H_i}|w_i-z_i|^{2H_i-2}\Big]dw_idz_i\\
  &=J_{21}^{\vare}(s_1, s_2)+J_{22}^{\vare }(s_1, s_2)\,,
\end{align*}
where
\begin{align}\label{jep21}
 J_{21}^{\vare}(s_1, s_2)
:&=  \EE^{X,X'} \bigg[\prod_{j\neq i}^d R_{H_j}(B_{s_1}^{1,j}+\varepsilon  X_j,B_{s_2}^{2,j}+\varepsilon  X_j')\notag\\
  &\quad \times\bigg(\nabla_{x_i}\rho(B_{s_1}^{1}+\varepsilon  X ) \nabla_{x_i}\rho(B_{s_2}^{2}+\varepsilon  X' )
  R_{H_i}(B_{s_1}^{1,i}+\varepsilon  X_i,B_{s_2}^{2,i}+\varepsilon  X_i')  \notag \\
  &\qquad + \nabla_{x_i}\rho(B_{s_1}^{1}+\varepsilon  X )\rho({B_{s_2}^{2}+\varepsilon  X'})\Big(
   H_i|B_{s_2}^{2,i}+\varepsilon  X_i'|^{2H_i-1} \sign(B_{s_2}^{2,i}+\varepsilon  X_i')\\
   &\qquad\qquad -H_i|B_{s_1}^{1,i}+\varepsilon  X_i -B_{s_2}^{2,i}-\varepsilon  X_i'|^{2H_i-1}\sign(B_{s_1}^{1,i}+\varepsilon  X_i-B_{s_2}^{2,i}-\varepsilon  X_i')  \Big) \notag \\
  &\qquad +    \rho(B_{s_1}^{1}+\varepsilon  X) \nabla_{x_i}\rho'(B_{s_2}^{2}+\varepsilon  X') \Big(
   H_i|B_{s_1}^{1,i}+\varepsilon  X_i |^{2H_i-1} \sign(B_{s_2}^{2,i}+\varepsilon  X_i')\notag\\
   &\qquad\qquad-H_i|B_{s_1}^{1,i}+\varepsilon  X_i-B_{s_2}^{2,i}-\varepsilon  X_i'|^{2H_i-1}\sign(B_{s_1}^{1,i}+\varepsilon  X_i -B_{s_2}^{2,i}-\varepsilon  X_i')  \Big)\bigg) \bigg]\notag
\end{align}
and
\begin{align*}
 J_{22}^{\vare}(s_1, s_2)
  &:=  \alpha_{H_i}\EE^{X,X'} \bigg[\prod_{j\neq i}^d R_{H_j}(B_{s_1}^{1,j}+\varepsilon  X_j,B_{s_2}^{2,j}+\varepsilon  X_j')\\
  &\qquad\times\rho(B_{s_1}^{1}+\varepsilon  X)\rho(B_{s_2}^{2}+\varepsilon  X')|B_{s_1}^{1,i}+\varepsilon  X_i-B_{s_2}^{2,i}-\varepsilon  X_i'|^{2H_i-2}\bigg]
  \end{align*}
with  $X=(X_1, \cdots, X_d),X'=(X_1', \cdots, X_d')$ being   independent  standard Gaussian random variables,
which are also independent of $B^1, B^2$.
%Recall the fact that $2H_i-2>-1$, for all $i=1,\ldots, d$ and $\rho(x)$ is a bounded smoothly function, all of these integrals up here are integrable.
From the definition, we can consider $J_{21}^{\vare}(s_1, s_2)$ as a random variable of $B^1$ and $B^2$.
From the above expression it is easy to see that
\begin{equation}\label{e.3.16}
\sup_{\vare\in (0, 1]}\left|J_{21}^{\vare}(s_1, s_2)\right|
\le   C\left(1+|B_{s_1}^1|^m+| B_{s_2}^2 |^m\right)\,,
\end{equation}
for some positive constants $C$ and $m$.

As concerns for $J_{22}^\vare(s_1, s_2)$, we can find two constants  $p,q$ satisfying $p<1/(2-2H_0)$ and $1/p+1/q=1$ such that by H\"older's inequality,
\begin{align*}
 J_{22}^{\vare}(s_1, s_2)
  &\le   \alpha_{H_i}\left\{\EE^{X,X'} \left[\prod_{j\neq i}^d \left(R_{H_j}(B_{s_1}^{1,j}+\varepsilon  X_j,B_{s_2}^{2,j}+\varepsilon  X_j')\right)^q\rho^q(B_{s_1}^{1}+\varepsilon  X)\rho^q(B_{s_2}^{2}+\varepsilon  X')  \right]  \right\}^{1/q} \\
&\qquad\qquad \times   \left\{\EE^{X,X'} \left[ |B_{s_1}^{1,i}+\varepsilon  X_i-B_{s_2}^{2,i}-\varepsilon  X_i'|^{(2H_i-2)p}\Big)\right]  \right\}^{1/p} \,.
  \end{align*}
By the Lemma A.1 of \cite{HNS2011}
the above second expectation  is bounded by
$  \left|B_{s_1}^{1,i} -B_{s_2}^{2,i} \right|^{(2H_i-2)p} $. Thus, by the assumption on $\rho$ and by the definition of $R_H$ we have
\begin{equation}\label{e.3.17}
\sup_{\vare\in (0, 1]}\left|J_{22}^{\vare}(s_1, s_2)\right|
\le   C\left(1+|B_{s_1}^1|^m+| B_{s_2}^2 |^m\right)\cdot\left|B_{s_1}^{1,i} -B_{s_2}^{2,i} \right|^{(2H_i-2)p}\,.
\end{equation}
%Now \eqref{e.3.15}-\eqref{e.3.17} enables us to apply the dominated convergence theorem to
%$I_{1,1,t}^{\vare, \eta}$ to obtain
Moreover, from \eqref{3.10star}, \eqref{jeta1} and \eqref{jep21} we have
\begin{align*}
   \lim_{\eta,\varepsilon\rightarrow 0} A_{1,i} ^{\vare, \eta}  =& \lim_{\eta,\varepsilon\rightarrow 0}  \EE^W\bigg[\nabla_{x} {\dot{ W}}_{\varepsilon, \eta}(s_1, B_{s_1}^1)
    \nabla_ {x} {\dot{ W}}_{\varepsilon, \eta}(s_2, B_{s_2}^2)\bigg]\\
    = & \alpha_{H_0} |s_2-s_1|^{2H_0 -2}\prod_{j\neq i}^d R_{H_j} (B_{s_1}^{1,j},B_{s_2}^{2,j})
     \Big[\nabla_{x_i}\rho(B_{s_1}^{1} )\nabla_{x_i} \rho(B_{s_2}^{2})
  R_{H_i} (B_{s_1}^{1,i},B_{s_2}^{2,i})\\
  &\qquad + \nabla_{x_i}\rho(B_{s_1}^{1}  )\rho( B_{s_2}^{2} )\Big(
   2H_i|B_{s_2}^{2,i} |^{2H_i-1} \sign(B_{s_2}^{2,i} )  -2H_i|B_{s_1}^{1,i} -B_{s_2}^{2,i} |^{2H_i-1}\sign(B_{s_1}^{1,i}-B_{s_2}^{2,i} )  \Big)  \\
  &\qquad +    \rho(B_{s_1}^{1} ) \nabla_{x_i}\rho(B_{s_2}^{2} ) \Big(
   2H_i|B_{s_1}^{1,i} |^{2H_i-1} \sign(B_{s_2}^{2,i} ) +2H_i|B_{s_1}^{1,i} -B_{s_2}^{2,i} |^{2H_i-1}\sign(B_{s_1}^{1,i} -B_{s_2}^{2,i})  \Big)  \\
  &\qquad  + \alpha_{H_i}\rho(B_{s_1}^{1})
   \rho(B_{s_2}^{2} )|B_{s_1}^{1,i} -B_{s_2}^{2,i} |^{2H_i-2}\Big]\,.
%      \\
%   =: &A_1^i.
\end{align*}
Using the spatial covariance $q(x,y)$, we can   write
\begin{equation}
   \lim_{\eta,\varepsilon\rightarrow 0}A_{1,i}^{\vare, \eta}  =\alpha_{H_0} |s_2-s_1|^{2H_0 -2} \frac{\partial^2}{\partial x_i
   \partial  y_i}q(x,y) \Big|_{x=B^1_{s_1}, y=B^2_{s_2}}\,.
\end{equation}
%And by the fact that, in \eqref{3.10star} $I_{1,i}(s_1,s_2)$ and $I_{1,j}(s_1,s_2)$ are also bounded continuous functions of $\ep$ in the concerned domain (almost surly with respect to $B$), we can obtain
%$$\displaystyle \EE^W(Z_{1,i}^{\vare, \eta}Z_{2,i}^{\vare, \eta}) = C\, A_1^i,$$
%where $C$ is a constant independent on $\ep,\eta$.  \\
%$\bullet\ Step\ 2.$ Next, we consider
%\begin{align*}
%  %  \EE^W(Z_{1,i}^{\vare, \eta} Y^{\vare, \eta})=&
%  \EE ^W\bigg[Z_{1,i}^{\vare, \eta}\,\cdot\int_t^T {\dot{W}}_{\eta,\varepsilon}(u,B^1_u)du\bigg]
%  &=\alpha_{H_0} \int^T_t \int_t^{s_1}\int_t^u \varphi_\eta ({s_1}-r) \varphi_\eta (u-v)|r-v|^{2H_0 -2}dvdr\\
%  &\qquad\cdot \int_{\RR^{2d}} \nabla_{x_i} p_\varepsilon  (B_s^{1} -y)p_\varepsilon (B_u^{1} -z)\rho(y)
%  \rho(z)\prod_{i=1}^d R_{H_i}(y_i,z_i)dy_idz_idu\,.
%\end{align*}
%
%\begin{align*}
%\int_{\RR^{2d}} &\nabla_{x_i} p_\varepsilon  (B_s^{1} -y)p_\varepsilon (B_u^{1} -z)\prod_{i=1}^d R_{H_i}(y_i,z_i)\rho(y_i)
%  \rho(z_i)dy_idz_i\\
%  &=\int_{\RR^{2}}\nabla_{x_i} p_\varepsilon  (B_s^{1,i} -y_i)p_\varepsilon (B_u^{1,i} -z_i)R_{H_i}(y_i,z_i)\rho(y_i)
%  \rho(z_i)dy_idz_i\\
%  &\quad \times\int_{\RR^{2d-2}} \rho(y)\rho(z)\prod_{i\neq j}^d p_\varepsilon  (B_s^{1,j} -y_j)p_\varepsilon (B_u^{1,j} -z_j) R_{H_j}(y_j,z_j)dy_jdz_j.
%\end{align*}
%where from our  results in last step, it is easy to see that
Analogously to \eqref{3.10star}, \eqref{e.3.15}, \eqref{e.3.16} and \eqref{e.3.17},  we can show the boundedness of other $A_{ji}^{\vare, \eta}$'s uniformly w.r.t. $\ep,\eta$.  So  we can apply the dominated convergence theorem below. In particular, we have
\begin{align*}
   \lim_{\eta,\varepsilon\rightarrow 0}
   A_{2,i}^{\vare, \eta}
 =\alpha_{H_0} \int_t^T |{s_1}-u|^{2H_0-2}
 \left[ \frac{\partial }{\partial x_i} q(x, y)\Big|_{x=B_{s_1}^1, y=B_u^2}+\frac{\partial }{\partial x_i} q(x, y)\Big|_{x=B_{s_1}^1, y=B_u^1}  \right]  du\,.
%  = \int_t^T \Phi({s_1},u,B_{s_1}^{1},B_u^{1})+\Phi({s_1},u,B_{s_1}^{1},B_u^{2})du\\
%  &=:A_2^i.
\end{align*}
\begin{align*}
   \lim_{\eta,\varepsilon\rightarrow 0}
   A_{3,i}^{\vare, \eta}
 =\alpha_{H_0} \int_t^T |{s_2}-u|^{2H_0-2}
 \left[ \frac{\partial }{\partial x_i} q(x, y)\Big|_{x=B_{s_1}^2, y=B_u^2}+\frac{\partial }{\partial x_i} q(x, y)\Big|_{x=B_{s_1}^2, y=B_u^1}  \right]  du\,.
%  = \int_t^T \Phi({s_1},u,B_{s_1}^{1},B_u^{1})+\Phi({s_1},u,B_{s_1}^{1},B_u^{2})du\\
%  &=:A_2^i.
\end{align*}
%
%\begin{align*}
%  \lim_{\eta,\varepsilon\rightarrow 0}\EE^W(Z_{2,i}^{\vare, \eta} Y^{\vare, \eta})&= \lim_{\eta,\varepsilon\rightarrow 0} \EE^W\bigg[\nabla_{x_i} {\dot{W}}_{\eta.\varepsilon}({s_2},B^2_{s_2})\cdot\int_t^T
%    {\dot{W}}_{\eta.\varepsilon}(u,B^1_u)+{\dot{W}}_{\eta.\varepsilon}(u,B^2_u)du\bigg]\\
%  &= \int_t^T \Phi({s_2},u,B_{s_2}^{2},B_u^{1})+\Phi({s_2},u,B_{s_2}^{2},B_u^{2})du\\
%  &=: A_3^i.
%\end{align*}
%And similar to Step 1, we get
%$$ \EE^W(Z_{1,i}^{\vare, \eta} Y^{\vare, \eta})= C \, A_2^i,\ \ \EE^W(Z_{2,i}^{\vare, \eta} Y^{\vare, \eta})= C \, A_3^i. $$
%$\bullet\ Step\ 3.$ Thirdly, we need to calculate
As for $A_4^{\vare, \eta}$, we have by definition of $Y^{\vare, \eta}$
\begin{eqnarray*}
A_4^{\vare, \eta}
&=& \EE^W\left[\int_t^T\int_t^T\Big({\dot{W}}_{\varepsilon,\eta}(u,B_u ^1 )+{\dot{W}}_{\varepsilon,\eta}(u,B_u^2  )\Big)
  \Big(\dot{W}_{\varepsilon,\eta}(v,B_v ^1 )+\dot{W}_{\varepsilon,\eta}(v,B_v^2  )\Big)dudv\right]\\
  &=& \sum_{i=1}^3 A_{4, i}^{\vare, \eta}\,,
\end{eqnarray*}
where
\begin{equation*}
\begin{split}
A_{41}^{\vare, \eta} :=&
\EE^W\left[\int_t^T\int_t^T {\dot{W}}_{\varepsilon,\eta}(u,B_u ^1 ) \dot{W}_{\varepsilon,\eta}(v,B_v ^1 )  dudv\right]\,, \\
  A_{42}^{\vare, \eta} := &2
\EE^W\left[\int_t^T\int_t^T {\dot{W}}_{\varepsilon,\eta}(u,B_u ^2 ) \dot{W}_{\varepsilon,\eta}(v,B_v ^1 )  dudv\right]\,,\\
A_{43}^{\vare, \eta} := &
\EE^W\left[\int_t^T\int_t^T {\dot{W}}_{\varepsilon,\eta}(u,B_u ^2 ) \dot{W}_{\varepsilon,\eta}(v,B_v ^2 )  dudv\right]\,.
\end{split}
\end{equation*}
%where from Proposition \ref{3.1} and \eqref{vtv} we have
Similar to the proof of Proposition \ref{3.1} we can show that $A_{4,i}^{\vare, \eta}$, $i=1,2,3$ can be  bounded by a bound analogous to \eqref{Eep}. Thus, we have
\begin{align}\label{A4limit}
  \lim_{\eta,\varepsilon\rightarrow 0}  A_4^{\vare, \eta}  =& \int_t^T\int_t^T \alpha_{H_0} |u-v|^{2H_0-2}
  \left[ q(B^1_u, B^1_v)+2 q(B^1_u, B^2_v)
  +q(B^2_u, B^2_v)\right] dudv  \,.
\end{align}
Combining  the above  with  \eqref{e.3.15}-\eqref{A4limit} enables us to apply the dominated convergence theorem to obtain
\begin{equation}
\begin{split}
\lim_{\vare, \eta\to 0} I_{1, i, t}^{\vare, \eta}
=& \al_{H_0} \int_t^T \int_t^T \EE
\big [ \xi(B^1) \xi(B^2) |s_2-s_1|^{2H_0-2}
 (\partial_{i,i}
q)( B_{s_1}^1,B^2_{s_2})\\
&\qquad\qquad   \Upupsilon(t, T, B^1, B^2)
  \big|  \cF_t^{B^1, B^2} \big ]\big |_{B^1=B^2=B} ds_1ds_2\,,
\end{split}
\end{equation}
where $\partial _{i,i}q(x,y)=\frac{\partial^2}{
\partial x_i\partial y_i} q(x,y)$ and
\begin{equation}
\Upsilon:=\exp\bigg\{ \int_t^T\int_t^T \alpha_{H_0} |u-v|^{2H_0-2}
  \big[ q(B^1_u, B^1_v)+2 q(B^1_u, B^2_v)
  +q(B^2_u, B^2_v)\big] dudv \bigg\} \,.
\end{equation}
In a similar way we can show the existence of the limits of $ I_{2, i, t}^{\vare, \eta}$,
 $I_{3, i, t}^{\vare, \eta}$, and we can further identity these limits.

Thus,  we can easily deduce that
$\displaystyle \EE  \int_0^T \big|Z_t^{\varepsilon, \eta}\big|^2 dt $ exists. In order to take the limit, it would be sufficient to show that, along a subsequence, $Z^{\ep,\eta}$ converges to some $Z\in \mathcal{M} ^2_{\mathbb{F} }(0,T;\RR^d)$. But this is guaranteed by the fact that $\displaystyle \EE  \int_0^T \big|Z_t^{\varepsilon, \eta}\big|^2 dt $ is bounded w.r.t. $\ep,\eta >0$. Indeed, as before we can also show that $Z^{ \varepsilon, \eta}$ is a Cauchy sequence in
$\mathcal{M} ^2_{\mathbb{F} }(0,T;\RR^d)$, whose   limit  is  denoted  by $Z=\{ Z_t,\,t\in[0,T]\}$. We can also  write $Z$ as \eqref{e.3.9} and \eqref{e.3.10} (whose justification is given through our above approximation).
\end{proof}
%Hilbert space bounded sequence has a convergence subsequence. Recall the part in Step 2, postgraduate grade 1, BSDE.

After we have found the limit  $Y$ (Theorem \ref{sup Y})
and the limit $Z$ (Theorem \ref{z}), we want to show that they are the solution to \eqref{e1}.   To this end we  shall   take limit in equation
 \eqref{e appro1}. Since we have shown the convergence of $Y_t^{\vare, \eta}$ and $Z_t^{\vare, \eta}$ as in Theorems \ref{sup Y} and Theorems \ref{z},  we only need to discuss the limit of $\int_t^T Y_s^{\vare, \eta} \dot W_{\vare, \eta}(s, B_s) ds$.
 Before discussing   this  limit we give the definition of a (Stratonovich) stochastic integral with respect to $\int_0^t F_s W(ds, B_s)$.

\begin{defn}\label{def stra}
  %Let $x\in \RR$ be a fixed value.
Let be given a random field $F=\{F_t,t\geq 0\}$ such that
  $\int_0^T |F_s| ds< \infty$
  almost surely, for all $T>0$. Then the Stratonovich integral $\int_t^T F_s {W}(ds,B_s)$ is defined as the following limit in probability
  if it exists (compared this with Proposition \ref{3.1}  when $F_s\equiv  1$):
  $$\int_t^T F_s{\dot{W}}_{\varepsilon,\eta}(s,B_s)ds.$$
\end{defn}

\begin{thm}\label{yw converge}
  Suppose $\sum_{i=1}^d (2H_i - \beta_i)<2$ and $\xi\in L^q(\Om)$ for $\displaystyle q>\frac{2}{2\underline{H} -1}$, where $\underline{H}=\min\{H_0,\ldots,H_d\}$. Then for any $t\in[0,T],$ we have
  $$\int_t^T Y_s^{\varepsilon,\eta}{\dot{W}}_{\varepsilon,\eta}(s,B_s)ds\rightarrow\int_t^T Y_s {W}(ds,B_s)$$
  in $L^2$ sense, as  $\varepsilon,\eta\downarrow 0$.
\end{thm}
\begin{proof}
  By \eqref{e appro1}, Lemma \ref{Y converge} and Theorem \ref{z}, we know
  $$\int_t^T Y_s^{\varepsilon,\eta} {\dot{W}}_{\varepsilon,\eta}(s,B_s)ds=Y_t^{\varepsilon,\eta}- \xi+\int_t^T Z_s^{\varepsilon,\eta}dB_s$$
  converges in $L^2$ sense to the random field
  $A_t := Y_t- \xi+\displaystyle \int_t^T Z_sdB_s$
  as $\ep, \eta$ tend to zero. Hence, if
 \begin{equation}
B_t^{\ep,\eta}:=\int_t^T \big(Y_s^{\varepsilon,\eta}- Y_s\big){\dot{W}}_{\varepsilon,\eta}(s,B_s)ds\rightarrow 0\label{e.3.25}
\end{equation}   in $L^2(\Omega)$, then we have
  $\displaystyle \int_t^T Y_s{\dot{W}}_{\varepsilon,\eta}(s,B_s)ds = \int_t^T Y_s^{\varepsilon,\eta}{\dot{W}}_{\varepsilon,\eta}(s,B_s)ds - B_t^{\ep,\eta}$
  will converge to $A$ in $L^2(\Omega)$. Previously, we have proved $A$ is well-defined, and then $Y_s$ will be Stratonovich integrable. Thus, by Definition \ref{def stra}, we directly have
  $$ \int_t^T Y_s {W}(ds,B_s) = \lim_{\ep,\eta \downarrow 0} \int_t^T Y_s{\dot{W}}_{\varepsilon,\eta}(s,B_s)ds = A,$$
  i.e., the equation \eqref{e1} is satisfied.

  In the remaining part of the proof, we shall show \eqref{e.3.25}.
  First we note that, recalling the definition of ${\dot{W}}_{\varepsilon,\eta}$ in \eqref{nsp} we have
  \begin{equation}\label{intWdot}
  \int_t^T {\dot{W}}_{\varepsilon,\eta}(s,B_s)ds =\int_{t}^T\int_{\RR^{d}} \int_{t}^T\varphi_{\eta}(s-r) p_{\ep} (B_s -y)\, W(dr,y) dy ds \,.
  \end{equation}
%  where $\tilde W$  is mean zero Gaussian random field whose    covariance is   given by
%\[
%\EE [ \tilde W(t,x) \tilde W(s,y)] = R_{H_0}(t,s) \prod_{i=1}^d R_{H_i}(x_i,y_i)\,.
%\]
Recall $\displaystyle F\cdot W(\phi) = \delta(F\phi)+\langle D^W F,\phi\rangle _\mathcal{H} $. Then, we obtaion
  \begin{equation*}
    \begin{split}
      \big(Y_s^{\varepsilon,\eta}- Y_s\big){\dot{W}}_{\varepsilon,\eta}(s,B_s) &= \big(Y_s^{\varepsilon,\eta}- Y_s\big)\int_{\RR^{d}}\int_t^s \varphi_{\eta}(s-r) p_{\ep} (B_s -y)
      W(dr,y)dy\\
      &=\int_t^s \int_{\RR^d} \big(Y_s^{\varepsilon,\eta}- Y_s\big) \varphi_{\eta}(s-r) p_{\ep} (B_s -y)\  W({\delta r,y})dy\\
      &\qquad +\big\langle D^W(Y_s^{\varepsilon,\eta}- Y_s),\varphi_{\eta}(s-\cdot ) p_{\ep} (B_s -\cdot)\big\rangle_\mathcal{H}.
    \end{split}
  \end{equation*}
  Hence,  by stochastic Fubini's Theorem, $B_t^{\ep,\eta}$ can be written as
  \begin{equation}
    \begin{split}
      B_t^{\ep,\eta} &= \int_{\RR^d}\int_t^T\int_t^T \big(Y_s^{\varepsilon,\eta}- Y_s\big) \varphi_{\eta}(s-r) p_{\ep} (B_s -y)ds
\, W({\delta r,y})dy\\
      &\qquad + \int_t^T \big\langle D^W(Y_s^{\varepsilon,\eta}- Y_s),  \varphi_{\eta}(s-\cdot ) p_{\ep} (B_s -\cdot)\big\rangle_\mathcal{H} ds\\
      &:= B_t^{\ep,\eta,1}+ B_t^{\ep,\eta,2}.
    \end{split}\label{e.3.27}
  \end{equation}
For the term $B_t^{\ep,\eta,1}$, we define
$$ \phi_{r,y}^{\ep,\eta}  = \int_t^T \big(Y_s^{\varepsilon,\eta}- Y_s\big) \varphi_{\eta}(s-r) p_{\ep} (B_s -y) ds,$$
and with the help of $L^2$ estimate for Skorokhod type stochastic integral, it yields:
\begin{equation}
  \EE\big[ \big(B_t^{\ep,\eta,1}\big)^2\big] \leq \EE\big[ \big\| \phi^{\ep,\eta} \big\Vert_\mathcal{H}^2\big] + \EE\big[ \big\| D ^W\phi^{\ep,\eta} \big\Vert_{\mathcal{H}
  \otimes\mathcal{H} }^2\big].\label{e.3.28}
\end{equation}
The above first term can be estimated as follows:
\begin{equation}\label{first term 1}
  \begin{split}
    \EE\big[\big\| \phi^{\ep,\eta} \big\Vert_\mathcal{H}^2\big]=\EE  \Big[\int_{[t,T]^2} &\big(Y_s^{\varepsilon,\eta}- Y_s\big)\big(Y_r^{\varepsilon,\eta}- Y_r\big)\\
    &\times\big\langle \varphi_{\eta}(s-\cdot )p_{\ep} (B_s -\cdot), \varphi_{\eta}(r-\cdot ) p_{\ep} (B_r -\cdot)\big\rangle_{\mathcal{H}}   dsdr\Big].
  \end{split}
\end{equation}
Recalling the definition in \eqref{inprod func true}, and combining with the proof in Proposition \ref{3.1} (refer to \eqref{estphieta} and \eqref{Eep}) we deduce that
\begin{equation}\label{inner prod}
  \begin{split}
&\big\langle \varphi_{\eta}(s-\cdot )p_{\ep}(B_s -\cdot), \varphi_{\eta}(r-\cdot ) p_{\ep} (B_r -\cdot)\big\rangle_{\mathcal{H}} \\
     =\,& \alpha_{H_0} \int_{[t,T]^2}\int_{\RR^{2d}}  \varphi_{\eta}(s-u) \varphi_{\eta}(r-v)p_{\ep}(B_s -y) p_{\ep} (B_r -z) \\
     &\qquad\qquad\times |u-v|^{2H_0-2} \rho(y)\rho(z)\prod_{i=1}^d R_{H_i} (y_i, z_i)dudvdydz\\
\leq \,&C |r-s|^{2H_0-2}\rho(B_s)\rho(B_r) \prod_{i=1}^d R_{H_i} (B_s^i, B_r^i).
  \end{split}
\end{equation}
Substituting  this into \eqref{first term 1} and with the help of \eqref{R split} we have
\begin{equation}\label{first term 2}
  \begin{split}
    \EE\big[\big\| \phi^{\ep,\eta} \big\Vert_\mathcal{H}^2\big]\leq & C\, \EE  \Big[\int_{[t,T]^2} \big(Y_s^{\varepsilon,\eta}- Y_s\big)\big(Y_r^{\varepsilon,\eta}- Y_r\big)|r-s|^{2H_0-2}  \rho(B_s)\rho(B_r) \prod_{i=1}^d R_{H_i} (B_s^i, B_r^i)\  dsdr\Big]\\
    %\leq& C \EE^B \EE^W \Big[\sup_{s\in[0,T]} \big(Y_s^{\varepsilon,\eta}- Y_s\big)^2 \int_{[t,T]^2} |r-s|^{2H_0-2}\\
    %&\quad \times\int_{\RR^{2d}}\prod_{i=1}^d R_{H_i} (y_i, z_i)\  p_{\ep} (B_s -y) \rho(y) p_{\ep} (B_r -z) \rho(z) dydzdsdr\Big]\\
    \leq& C\, \EE  \Big[\sup_{s\in[0,T]} \big(Y_s^{\varepsilon,\eta}- Y_s\big)^4 \Big]^{1/2}\\
    &\quad \times \EE  \Big[ \Big(\int_{[t,T]^2} |r-s|^{2H_0-2} \prod_{i=1}^d
    \Big[(1+|B_s^i|)^{2H_i-\beta_i}(1+|B_r^i|)^{2H_i-\beta_i}dsdr \Big)^2\Big]^{1/2}.
  \end{split}
\end{equation}
Thanks to Theorem \ref{sup Y}, Proposition \ref{exp int V} and the dominated convergence theorem, we see  that $\EE\big[\big\| \phi^{\ep,\eta} \big\Vert_\mathcal{H}^2\big]$ converges to zero as $\ep,\ \eta$
tend to zero. \\

Secondly, we have to deal with $\EE\big[ \big\| D^W \phi^{\ep,\eta} \big\Vert_{\mathcal{H} \otimes\mathcal{H} }^2\big]$, the second term in
\eqref{e.3.28}. By Malliavin calculus
 and \eqref{intWdot} we  have
\begin{equation}\label{DW Y ep}
  \begin{split}
    D^W Y_t^{\ep,\eta} &= \EE  \big[ \xi\, D^W \exp\big(\int_t^T {\dot{W}}_{\varepsilon,\eta}(s,B_s)ds\big)\big|\mathcal{F}_t \big]\\
&= \EE \big[  \xi\,\exp\big(\int_t^T {\dot{W}}_{\varepsilon,\eta}(s,B_s)ds\big)  \int_{t}^T\varphi_{\eta}(s-\cdot) p_{\ep} (B_s -\cdot) ds\big|\mathcal{F}_t \big]
\,.
  \end{split}
\end{equation}
We denote $\mathcal{F}^{B^1,B^2}_{t,s}= \sigma(B^1_u, B^2_v, 0\leq u\leq t,\,  0\leq v\leq s;W(t,x), t\geq 0,x\in\RR^d) $
%and $\mathcal{F}^{B^1,B^2}_{t}= \sigma(B^1_u, B^2_v, 0\leq u, v\leq t) $
the $\sigma$-algebra generated by $B^1,B^2$ and $W$. %We stress again that $\EE^W$ denotes by the expectation only with respect to our fractional Brownian motion $W$, and $\EE^{B^1,B^2}$ denotes by the expectation with respect to standard Brownian motion $B^1,B^2$, which is the same setting as in Theorem \ref{z}.
Recalling the definition \eqref{nsp} we know that, for random variable $W$ (namely for fixed $B$), $\int_t^T {\dot{W}}_{\varepsilon,\eta}(s,B_s)ds$ is Gaussians. Then Proposition \ref{3.1} and \eqref{inner prod} tell us
\begin{align}\label{second term 2}
\EE^W\langle& D^W Y_t^{\ep,\eta},D^W Y_t^{\ep',\eta'} \rangle_\mathcal{H} \notag\\
&= \EE^W\EE \Big[ \xi(B^1)\xi(B^2)\, \exp\big(\int_t^T {\dot{W}}_{\varepsilon,\eta}(s,B_s^1)ds + \int_t^T {\dot{W}}_{\varepsilon',\eta'}(s,B_s^2)ds \big)\notag\\
&\qquad \times \int_{[t,T]^2}  \langle \varphi_{\eta}(s-\cdot)p_{\ep} (B_s^1 -\cdot), \varphi_{\eta'}(r-\cdot)p_{\ep'} (B_r^2 -\cdot) \rangle_\mathcal{H} dsdr\Big|\mathcal{F}_t^{B^1,B^2}\Big]\Bigg|_{B^1=B^2=B}\notag\\
&\leq \alpha_{H_0}\EE \Bigg[ \xi(B^1)\xi(B^2)\, \EE^W\bigg[\exp\big(\int_t^T {\dot{W}}_{\varepsilon,\eta}(s,B_s^1)ds + \int_t^T {\dot{W}}_{\varepsilon',\eta'}(s,B_s^2)ds \big)\bigg] \\
 & \qquad\times \int_{[t,T]^2}  |s-r|^{2H_0-2} \rho(B_s^{1})\rho(B_r^{2})\prod_{i=1}^d R_{H_i}(B_s^{1,i},B_r^{2,i})dsdr\Big|\mathcal{F}_t^{B^1,B^2}\Bigg]\Bigg|_{B^1=B^2=B}\notag\\
 & =\alpha_{H_0}\EE \bigg[ \xi(B^1)\xi(B^2)\, \exp\Big(\sum_{j,k=1}^2 \int_t^T\int_t^T  |s-r|^{2H_0-2} \rho(B_s^{j})\rho(B_r^{k})\prod_{i=1}^d R_{H_i}(B_s^{j,i},B_r^{k,i})dsdr \Big)\notag\\
 & \qquad\times \int_{[t,T]^2}  |s-r|^{2H_0-2} \rho(B_s^{1})\rho(B_r^{2})\prod_{i=1}^d R_{H_i}(B_s^{1,i},B_r^{2,i})dsdr\Big|\mathcal{F}_t^{B^1,B^2}\bigg]\Bigg|_{B^1=B^2=B}.\notag
\end{align}
We have to prove the integrability of \eqref{second term 2}. Put $a,b$ be two positive constants such that $ 1/a+1/b =1$ and $2a<q$. With the help of  Proposition \ref{exp int V} and H\"{o}lder's inequality,
\begin{align}\label{second term 3}
&\EE \Big[ \xi(B^1)\xi(B^2)\, \exp\big(\sum_{j,k=1}^2 \int_t^T\int_t^T  |s-r|^{2H_0-2} \prod_{i=1}^d R_{H_i}(B_s^{j,i},B_r^{k,i})\rho(B_s^{j,i})\rho(B_r^{k,i})dsdr \big)\notag \\
 &\qquad\qquad \times \int_{[t,T]^2}  |s-r|^{2H_0-2} \rho(B_s^{1})\rho(B_r^{2})\prod_{i=1}^d R_{H_i}(B_s^{1,i},B_r^{2,i})dsdr\Big]\notag \\
\leq &\| \xi \|^2_{q} \Bigg(\EE \Big[\exp\big(2b\,\sum_{j,k=1}^2 \int_t^T\int_t^T  |s-r|^{2H_0-2} \prod_{i=1}^d R_{H_i}(B_s^{j,i},B_r^{k,i})
 \rho(B_s^{j,i})\rho(B_r^{k,i})dsdr \big)\Big]\Bigg)^{1/2b}\\
  &\qquad\qquad \times \Bigg(\EE \Big[ \Big|\int_{[t,T]^2}  |s-r|^{2H_0-2} \rho(B_s^{1})\rho(B_r^{2}) \prod_{i=1}^dR_{H_i}(B_s^{1,i},B_r^{2,i})dsdr\Big|^{2b} \Big]\Bigg)^{1/2b}\notag \\
<& \infty.\notag
\end{align}
%%%!!!!!!!!!!!!!!!!!!!!!!!!!!!!!!!!!!!!!!!!!!!!!!!!!!!!!!!!!!!\\
%%%!Important!\\
%%%We still have to prove  $Y_t^{\ep,\eta}$ also converges to $Y_t$ in $\mathbb{D}^{1,2}_W$.\\
%%%if we let $\int_{t}^T\varphi_{\eta}(s-\cdot) p_{\ep} (B_s -\cdot) ds:= \delta^{\ep,\eta}$
%%%\begin{equation}
%%%\begin{split}
%%%&\EE\left|\EE^B[\xi\exp(V_t)\delta-\xi\exp(V^{\ep,\eta}_t) \delta^{\ep,\eta}\Big|\mathcal{F}_t^B \right|^2_\mathcal{H}\\
%%%\leq &\EE\left|\EE^B[\xi\left(\exp(V_t)-\exp(V_t^{\ep,\eta})\right)\delta\Big|\mathcal{F}_t^B \right|^2_\mathcal{H}+\EE\left|\EE^B[\xi\exp(V_t^{\ep,\eta})\left(\delta^{\ep,\eta}-\delta\right)\Big|\mathcal{F}_t^B \right|^2_\mathcal{H}\\
%%%\leq & 2 \EE\left[\EE^{B^1,B^2}[\xi^1\xi^2\left(\exp(V_t^1)-\exp(V_t^{\ep,\eta,1})\right)\left(\exp(V_t^2)-\exp(V_t^{\ep,\eta,2})\right)
%%%\langle\delta^1,\delta^2\rangle_\mathcal{H}\Big|\mathcal{F}_t^{B^1,B^2}\Bigg|_{B^1=B^2=B} \right]\\
%%%&\qquad + \EE\left[\EE^{B^1,B^2}[\xi^1\xi^2\exp(V_t^{\ep,\eta,1}+V_t^{\ep,\eta,2})
%%%\langle\delta^1-\delta^{\ep,\eta,1},\delta^2-\delta^{\ep,\eta,2}\rangle_\mathcal{H}\Big|\mathcal{F}_t^{B^1,B^2}\Bigg|_{B^1=B^2=B} \right]
%%%\end{split}
%%%\end{equation}
%%%With the same idea as in \eqref{second term 3}, combining with \eqref{expVt}, \eqref{Third term 1} and the result of $B^{\ep,\eta,3,4}$ we have the left side of above formula tends to zero as $\ep,\eta\rightarrow 0$.
%%%
That is, we get $\EE \,\langle D^W Y_t^{\ep,\eta},D^W Y_t^{\ep',\eta'} \rangle_\mathcal{H}$ is integrable.
Hence, in a  similar idea as that shown in \eqref{first term 2}, we  obtain  $Y_t^{\ep,\eta}$ also converges to $Y_t$ in $\mathbb{D}^{1,2}_W$ as $\ep,\ \eta \downarrow 0$. Then putting $\ep=\ep',\ \eta=\eta'$,
$$\sup_{\ep,\eta\in(0,1]} \sup_{t\in[0,T]} \EE \| D^W Y_t^{\ep,\eta}\|^2_\mathcal{H} <\infty.$$
%Moreover, Theorem \ref{sup Y} tells us that
%\begin{equation}\label{sup DY}
 % \begin{split}
  %  \EE\sup_t \| &D^W Y_t^{\ep,\eta}\|^2_\mathcal{H}\\
   % &= \EE\Big[\sup_{t\in[0,T]} |Y_t^{\ep,\eta}|^2  \int_{[t,T]^2} \int_{\RR^{2d}} \langle \varphi_{\eta}(s-\cdot)
%\mathbf{1}_{[0,y]}(\cdot), \varphi_{\eta}(r-\cdot)\mathbf{1}_{[0,z]}(\cdot) \rangle_\mathcal{H}\\
%&\qquad \qquad p_{\ep} (B_s -y) p_{\ep} (B_r -z)\rho(y)\rho(z)dydzdsdr\Big]\\
%&\leq \big(\EE \sup_{t\in[0,T]}|Y_t^{\ep,\eta}|^4\big)^{1/2}  \EE \Big[ \Big(\int_{[0,T]^2} |r-s|^{2H_0-2} \\
%&\qquad\quad\times \prod_{i=1}^d \big[(1+|B_s^i|^{2H_i-\beta_i})+(1+|B_r^i|^{2H_i-\beta_i})\big]dsdr \Big)^2\Big]^{1/2}<\infty,
%  \end{split}
%\end{equation}
%thus, we can get
%$$\sup_{\ep,\eta\in[0,1]} \EE\Big[\sup_{t\in[0,T]} \| D^W Y_t^{\ep,\eta}\|^2_\mathcal{H} \Big] <\infty.$$
Hence, combining \eqref{inner prod}, \eqref{DW Y ep} and \eqref{second term 2}  we have
\begin{align}\label{second term 4}
\notag\EE\Big[ &\| D^W \phi^{\ep,\eta} \|_{\mathcal{H}\otimes\mathcal{H} }^2\Big] \\
\notag& = \EE\Big[ \int_{[t,T]^2}   \big\langle D^W (Y_s^{\ep,\eta}-Y_s),D^W (Y_r^{\ep,\eta}-Y_r)\big\rangle_\mathcal{H}\\
&\notag\qquad\qquad\times\big\langle \varphi_{\eta}(s-\cdot)p_{\ep} (B_s -\cdot),\varphi_{\eta}(r-\cdot)p_{\ep}
(B_r - \cdot) \big\rangle_\mathcal{H} dsdr\Big]\\
& =\alpha_{H_0} \EE\Big[\int_{[t,T]^2} \langle  D^W(Y_s^{\ep,\eta} - Y_s) ,D^W(Y_r^{\ep,\eta} - Y_r)\rangle_\mathcal{H}\\
\notag&\qquad \qquad\times |s-r|^{2H_0-2} \rho(B_s)\rho(B_r)\prod_{i=1}^d R_{H_i}(B_s^i,B_r^i) dsdr\Big] \\
\notag&\leq C\,\EE\Big[\int_{[t,T]^2} \langle  D^W(Y_s^{\ep,\eta} - Y_s) ,D^W(Y_r^{\ep,\eta} - Y_r)\rangle_\mathcal{H} \\
\notag&\qquad\qquad\times |s-r|^{2H_0-2}\prod_{i=1}^d \big[(1+|B_s^i|)^{2H_i-\beta_i}(1+|B_r^i|)^{2H_i-\beta_i}\big]dsdr \Big].
\end{align}
In  a similar method as in the proof of Theorem \ref{z}: there are two positive constants $p',q',\ 1/p'+1/q' =1$ such that $\displaystyle 1<p'< \frac{1}{2-2H_0}$ and $2q'< q$ for which we can deduce\\
\begin{equation}\label{second term 5}
  \begin{split}
&\EE\Big[\| D^W \phi^{\ep,\eta} \|_{\mathcal{H}\otimes\mathcal{H} }^2\Big] \\
   \leq\, &\, \bigg(\int_{[t,T]^2}\EE\Big[ |\langle  D^W(Y_s^{\ep,\eta} - Y_s) ,D^W(Y_r^{\ep,\eta} - Y_r)\rangle_\mathcal{H}|^
    {q'}\Big] dsdr\bigg) ^{1/{q'}}\\
    & \qquad\qquad \times\bigg(\int_{[t,T]^2}\EE\Big[|s-r|^{(2H_0-2)p'}\prod_{i=1}^d \big[(1+|B_s^i|)^{2H_i-\beta_i}(1+|B_r^i|)^{2H_i-\beta_i}\big]^{p'}\Big]dsdr\bigg) ^{1/{p'}},\\
  \end{split}
\end{equation}
where
\begin{equation}\label{second term 6}
  \begin{split}
    &\int_{[t,T]^2}\EE\Big[|s-r|^{(2H_0-2)p'}\prod_{i=1}^d \big[(1+|B_s^i|)^{2H_i-\beta_i}(1+|B_r^i|)^{2H_i-\beta_i}\big]^{p'}\Big]dsdr\\
\leq &\, C \int_{[t,T]^2}|s-r|^{(2H_0-2)p'} dsdr< \infty.
  \end{split}
\end{equation}
Now we only need to study the integrability of the first term on the right side of \eqref{second term 5}. Pick two constants $a,b>1,\ 1/a+1/b=1$ such that $a$ is sufficiently small to satisfy $2aq'<q$. With the help of proof in Proposition \ref{exp int V} and \eqref{second term 3}, we have that
\begin{align}\label{second term 7}
\notag&\left(\int_{[t,T]^2}\EE\Big[ \big|\big\langle  D^WY_s^{\ep,\eta}  ,D^WY_r^{\ep,\eta} \big\rangle_\mathcal{H}\big|^{q'} \Big]dsdr\right)^{1/q'}\\
\notag&\leq\| \xi\| ^{2}_{q}\Bigg(\int_{[t,T]^2}\EE \bigg[ \EE \Big[ \, \exp\big(bq'\sum_{j,k=1}^2 \int_s^T\int_r^T  |u-v|^{2H_0-2} \rho(B_u^{j})\rho(B_v^{k})\prod_{i=1}^d R_{H_i}(B_u^{j,i},B_v^{k,i})
dudv\big)\\
\notag& \qquad \times \Big(\int_s^T\int_r^T |u-v|^{2H_0-2} \rho(B_u^{1})\rho(B_v^{2})\prod_{i=1}^d R_{H_i}(B_u^{1,i},B_v^{2,i})dudv\Big)^{bq'}\big|\mathcal{F}_{s,r}^{B^1,B^2}\Big]
\Big|_{B^1=B^2=B}\bigg]^{1/b}dsdr\Bigg)^{1/q'}\\
&<\infty.
\end{align}
It yields that $\EE\Big[ \| D^W \phi^{\ep,\eta} \|_{\mathcal{H}\otimes\mathcal{H} }^2\Big] $ is integrable. Since we have deduced that $Y^{\ep,\eta}\rightarrow Y$ in $\mathbb{D}^{1,2}_W$, $\ep,\eta\rightarrow 0$, therefore $\EE\Big[ \| D^W \phi^{\ep,\eta} \|_{\mathcal{H}\otimes\mathcal{H} }^2\Big] $ converges to zero as $\ep,\ \eta$ tend to zero. Thus, we get $B_t^{\ep,\eta,1}$ defined in \eqref{e.3.27}
converges to zero in $L^2$ as $\ep,\ \eta$ tend to zero.\\

Now we are going to bound $B_t^{\ep,\eta,2}$.   We have
\begin{equation}\label{DY}
  \begin{split}
    D^W Y_s &= D^W \EE  \big[\xi \exp(\int_s^T {W}(dr,B_r)) \big|\mathcal{F}_s  \big] \\
    &= D^W \EE  \big[\xi \exp(\int_{\RR^d} \int_s^T \delta{(B_r-y)} {W}(dr,y)dy) \big|\mathcal{F}_s  \big]\\
      &=\EE  \big[\xi \exp\big(\int_s^T {W}(dr,B_r)\big) \delta{(B_\cdot-\cdot)}|\mathcal{F}_s  \big].
  \end{split}
\end{equation}
Thus, by \eqref{DW Y ep} and \eqref{DY} we have
\begin{equation}\label{Third term 1}
  \begin{split}
    B_t^{\ep,\eta,2} =& \int_t^T \left\langle D^W(Y_s^{\varepsilon,\eta}- Y_s), \varphi_{\eta}(s-\cdot ) p_{\ep} (B_s -\cdot)\right\rangle_\mathcal{H} ds\\
=&\int_t^T  \EE  \big[\xi\exp\big(\int_s^T {\dot{W}}_{\varepsilon,\eta}(r,B_r)dr\big)\\
&\qquad\qquad\times\int_s^T \big\langle\varphi_{\eta}(r-\cdot)p_{\ep} (B_r -\cdot), \varphi_{\eta}(s-\cdot )p_{\ep} (B_s -\cdot)\big\rangle_\mathcal{H} dr|\mathcal{F}_s  \big]ds \\
&\quad- \int_t^T  \EE  \big[\xi\exp\big(\int_s^T {W}(dr,B_r)\big)\big\langle \delta(B_\cdot -\cdot),\varphi_{\eta}(s-\cdot )p_{\ep} (B_s -\cdot)\big\rangle_\mathcal{H}
|\mathcal{F}_s  \big]ds \\
:=& B_t^{\ep,\eta,3}-B_t^{\ep,\eta,4} .
  \end{split}
\end{equation}
Note that,
\begin{equation*}
  \begin{split}
   &\big\langle \delta(B_\cdot -\cdot),\varphi_{\eta}(s-\cdot )p_{\ep} (B_s -\cdot)\big\rangle_\mathcal{H}\\
= \,&\, \int_{[s,T]^2} \int_{\RR^{2d}} |u-v|^{2H_0-2}\delta(B_u-y)\varphi_{\eta}(s-v )p_{\ep} (B_s -z)\rho(y)\rho(z)\prod_{i=1}^d R_{H_i}(y^i ,z^i)drdv\\
= \,&\, \int_{[s,T]^2} \int_{\RR^{d}} |u-v| ^{2H_0-2}\varphi_{\eta}(s-v )p_{\ep} (B_s -z)\rho(B_u)\rho(z)\prod_{i=1}^d R_{H_i}(B_u^i,y^i)dudvdz.
  \end{split}
\end{equation*}
Thus, by Fubini's Theorem and previous estimates, we have
\begin{equation}\label{Third term 2}
\begin{split}
    |B_t^{\ep,\eta,3}|\leq&  \int_t^T   \EE  \big[\xi\,\exp\big(\int_s^T {\dot{W}}_{\varepsilon,\eta}(r,B_r)dr\big)\int_s^T  |s-r|^{2H_0-2} \rho(B_s) \rho(B_r)\prod_{i=1}^d R_{H_i}(B_s^i,B_r^i) dr|\mathcal{F}_s \big]ds \\
   % \leq&  \EE^B \big[\xi\,\exp\big(\int_s^T {\dot{W}}_{\varepsilon,\eta}(r,B_r)dr\big) \int_t^T\int_s^T |s-r| ^{2H_0-2} \rho(B_s)\rho(B_r)\prod_{i=1}^d R_{H_i}(B_s^i,B_r^i)
%   drds |\mathcal{F}_s \big],
\end{split}
\end{equation}
and
\begin{align}\label{Third term 3}
\notag    |B_t^{\ep,\eta,4}|&= \int_t^T \EE  \big[\xi\exp\big(\int_s^T {W}(dr,B_r)\big)\int_{[s,T]^2} \int_{\RR^{d}} |u-v|^{2H_0-2}\varphi_{\eta} (s-v)\\
    &\qquad\qquad\qquad\qquad \times p_{\ep} (B_s -y)\rho(B_u)\rho(y)\prod_{i=1}^d  R_{H_i} (B_u^i, y_i)dudv dy
    |\mathcal{F}_s  \big]ds \\
\notag    &\leq \int_t^T\EE  \big[\xi\, \exp\big(\int_s^T {W}(dr,B_r)\big) \int_s^T  |s-v|^{2H_0-2}\rho(B_v)\rho(B_s) \prod_{i=1}^d  R_{H_i} (B_v^i, B_s^i)dv |\mathcal{F}_s \big]ds.
\end{align}
 Proposition \ref{exp int V} and dominated convergence theorem guarantee the integrability of these two expressions. Now, with the help of dominated
convergence theorem we get
$B_t^{\ep,\eta,3}$ and $B_t^{\ep,\eta,4}$ converge in $L^2$   to
$$\int_t^T\EE  \big[\xi\,\exp\big(\int_s^T {{W}}(dr,B_r)\big)\int_s^T |s-r| ^{2H_0-2} \rho(B_s)\rho(B_r)\prod_{i=1}^d R_{H_i}(B_s^i,B_r^i) dr |\mathcal{F}_s \big]ds$$
as $\ep,\ \eta$ tend to zero which also mean that $B_t^{\ep,\eta,2}$ converges in $L^2$   to zero as $\ep,\ \eta$ tend to zero.
\end{proof}
%\begin{remark}
%Note that, since we have proved the solution of BSDE \eqref{e1} has a explicit expression \eqref{e.1.6},  the uniqueness is then obvious.
%\end{remark}

\section{H\"{o}lder continuity of Y and Z}\label{s.4}
%\begin{thm}\label{conti}
%Suppose $\sum_{i=1}^d (2H_i - \beta_i)<2$ and $\xi\in D^{1,q}(\Om)$ for all $\displaystyle q >2$, $\EE |D_t\xi-D_s\xi|^q\le
%C|t-s|^{(2H_0+ \underline{H} -1)q/2}$. Let $(Y,Z)$ be the solution of \eqref{e1}. Then
%for any $a>1$ and for any $\vare>0$,  we have
%\begin{equation}
%\EE  |Y_t-Y_s|^a  \leq C_a \,|t-s|^{a/2},\ \quad  \EE  \big| Z_t  - Z_s \big|^2  \leq C_\vare |t-s|^{ 2H_0+ \underline{H} -1 -\vare     }.
%\end{equation}
%\end{thm}
\ \ \ \
Let the Assumption (2) in Theorem \ref{th1} be satisfied. Now we can prove the H\"older continuity of $Y$ and $Z$.

\begin{proof}First we prove the H\"older continuity of $Y$. Recall $\displaystyle q>\frac{2}{2\underline{H} -1}$, where $\underline{H}=\min\{H_0,\ldots,H_d\}$. Thus for all $a\in(1,q)$, we have
\begin{equation}\label{contiY}
\begin{split}
\EE  \big|Y_t-Y_s\big|^a &= \EE\, \Big|\EE \big[\xi\exp(V_t)\big|\mathcal{F}_t  \big]-\EE \big[\xi\exp(V_s)\big|\mathcal{F}_s  \big]\Big|^a\\
&\leq 2\bigg(\EE \Big|\EE \big[\xi\exp(V_t)\big|\mathcal{F}_t \big]-\EE^B\big[\xi\exp(V_s)\big|\mathcal{F}_t  \big]\Big|^a \\
&\quad +\EE  \Big|\EE \big[\xi\exp(V_s)\big|\mathcal{F}_t  \big]-\EE \big[\xi\exp(V_s)\big|\mathcal{F}_s  \big]\Big|^a\bigg)\\
&=: 2\bigg(I_1 +I_2\bigg).
\end{split}
\end{equation}
For $I_1$, one can use Jensen's inequality and the exponential integrability of Proposition \ref{exp int V} to get, for two positive constants $p',q'$ satisfying $1/p' +1/q' =1$
and $aq'<q$,
\begin{align}\label{perI1}
\notag I_1&=\EE\Big|\EE \big[\xi\exp(V_t)\big|\mathcal{F}_t  \big]-\EE \big[\xi\exp(V_s)\big|\mathcal{F}_t  \big]\Big|^a\\
&\leq\EE\Big|\EE \left[\xi\big(|V_t-V_s|\exp\left(\max\{V_t,V_s\} \right)\big)\big|\mathcal{F}_t \right]\Big|^a\\
&\leq \EE\left[\left(\EE \left[ \xi^{q'} \exp(q'\max\{V_t,V_s \} ) \big|\mathcal{F}_t \right]\right)^{a/q'} \left(\EE  \big[ \big|V_t-V_s\big|^{p'}\big|\mathcal{F}_t \big]\right)^{\frac{a}{p'}}\right]\notag\\
%&\leq \|\xi\|_{aq'}^{1/q'} \Big(\EE  \big[ \big|V_t-V_s\big|^{ap'} \exp\big(ap'\,\max\{V_t,V_s\}\big)\big]\Big)^{1/p'}\\
\notag&\leq C \bigg(\EE  \big[ \big|V_t-V_s\big|^{ap'}\big]\bigg)^{\frac{1}{p'}}.
\end{align}
By \eqref{vt}, \eqref{vtv} and the equivalence between the $L^2$-norm and the
$L^p$-norm for a Gaussian random variable, it yields that
\begin{align}\label{YI1}
\notag I_1&\leq\bigg(\EE \big[ \big|V_t-V_s\big|^{ap'}\big]\bigg)^{\frac{1}{p'}} = \bigg(\EE  \left[ \Big|\int_s^t W(dr,B_r)\Big|^{ap'}\right]\bigg)^{\frac{1}{p'}}\leq C\bigg(\EE \Big(\EE^W \big|\int_s^t W(dr,B_r)\big|^{2}\Big)^{ap'/2}\bigg)^{\frac{1}{p'}}\\
&\leq C\bigg(\EE \Big(\int_s^t\int_s^t \alpha_{H_0} |u-v|^{2H_0 -2} \rho(B_u)\rho(B_v)\prod^d_{i=1} R_{H_i}(B_u^i,B_v^i)dudv\Big)^{ap'/2}  \bigg)^{\frac{1}{p'}}\\
&\notag \leq C\bigg(\Big(\int_s^t\int_s^t  |u-v|^{(2H_0 -2)m}dudv\Big)^{\frac{ap'}{2m}}\, \EE \Big(\int_s^t\int_s^t\big(\rho(B_u)\rho(B_v)\prod^d_{i=1} R_{H_i}(B_u^i,B_v^i)\big)^n dudv\Big)^{\frac{ap'}{2n} } \bigg)^{\frac{1}{p'}}\\
\notag&{\leq C\, |t-s|^{aH_0-\ep}}\,,
\end{align}
where $\ep$ is an arbitrary positive constant, $n,m>1$ such that $ \frac{1}{n} + \frac{1}{m} =1$.\\

For $I_2$, denote by $\psi_t=\exp\big(\int_0^t {{W}}(ds,B_s)\big).$ Proposition \ref{exp int V} tells us that, $\xi \psi_T$ is $L^q(\Omega)$-integrable for $\displaystyle q>\frac{2}{2\underline{H} -1}$. Moreover, Clark-Ocone formula implies that,
$$\xi \psi_T = \EE^B [\xi \psi_T] + \int_0^T f_r dB_r,$$
where
\begin{equation}\label{fr}
f_r = \EE [D_r^B (\xi \psi_T)  |\mathcal{F}_r^B]= \EE [\psi_T D_r^B (\xi) |\mathcal{F}_r ] +\EE [\xi D_r^B ( \psi_T)  |\mathcal{F}_r ].
\end{equation}
Thus, from  the  Burkholder-Davis-Gundy inequality and the fact that $a>2$   we  deduce that
\begin{align}\label{YI2}
\notag I_2 &= \EE  \bigg|\psi_s^{-1}\Big( \EE \big[\xi\psi_T\big|\mathcal{F}_t \big]-\EE \big[\xi\psi_T\big|\mathcal{F}_s \big] \Big)\bigg|^a\\
&= \EE  \bigg|\psi_s^{-1} \int_s^t f_r dB_r \bigg|^a\leq \left(\EE  \left[\psi_{s}^{-2a}\right]\right)^{1/2} \left(\EE \left[\int_s^t \big|f_r\big|^2 dr \right]^a\right)^{1/2}\\
\notag&\leq C\,\left( \EE \left[\int_s^t \big|f_r\big|^{2} dr \right]^a\right)^{1/2}.
\end{align}
Taking \eqref{fr} into above formula   yields that
\begin{align*}
&\mathbb{E} \left[\int_s^t \big|f_r\big|^2 dr \right]^{a}\\
&\leq C \left(\EE \int_s^t\left|\EE \left[\psi_T D_r^B(\xi)\big|\mathcal{F}_r\right]\right|^2 dr\right)^{a}
+C\left(\EE \int_s^t\left|\EE \left[\xi D_r^B(\psi_T)\big|\mathcal{F}_r\right]\right|^2 dr\right)^{a/2}\\
&\leq C \left(\int_s^t\EE  \left(D_r^B\xi\right)^{2q'} dr\right)^{a/q'}  \left(\int_s^t \EE \left(\psi_T\right)^{2p'} dr\right)^{a/p'}\\
 &\qquad\qquad+C \left(\int_s^t\EE  [\xi^{2q'}] dr\right)^{a/q'}  \left(\int_s^t \EE \left(D_r^B\psi_T\right)^{2p'} dr\right)^{a/p'}\\
&\leq C |t-s|^{a/q'}\left( \left(\int_s^t \EE \left(\EE^W\left(\psi_T\right)^{2}\right)^{p'} dr\right)^{a/p'} + \left(\int_s^t \EE \left(\EE^W\left(D_r^B\psi_T\right)^{2}\right)^{p'} dr\right)^{a/p'}\right),
\end{align*}
where $C$ is a constant only depends on $p',q',\|D_r^B\xi\|_{L_q}^2$ and $ \|\xi\|_{L_q}^2$. We recall $\psi_s$ and $D_r^B\psi_s$ are centralized Gaussian processes given $B$. Moreover, $\EE^W\left(D_r^B\psi_T\right)^2 = \EE^W\left(D_r^B \exp\big(\int_0^t {{W}}(ds,B_s)\big)\right)^2$ can be treated in a similar way as we did in the proof of Theorem \ref{z}. By Proposition \ref{3.1} and Theorem \ref{z}, we can directly obtain the boundedness of $ \EE \left(\EE^W\left(\psi_T\right)^{2}\right)^{p'}$ and $ \EE \left(\EE^W\left(D_r^B\psi_T\right)^{2}\right)^{p'} $. Thus, we   deduce
$$I_2\leq \left(\mathbb{E} \left[\int_s^t \big|f_r\big|^2 dr \right]^a\right)^{1/2} \leq C \left|t-s\right|^{a/2}.$$
Because we assume that $\displaystyle H>\frac{1}{2}$, the  H\"{o}lder continuous coefficient can only be $\displaystyle\frac{1}{2}$.

Next we have to consider the H\"older continuity of $Z$. Recall (\ref{e.3.10})  for  the expression of $Z$:
\begin{align*}
Z_t &= D_t^B Y_t =\EE  \left[e^{\int_t^T W(d\tau, B_\tau)d\tau} D_t^B \xi + \xi\,\exp\big(\int_t^T {W}(du,B_u)du\big)\int_t^T\nabla_x {W}(ds, B_s)\,\Big|\, \mathcal{F}_t \right]\\
&=\EE  \left[e^{V_t} D_t^B \xi + \xi\,\exp\big(V_t \big) \nabla_x V_t\,\Big|\, \mathcal{F}_t \right]=:    Z_t^1 +Z_t^2\,,
\end{align*}
where we recall the definition of \eqref{vt} for $V_t$ and where we denote
$\nabla_x V_t=
\int_t^T\nabla_x {W}(ds, B_s)$.   $Z^1$ is easy to deal with. In fact, similar to the way to treating    \eqref{contiY}, \eqref{YI1}, \eqref{YI2}, and by the assumption that $D^B\xi\in L^q(\Om)$ and  $\EE |D_t\xi-D_s\xi|^q\le C|t-s|^{\kappa q/2}$ for some $\displaystyle \kappa >0$,  we see
$$\EE\big|Z_t^1 - Z_s^1\big|^2 \leq  C|t-s|^{\kappa\wedge1}.$$
We shall focus on  $Z^2$.
%(We assume $\xi$ =1. Don't forget to add it ).
\begin{equation}
\begin{split}
\EE  \big| Z_t^2 - Z_s^2\big|^2 &= \EE  \left( \EE \left[\xi\exp (V_t) \nabla_xV_t\,\big|\,\mathcal{F}_t \right] - \EE \left[\xi\exp (V_s) \nabla_xV_s\,\big|\,\mathcal{F}_s \right]\right)^2 \\
&\leq 2 \EE  \left( \EE \left[\xi\exp (V_t) \nabla_xV_t\,\big|\,\mathcal{F}_t \right] - \EE \left[\xi\exp (V_s) \nabla_xV_s\,\big|\,\mathcal{F}_t \right]\right)^2\\
&\qquad + 2\EE  \left( \EE \left[\xi\exp (V_s) \nabla_xV_s\,\big|\,\mathcal{F}_t \right] - \EE \left[\xi\exp (V_s) \nabla_xV_s\,\big|\,\mathcal{F}_s \right]\right)^2\\
&:= 2(I_1 +I_2).
\end{split}
\end{equation}
For $I_1$, with the help of Jensen's inequality we have
\begin{align*}
I_1&\leq \EE  \left[\left| \xi\exp (V_t) \nabla_xV_t - \xi \exp (V_s) \nabla_xV_s\right|^2\right]\\
&\leq 2 \EE  \left[\left|\xi\left( \exp (V_t)-\exp (V_s)\right) \nabla_xV_t\right|^2 +\left|\xi \exp (V_s)( \nabla_xV_t-\nabla_xV_s)\right|^2\right]\\
&\leq 2 \EE  \left[\left|\xi \exp(\max\{V_t,V_s\}) \nabla_xV_t(V_t-V_s)\right|^2\right] +2\EE \left[\left|\xi \exp (V_s)( \nabla_xV_t-\nabla_xV_s)\right|^2\right]\\
& :=2( I_{1,1} + I_{1,2}).
\end{align*}
We can find two constant $a,b$ such that $1/a +1/b =1$,  $1<a <\frac{ 1}{2-2\underline{H} }$ and $2b<q$.
%, $2b<q$.
Then we have
\begin{align}\label{I11}
\notag I_{1,1} &\leq \left(\EE  \left[ \left( \nabla_xV_t\right)^{2a}\right]\right)^{1/a} \left(\EE\,  \xi^{2b}\left(V_t-V_s\right)^{2b}\right)^{1/b}\\
&\leq C\left(\EE \ \xi^{2b}\left(V_t-V_s\right)^{2b}\right)^{1/b}\\
\notag&\qquad\qquad \left(\sum_{i=1}^d \EE^B\left( \int_t^T\int_t^T |u-v| ^{2H_0 -2} |B_u^i-B_v^i|^{2H_i -2} \prod_{j\neq i}^d R_{H_j}(B_u^j-B_v^j) \rho(B_u^j)\rho(B_v^j)dudv\right)^a\right)^{1/a}\\
\notag&\leq C|t-s|^{2H_0-\ep}.
\end{align}
As for  $I_{1,2}$, we deduce similarly that
\begin{align}\label{I12}
\notag I_{1,2} &\leq \EE  \left[|\xi|^{2b} \exp(2b V_t)\right]^{1/b}\left(\EE   \left[ \nabla_xV_t-\nabla_xV_s\right]^{2a} \right)^{1/a}\\
 &\notag\le C_a\EE  \left[|\xi|^{2b} \exp(2b V_t)\right]^{1/b}  \left(\EE   \left[\EE ^W \left\{ \int_s^t\int_s^t \left(\nabla_x W(du,B_u)\right)^T\left(\nabla_x W(dv,B_v)  \right)\right\}   \right]^a\right)^{1/a}\\
&\leq C   \left[\EE  \left(\left|\sum_{i=1}^d \int_s^t\int_s^t|u-v|^{2H_0-2} |B_u^i-B_v^i|^{2H_i-2}\prod_{j\neq i}^d R_{H_j}( B_u^j,B_v^j)\rho(B_u^j)\rho(B_u^j)dudv\right|^a\right)\right]^{1/a}\\
&\notag\leq C \sum_{i=1}^d \left(\int_s^t\int_s^t |u-v|^{(2H_0-2)a+(H_i-1)a}   dudv\right)^{1/a}\\
&\notag\leq C |t-s|^{2H_0  + \underline{H}-1-\ep}.
\end{align}
As  $I_2$, the
Clark-Ocone formula  yields
%that we have the following expression of $\exp(V_s) \nabla_x V_s$:
\begin{equation}\label{I2}
\xi\exp(V_s) \nabla_x V_s = \EE^B\left[\xi\exp(V_s) \nabla_x V_s\right] + \int_s^T \EE \left[D_r^B\left(\xi\exp(V_s) \nabla_x V_s\right)\Big|\mathcal{F}_r \right] dB_r.
\end{equation}
Thus we have
\begin{equation}\label{I2 express}
\begin{split}
I_2 &= \EE \left|\int_s^t \EE  \left[D_r^B \left(\xi\exp(V_s) \nabla_x V_s\right)\Big|\mathcal{F}_r \right]dB_r\right|^2= \EE  \int_s^t\left(\EE  \left[D_r^B\left(\xi\exp(V_s) \nabla_x V_s\right)\Big|\mathcal{F}_r \right]\right)^2dr\\
& = \EE  \int_s^t\left(\EE  \left[\xi\nabla_x V_s D_r^B\left(\exp(V_s)\right)\Big|\mathcal{F}_r \right]\right)^2dr+ \EE  \int_s^t\left(\EE \left[\xi\exp(V_s)D_r^B\left(\nabla_x V_s\right)\Big|\mathcal{F}_r \right]\right)^2dr\\
&\qquad\qquad +\EE  \int_s^t\left(\EE  \left[\exp(V_s)\nabla_x V_s \,D_r^B\xi\Big|\mathcal{F}_r \right]\right)^2dr \\
& =: I_{2,1} + \int_s^t I_{2,2}dr +I_{2,3}.
\end{split}
\end{equation}
The integrability inside the integral of  $I_{2,3}$ is obvious
 due to \eqref{I11}. For $I_{2,1}$ we have
\begin{align*}
\EE  &\int_s^t\left(\EE  \left[D_r^B\left(\exp(V_s)\right)\xi\,\nabla_x V_s\Big|\mathcal{F}_r \right]\right)^2dr=\EE   \int_s^t\left(\EE  \left[ \xi \exp(V_s)\,\left(\nabla_x V_s\right)^2 \Big|\mathcal{F}_r \right]\right)^2dr\\
&\qquad\leq \int_s^t\left(\EE [ \xi^b\,\exp(b V_s)]\right)^{2/b} \left(\EE  (\nabla_x V_s)^{2a}\right)^{2/a}   ds\le C|t-s|\,.
\end{align*}
Finally, we deal with $I_{2,2}$.  We shall use the technique as in \eqref{eqn.z_compute}.   Notice  that,
$$D_r^B\left(\nabla_x V_s\right)=D_r^B\left(\int_s^T \nabla_xW(du,B_u) \mathrm{1}_{[0,u]}(r)\right)=\int_{s\vee r} ^T \nabla^2_x W(du,B_u)  \,. $$
%is a $d\times d$-dimensional matrix.
  We have analogously  to \eqref{eqn.z_compute}
\begin{align}\label{I22}
%\notag\EE^B&\EE^W \left(\EE^B \left[\xi\,\exp(V_s)D_r^B\left(\nabla_x V_s\right)\Big|\mathcal{F}_r^B\right]\right)^2\\
I_{2,2} &= \EE  \Bigg(\EE  \bigg[\xi(B^1)\,\xi(B^2)\exp\left(\sum_{j,k=1}^2 \frac{\alpha_{H_0}}{2} \int_s^T \int_s^T |u-v|^{2H_0-2}R_{H_i}(B_u^j,B_v^k)\rho(B_u^j)\rho(B_v^k)dudv\right)\\
\notag &\qquad\times \left(\int_{s\vee r} ^T\int_{s\vee r}^T  \tr \left[ \left( \nabla_x^2 W(du,B_u^1)\right)^T \nabla_x^2 W(dv,B_v^2)\right]  dudv\right)\bigg|\mathcal{F}_r^{B^1,B^2}\bigg]\Bigg|_{ B^1=B^2=B}\Bigg).
\end{align}
Using the H\"older inequality, we have
\begin{align}\label{I22}
I_{2,2} &= \Bigg[\EE  \Bigg(\EE  \bigg[\bigg|\xi(B^1)\,\xi(B^2)\exp\bigg(\sum_{j,k=1}^2 \frac{\alpha_{H_0}}{2} \int_{s\vee r}^T \int_{s\vee r}^T |u-v|^{2H_0-2}\nonumber \\
&\qquad\qquad  R_{H_i}(B_u^j,B_v^k)\rho(B_u^j)\rho(B_v^k)dudv\bigg)\bigg|^b \bigg|\mathcal{F}_r^{B^1,B^2}\bigg]\Bigg|_{ B^1=B^2=B}\Bigg)\Bigg]^{1/b} \\
\notag &\qquad\times
 \Bigg[\EE  \Bigg(\EE \Bigg[\bigg|\int_{s\vee r} ^T\int_{s\vee r}^T  \tr \bigg[ \left( \nabla_x^2 W(du,B_u^1)\right)^T
  \nabla_x^2 W(dv,B_v^2)\bigg]  dudv\bigg|^a\bigg|\mathcal{F}_r^{B^1,B^2}\Bigg]\Bigg|_{ B^1=B^2=B}\Bigg)\Bigg]^{1/a}\nonumber\\
 &\le  C  I_{2,2,1}^{1/a} \,, \nonumber
\end{align}
where
\[
I_{2,2,1}=\EE  \Bigg(\EE  \Bigg[\left|\int_{s\vee r} ^T\int_{s\vee r}^T  \tr \left[ \left( \nabla_x^2 W(du,B_u^1)\right)^T \nabla_x^2 W(dv,B_v^2)\right]  dudv\right|^a\bigg|\mathcal{F}_r^{B^1,B^2}\Bigg]\Bigg|_{ B^1=B^2=B}\Bigg)\,.
\]
We shall  consider the term that
contains  $\displaystyle J=\int_{s\vee r}^T\int_{s\vee r}^T  \frac{\partial^2}{\partial x_i^2}W(du,B_u^1)\frac{\partial^2}{\partial x_i^2}W(dv,B_v^2)   $
(denote the corresponding term by $J_i$)  since the other terms can be treated in similar way.   When $r\ge s$,  we have   for any $a>1$,
%, because in fact, the integrability of other terms can be deduced as following:
%\begin{align*}
%&\EE \int_s^T\int_s^T \sum_{j>i}^d\sum_{i=1}^d \Big|\frac{\partial^2}{\partial x_i\partial x_j}W(du,B_u^1)\frac{\partial^2}{\partial x_i\partial x_j}W(dv,B_v^2)\Big|\mathrm{1}_{[0,u]}(r)\mathrm{1}_{[0,v]}(r)\\
%\leq & \sum_{j>i}^d\sum_{i=1}^d \EE^B\left[\int_r^T\int_r^T |u-v|^{2H_0-2} |B_u^{1,i}-B_v^{2,i}|^{2H_i-2}|B_u^{1,j}-B_v^{2,j}|^{2H_j-2}\prod^d_{k\neq i,j} R_{H_k}(B^{1,k}_u,B^{2,k}_v)\rho(B^{1,k}_u)\rho(B^{2,k}_v)dudv\right]\\
%<& \infty.
%\end{align*}
%So the item  $\displaystyle \int_s^T\int_s^T \sum_{i=1}^d \Big|\frac{\partial^2}{\partial x_i^2}W(du,B_u^1)\frac{\partial^2}{\partial x_i^2}W(dv,B_v^2)\Big|\mathrm{1}_{[0,u]}(r)\mathrm{1}_{[0,v]}(r)$ can be treated like:
\begin{align*}
J_i=&\EE \Bigg(\EE   \left[\left|\int_r^T\int_r^T  \frac{\partial^2}{\partial x_i^2}W(du,B_u^1)\frac{\partial^2}{\partial x_i^2}W(dv,B_v^2)  \right|^a \mathcal{F}_r^{B^1,B^2}\bigg]\right|_{B^1=B^2=B}\Bigg)\\
\le &C_a \EE   \Bigg(\EE   \left[ \left(\EE^W  \left|\int_r^T\int_r^T  \frac{\partial^2}{\partial x_i^2}W(du,B_u^1)\frac{\partial^2}{\partial x_i^2}W(dv,B_v^2)  \right|^2\right) ^{a/2} \Bigg| \mathcal{F}_r^{B^1,B^2}\right]\Bigg|_{B^1=B^2=B}\Bigg)\\
\le &C_a\EE  \Bigg( \EE \bigg[
\bigg( \int_r^T\int_r^T |u-v|^{2H_0-2} |B_u^{1,i}- B_v^{2,i}|^{2H_i-4} \\
& \qquad\qquad  \rho(B_u^1) \rho(B_v^2)
\prod_{j\not=i} |R_{H_j}(B_u^{1,j}, B_v^{2,j})|dudv\bigg)^{a/2} \bigg|\mathcal{F}_r^{B^1,B^2}\bigg]\Bigg|_{B^1=B^2=B}\Bigg) +C_a\,,
\end{align*}
where in the above first inequality, we used the hypercontractivity for $\EE^W$ and in the above last inequality, there are terms such as the derivatives with respect to $\partial _{x_i}^2 \rho$ and $\partial _{x_i}  \rho \partial _{x_i} R_{H_i}$ which are easy to  be bounded.  By using   H\"older's inequality again, the above expectation is bounded by a multiple of $1/a'$ power of   (for any $a'>1$)
\begin{align*}
 &\EE  \left( \EE \left[    \left( \int_r^T\int_r^T |u-v|^{2H_0-2} |B_u^{1,i}- B_v^{2,i}|^{2H_i-4} dudv\right)^{aa'/2} \bigg|\mathcal{F}_r^{B^1,B^2}\right]\Bigg|_{B^1=B^2=B}\right)\\
=&\EE  \Bigg( \EE \bigg[ \bigg|\int_r^T\int_r^T |u-v|^{2H_0-2} |(B_u^{1,i}-B_r^{1,i})- (B_v^{2,i}-B_r^{2,i})\\
& \qquad \qquad \qquad\qquad\qquad\qquad+B_r^{1,i}-B_r^{2,i}|^{2H_i-4} dudv\bigg|^{aa'/2}\bigg|\mathcal{F}_r^{B^1,B^2}\bigg]\Bigg|_{B^1=B^2=B}\Bigg)\\
=&\EE  \Bigg( \EE^{X,Y}\bigg[ \int_r^T\int_r^T |u-v|^{2H_0-2} |\sqrt{u-r}X-\sqrt{v-r}Y\\
&\qquad \qquad\qquad\qquad \qquad\qquad+ B_r^{1,i}-B_r^{2,i}|^{2H_0-4} dudv\bigg]^{aa'/2}\Bigg|_{B^1=B^2=B}\Bigg)\,,
\end{align*}
where $X$ and $Y$ are two independent standard
Gaussians.
The above expectation in $X$ and $Y$ are bounded by (denoting $Z=B_r^{1,i}-B_r^{2,i}$
and choosing $aa'/2<1$)
\begin{align*}
&\EE^{X,Y}\bigg[ \int_r^T\int_r^T |u-v|^{2H_0-2} |\sqrt{u-r}X-\sqrt{v-r}Y  + Z |^{2H_0-4} dudv\bigg]^{aa'/2}\\
 & \le  \left[\int_r^T\int_r^T |u-v|^{2H_0-2} \int_{\RR^2}\frac{1}{\sqrt{2\pi}} e^{-x^2/2}\frac{1}{\sqrt{2\pi}} e^{-y^2/2} |\sqrt{u-r}x-\sqrt{v-r}y+Z|^{2H_i-4} dxdy  dudv \right]^{aa'/2}\\
 & \le \Bigg[\int_r^T\int_r^T |u-v|^{2H_0-2}\frac{1}{\sqrt{(u-r)(v-r)}}\\
&\qquad \qquad \left[\int_{\RR^2}\frac{|xy|}{4}\frac{1}{\sqrt{2\pi}} e^{-x^2/2}\frac{1}{\sqrt{2\pi}} e^{-y^2/2} |\sqrt{u-r}x-\sqrt{v-r}y+Z|^{2H_i-2} dxdy\right] dudv\Bigg] ^{aa'/2} \\
& \le \Bigg[\int_r^T\int_r^T |u-v|^{2H_0-2}\frac{1}{\sqrt{(u-r)(v-r)}}\\
&\qquad \qquad \left[\int_{\RR^2} \frac{1}{\sqrt{2\pi}} e^{-x^2/4}\frac{1}{\sqrt{2\pi}} e^{-y^2/4} |\sqrt{u-r}x-\sqrt{v-r}y+Z|^{2H_i-2} dxdy\right] dudv\Bigg] ^{aa'/2} \\
 &= C\, \left[ \int_r^T\int_r^T |u-v|^{2H_0-2} \frac{1}{\sqrt{(u-r)(v-r)}} \EE^{X,Y}|\sqrt{u-r}X-\sqrt{v-r}Y+Z|^{2H_i-2} dudv\right]^{aa'/2}  \\
 & \le C\,   \left[ \int_r^T\int_r^T |u-v|^{2H_0-2} |u-r|^{-1/2}|v-r|^{-1/2}|u-r+v-r|^{2H_i-2} dudv\right]^{aa'/2}  \\
 & \le C\,   \left[ \int_r^T\int_r^T |u-v|^{2H_0-2} |u-r|^{-1/2}|v-r|^{-1/2}|u-r|^{H_i-1}|v-r|^{ H_i-1} dudv\right]^{aa'/2}  \\
 & <\infty\,,
\end{align*}
where the third last inequality follows from
Lemma A.1 of \cite{HNS2011} and the last inequality holds true since $H_i>1/2$.
%So we recall those two constant $a,b$ such that $1/a +1/b =1$ and satisfy $1<a <\frac{-1}{2\underline{H}-2}$, $2b<q$. Then H\"{o}lder's inequality helps us deal with \eqref{I22} that
%\begin{align*}
%&\EE  \left(\EE^B \left[\xi\,\exp(V_s)D_r^B\left(\nabla_x V_s\right)\Big|\mathcal{F}_r^B\right]\right)^2\\
%\leq & \EE^B \Bigg(\EE^{B^1,B^2} \bigg[\xi^b(B^1)\xi^b(B^2)\exp\left(\sum_{j,k=1}^2 \frac{b\,\alpha_{H_0}}{2} \int_s^T \int_s^T |u-v|^{2H_0-2}R_{H_i}(B_u^j,B_v^k)\rho(B_u^j)\rho(B_v^k)dudv\right)\bigg|\mathcal{F}_r^{B^1,B^2}\bigg]^{1/b}\Bigg|_{B^1=B^2=B}\\
%&\quad \times\EE^{B^1,B^2} \bigg[\left(\int_s^T\int_s^T \left( \triangle_x W(du,B_u^1)\right)^T \triangle_x W(dv,B_v^2)\mathrm{1}_{[0,u]}(r)\mathrm{1}_{[0,v]}(r)dudv\right)\bigg|\mathcal{F}_r^{B^1,B^2}\bigg]^{1/a}\Bigg|_{B^1=B^2=B}\Bigg)\\
%&< \infty.
%\end{align*}
%Taking this estimate back into \eqref{I2 express} and we have $\displaystyle I_2 \leq C|t-s|.$ And finally, Combing \eqref{I11}, \eqref{I12} and \eqref{I2 express}, we have
%$$\EE  \big| Z_t^1 - Z_s^1\big|^2 \leq C |t-s|^{(2H_0 -1 + \underline{H})\wedge 1}.$$
This proves that $I_{2,2}$ is bounded and hence $I_2\le C|t-s|$.
Hence, Combing \eqref{I11}, \eqref{I12} and \eqref{I2 express}, we have
$$\EE  \big| Z_t^2 - Z_s^2\big|^2 \leq C |t-s|^{2H_0 -1 + \underline{H}-\ep},$$
and finally we deduce
$$\EE  \big| Z_t - Z_s\big|^2 \leq C |t-s|^{(2H_0 -1 + \underline{H}-\ep)\wedge\kappa} ,\ \mbox{for all}\  \ep >0.$$
\end{proof}
\section{Uniqueness of solution}
We have proved  parts (1) and (2) of Theorem \ref{th1}.
In this section, we are going to prove part (3),  the uniqueness of BSDEs \eqref{e1}.
  We   need the following  proposition.
\begin{prop}\label{Uniq}
Suppose that the  conditions in Theorem \ref{th1} are  satisfied. Let $(Y,Z)\in\mathcal{S}^2_{\mathbb{F} }(0,T;\RR) \times\mathcal{M} ^2_{\mathbb{F} }(0,T;\RR^d)$ be the solution of BSDEs \eqref{e1} so that  $Y, D^B Y$ are  $ \mathbb{D}^{1,2}$. Then the solution   has the  explicit expression \eqref{e.1.6} and hence the BSDEs \eqref{e1} has a unique solution.
\end{prop}
Before we prove  Proposition \ref{Uniq}, we first need the  following lemma.

%Recall $\displaystyle q>\frac{2}{2\underline{H} -1}$. From previous work we know BSDEs \eqref{e appro1} is an approximation of \eqref{e1}, and their solution $(Y^{\ep, \eta},Z^{\ep, \eta})$ of BSDEs \eqref{e appro1} belongs to $\mathcal{S} ^p_{\mathbb{F} }(0,T;\RR)\times \mathcal{M} ^2_{\mathbb{F} }(0,T;\RR^d)$, for all $p\in[1,q)$ and $\ep,\eta>0$. Since these two spaces are completed, we denote by $(Y,Z)\in \mathcal{S} ^p_{\mathbb{F} }(0,T;\RR)\times \mathcal{M} ^2_{\mathbb{F} }(0,T;\RR^d)$ the limit of $(Y^{\ep, \eta},Z^{\ep, \eta})$.
\begin{lem}
Recall the  notation
\begin{equation}
\al_s^t=\exp\left\{\int_s^t {W} (dr,B_r)\right\}\,.
\label{e.3.43}
\end{equation}
Then $\al_s^t $ satisfies the following equation.
\begin{equation}\label{alexpr}
\alpha_0^t = \alpha_0^s+  \int^{t}_{0}\alpha_0^r\, W(dr,B_r).
\end{equation}
\end{lem}
\begin{proof}
Define $ K_t = \int_0^t {W} (dr,B_r)$.
%we get
%\begin{equation}
%d\alpha_0^t Y_t = \alpha_0^t d Y_t + Y_t d\alpha_0^t + dY_t d\alpha_0^t,
%\label{e.3.44}
%\end{equation}
%and from \eqref{e1} we have
%\begin{equation}
%dY_t+ Y_t    W  (dt,B_t) = Z_tdB_t,\quad  t\in [0,T],\ Y_T = \xi\,,
%\label{e.3.45}
%\end{equation}
Consider  a sequence of partitions $\pi_n = \{0=t_0 <t_1 <\ldots<t_{n} =t\}$ such that  $|\pi_n|=\max_{0\le i\le n-1} (t_{i+1} -t_i ) \rightarrow 0$ when $n\rightarrow \infty$  (the $t_i$'s  depend
on $n$  and we omit this explicit dependence to simplify notation). Since $H\in(1/2,1)$ and since $W$ satisfies \eqref{e.1.1}, by Markov's inequality, Proposition \ref{3.1} and the estimate \eqref{EIinfyty}, it is easy to obtain
\begin{align}
\lim_{n \rightarrow \infty} P\left\{\sum_{i=1}^{n}\left|\int_{t_i}^{t_{i+1}} W(dr,B_r)\right|^2 >\ep \right\}&\leq \lim_{n \rightarrow \infty}\frac{ \sum_{i=1}^{n} \EE\left[| \int_{t_i}^{t_{i+1}} W(dr,B_  r)|^2\right]}{\ep}\notag\\
&\leq \lim_{n \rightarrow \infty}\frac{ C \sum_{i=1}^{n} |t_{i+1}-t_{i}|^{2H_0}}{\ep}\notag\\
&=0,\label{e.3.46}
\end{align}
for any $\ep>0$.  On the other hand,  we have
\begin{equation}
\al_0^t-1 = \sum_{i=0}^n \left[ e^{K_{t_{i+1}}} -e^{K_{t_{i}}} \right] = \sum_{i=0}^n  \alpha_0^{t_i} \left(K_{t_{i+1}}- K_{t_{i}} \right) + R_t^n,
\end{equation}
where
$$R_t^n = \sum_{i=0}^n  \left(K_{t_{i+1}}- K_{t_{i}} \right)\int_0^1
\left[ e^  {K_{t_{i }} +(K_{t_{i+1 }}-K_{t_{i }})u}-e^{K_{t_{i }}}\right]   du. $$
Combining  \eqref{e.3.46} with    \eqref{estYinS2} yields
$$|R_t^n|\leq C \sup_{0\leq r\leq t} e^{K_r}\cdot \sum_{i=0}^n | K_{t_{i+1}} -K_{t_i}|^2\stackrel{P}{\longrightarrow} 0, \ \ n\rightarrow \infty. $$
This proves  that $\al_0^t$ satisfies \eqref{alexpr}.
\end{proof}
\begin{lem}  Let $(Y,Z)$ satisfy   \eqref{e1} and let $\al_t$ be given as above.
Suppose the conditions of Proposition \ref{Uniq} are satisfied.
Then
\begin{align}\label{e.5.5}
\alpha_0^T \xi -\alpha_0^t Y_t & =  \int_t^T   \alpha_0^{s}  Z_s     dB_s  \,.
\end{align}
\end{lem}
\begin{proof}
Let $(Y,Z)$ satisfy   \eqref{e1} and we use partition $\pi_n = \{t=t_0 <t_1 <\ldots<t_{n} =T\}$. Taking \eqref{alexpr} into account    we have
\begin{align}\label{e.3.47}
\alpha_0^T \xi -\alpha_0^t Y_t & = \sum_{i=1}^{n} \left(\alpha_0^{t_{i+1}} Y_{t_{i+1}} - \alpha_0^{t_{i}} Y_{t_{i}} \right)
=    {\sum_{i=1}^{n} \left(\alpha_0^{t_{i+1}} (Y_{t_{i+1}} -Y_{t_{i}}) + Y_{t_{i}}(\alpha_0^{t_{i+1}}- \alpha_0^{t_{i}} ) \right)}\notag\\
& = \sum_{i=1}^{n} \left(\alpha_0^{t_{i+1}} \big(-\int^{t_{i+1}}_{t_{i}} Y_r W(dr,B_r) + \int^{t_{i+1}}_{t_{i}}Z_r dB_r\big) \right)\notag \\
&\qquad\qquad\qquad\qquad\qquad\qquad\qquad+ \sum_{i=1}^{n} Y_{t_{i}} \int^{t_{i+1}}_{t_{i}}\alpha_0^r W(dr,B_r) \\
%& = \int_t^T\alpha_0^r Z_r dB_r\\
& = \sum_{i=1}^{n} \left(-\alpha_0^{t_{i }}   Y_{t_i} \int_{t_i}^{t_{i+1}}
 W(dr,B_r) + \alpha_0^{t_{i }}  Z_{t_i} \int^{t_{i+1}}_{t_{i}}  dB_r\big) \right)\notag \\
&\qquad\qquad\qquad\qquad\qquad\qquad\qquad+ \sum_{i=1}^{n} Y_{t_{i}}\alpha_0^{t_{i }}   \int^{t_{i+1}}_{t_{i}}  W(dr,B_r)+\tilde R_t^n  \notag\\
& = \sum_{i=1}^{n}   \alpha_0^{t_{i }}  Z_{t_i} \int^{t_{i+1}}_{t_{i}}  dB_r +\tilde R_t^n\,,\notag
\end{align}
where
\begin{equation}\label{tilR}
\begin{split}
\tilde R_t^n=&\sum_{i=1}^{n} \left(\alpha_0^{t_{i+1}}  \int^{t_{i+1}}_{t_{i}} \left[ Y_r -Y_{t_i}\right] W(dr,B_r)  \right)
+\sum_{i=1}^{n} \left( \alpha_0^{t_{i+1}}- \alpha_0^{t_{i }}\right)  Y_{t_i}  \int^{t_{i+1}}_{t_{i}}   W(dr,B_r)   \\
&+\sum_{i=1}^{n}  \alpha_0^{t_{i }}  \int^{t_{i+1}}_{t_{i}}\left[  Z_r-Z_{t_i}\right]  dB_r
+\sum_{i=1}^{n}  \left(\alpha_0^{t_{i+1}}- \alpha_0^{t_{i }}\right)  \int^{t_{i+1}}_{t_{i}}   Z_r   dB_r\\
&+ \sum_{i=1}^{n} Y_{t_{i}} \int^{t_{i+1}}_{t_{i}}\left[ \alpha_0^r - \alpha_0^{t_{i }}\right] W(dr,B_r)  \\
=& \sum_{i=1}^n R_{1,i} +R_{2,i} +R_{3,i} +R_{4,i} +R_{5,i}.
\end{split}
\end{equation}
%We recall $\displaystyle q>\frac{2}{2\underline{H} -1}$, and from Proposition \ref{exp int V}, Theorem \ref{z} and Theorem \ref{yw converge}, it has been proved that for all $1/a +1/b=1$, where $b \in[1,q)$,
%\begin{equation}
%\begin{split}
%\EE\left[|\alpha_0^t|^a+ |Y_r|^a+ \left|\int_0^T \alpha_0^r W(dr,B_r)\right|^b+ \left|\int_0^T Y_r W(dr,B_r)\right|^b +\left|\int_0^T Z_r dr\right|^2 \right]<\infty.
%\end{split}
%\end{equation}
%Note that, by a similar way in Section 4, the continuity coefficients of $(Y^{\ep,\eta},Z^{\ep,\eta}) $ in \eqref{e appro1} are the same as the conclusion (2) in Theorem \ref{th1}, uniformly  with $\ep,\eta$. So its limit $(Y,Z)\in \mathcal{S} ^p_{\mathbb{F} }(0,T;\RR)\times \mathcal{M} ^2_{\mathbb{F} }(0,T;\RR^d)$ also has the same continuity coefficients. Thus, combining with \eqref{inprod func true}, \eqref{EIinfyty} we have the following estimate:
For $R_{1,i}$, using  \eqref{F2split} we get
\begin{equation}\label{R1}
\begin{split}
|R_{1,i} |^2  &\lesssim \EE\left[\int^{t_{i+1}}_{t_{i}} \left[ Y_r -Y_{t_i}\right] W(dr,B_r) \right]^2\\
&=\EE\left[\int_{[t_{i},t_{i+1}]^2}(Y_r-Y_{t_i})  (Y_s-Y_{t_i})  |s-r|^{2H_0-2} q(B_r,B_s)  drds\right]\\
&\qquad+  \EE\Bigg[\int_{[t_{i},t_{i+1}]^2}\int_{[r,t_{i+1}]}\int_{[t_i,s]}\int_{\RR^{2d}}D_{r,y}^W (Y_u-Y_{t_i}) \, D_{v,w}^W (Y_s-Y_{t_i}) \\
&\qquad \qquad   \qquad \times |u-v|^{2H_0-2}|s-r|^{2H_0-2} q(B_u,w) q(B_s ,y)dw dy dv du drds\Bigg]\,.
\end{split}\end{equation}
%From Theorem \ref{yw converge} we know BSDE \eqref{e1} exists a solution component $Y$ such that it belongs to the space $ \mathbb{D}_W^{1,2}$.
Recalling \eqref{R split} that covariance $q(x,y)$ satisfies
\begin{equation}
\left| q(x, y) \right|
     \leq C \prod^d_{i=1} (1+|x_i|)^{2H_i-\beta_i}
     (1+|y_i|)^{2H_i-\beta_i} ,
\end{equation}
where $\beta_i>2H_i +1,i=1,\ldots,d$, it yields
\begin{align}\label{mesh1}
%\EE\left[\left| \int^{t_{i+1}}_{t_{i}}\left[ Y_r -Y_{t_i}\right] W(dr,B_r)\right|^2\right]
|R_{1,i} |^2  &\lesssim      \int_{t_i}^{t_{i+1}} \int_{t_i}^{t_{i+1}}\EE\Big[|Y_r -Y_{t_i}|^2 \Big]^{1/2} \EE\Big[|Y_s -Y_{t_i}|^2 \Big]^{1/2} |s-r|^{2H_0-2}\sup_{\om,s,r}\Big|q(B_s,B_r)\Big| dsdr \notag\\
&\qquad  + \int_{t_i}^{t_{i+1}} \int_{t_i}^{t_{i+1}} \int_{r}^{t_{i+1}} \int_{t_i}^{s} \EE\Big[|D_{r,y}^W(Y_r -Y_{t_i})|^2 \Big]^{1/2} \EE\Big[D_{v,w}^W(|Y_s -Y_{t_i})|^2 \Big]^{1/2} \\
&\qquad\qquad \qquad \times |u-v|^{2H_0-2}|s-r|^{2H_0-2}\sup_{\om,u,s,w,y}\Big|q(B_u,w)q(B_s,y)\Big| dvdudsdr. \notag
\end{align}
If $Y$ satisfies condition (3) in Theorem \ref{th1}, that is, the continuity coefficient of $Y$ is $\frac{1}{2}$, then we can directly obtain
\begin{align}
|R_{1,i} |^2   \lesssim    C \Big(|t_{i+1}-t_i|^{1+2H_0} + |t_{i+1}-t_i|^{4H_0}\Big)
\end{align}
If $Y$ satisfies condition (4) in Theorem \ref{th1},  then from \eqref{e1}, we have
\begin{equation}\label{e2}
  \begin{split}
  Y_r-Y_{t_i}= \int_r^{t_i} Y_s    W  (ds,B_s)  -\int_r^{t_i} Z_sdB_s,\quad  r\in [{t_i},t_{t+1}]\,.
  \end{split}
\end{equation}
With the help of \eqref{F2split} again and the fact that $(Y,Z)\in\mathcal{S}^2_{\mathbb{F} }(0,T;\RR) \times\mathcal{M} ^2_{\mathbb{F} }(0,T;\RR^d)$ as well as  that $Y$ also belongs  to $ \mathbb{D}^{1,2}$, it holds
\begin{align}
\EE\left|Y_r-Y_{t_i}\right|^2\leq& \,2\EE\left[\int_{[{t_i},r]^2}Y_s   Y_u |s-u|^{2H_0-2} q(B_s,B_u)  dsdu\right]\notag\\
&\qquad+  2\EE\Bigg[\int_{[{t_i},r]^2}\int_{[u,r]}\int_{[t_i,s]}\int_{\RR^{2d}}D_{s',y}^W Y_u \, D_{v,w}^W Y_s \notag\\
&\qquad \qquad    \times |u-v|^{2H_0-2}|s-s'|^{2H_0-2} q(B_u,w) q(B_s ,y)dw dy dudv ds'ds\Bigg]\,\notag\\
&\qquad + \EE\left|\int_r^{t_i} Z_sdB_s \right|^2\notag\\
\leq &\, 2\int_{[{t_i},r]^2}\EE\left[|Y_s|^2\right]^{1/2}   \EE\left[|Y_u|^2\right]^{1/2} |s-u|^{2H_0-2} \sup_{\om,s,u}\big|q(B_s,B_u)\big|  dsdu\\
&\qquad+  2 \int_{[{t_i},r]^2}\int_{[u,r]}\int_{[t_i,s]}\int_{\RR^{2d}}\EE\left[|D_{s',y}^W Y_u|^2\right]^{1/2} \EE\left[|D_{v,w}^W Y_s|^2\right]^{1/2} |u-v|^{2H_0-2} \notag\\
&\qquad \qquad   \times|s-s'|^{2H_0-2} \sup_{\om,s,u,w,y}\big|q(B_u,w)q(B_s ,y)\big|dw dy dv du ds'ds \,.\notag\\
 &\qquad+\int_r^{t_i} \EE\left|Z_s \right|^2 ds\notag\\
\leq &\,C( |r-t_i|^{2H_0} + |r-t_i|^{4H_0} + |r-t_i|).\notag
\end{align}
Taking this result back to \eqref{mesh1}   we get
\begin{align}
|R_{1,i} |^2  &\lesssim    C   \left(\int_{t_i}^{t_{i+1}} \int_{t_i}^{t_{i+1}}|r-t_i|^{1/2} |s-t_i|^{1/2}|s-r|^{2H_0-2}dsdr\right.\notag\\
&\qquad\qquad +\left.\int_{t_i}^{t_{i+1}} \int_{t_i}^{t_{i+1}} \int_{r}^{t_{i+1}} \int_{s}^{t_{i+1}} |u-v|^{2H_0-2}|s-r|^{2H_0-2}dudvdsdr\right).\notag\\
&\leq  C \left(|t_{i+1}-t_i|  \int_{t_i}^{t_{i+1}} \int_{t_i}^{t_{i+1}}|s-r|^{2H_0-2}dsdr\right.\\
&\qquad\qquad +\left.\int_{t_i}^{t_{i+1}} \int_{t_i}^{t_{i+1}} \int_{r}^{t_{i+1}} \int_{s}^{t_{i+1}} |u-v|^{2H_0-2}|s-r|^{2H_0-2}dudvdsdr\right).\notag\\
&\leq  C \Big(|t_{i+1}-t_i|^{1+2H_0} + |t_{i+1}-t_i|^{4H_0}\Big).\notag
\end{align}
Hence we have
\[
\sum_{i=0}^{n-1}  |R_{1,i} |\lesssim  \sum_{i=0}^{n-1} (t_{i+1}-t_i)^{{1/2}+ H_0}
  \leq  \max_{0\le i\le n-1} (t_{i+1}-t_i)^{{1/2}+ H_0-1}\sum_{i=0}^{n-1}  (t_{i+1}-t_i) \to 0\,,\ n\to \infty.
\]
Using  \eqref{F2split} again, we get
\begin{align}\label{R5}
|R_{5,i} |^2  &\lesssim \EE\left(\int_{t_i}^{t_{i+1} } ( \alpha_0^s -\al_0^{t_i}) W(ds,B_s)\right)^2\notag\\
 &=
\EE \left[\int_{t_i}^{t_{i+1} }\int_{t_i}^{t_{i+1} } ( \alpha_0^s -\al_0^{t_i})( \alpha_0^r -\al_0^{t_i})  |s-r|^{2H_0-2} q(B_s, B_r)   dsdr \right] \notag\\
&\qquad+\,\EE \bigg[ \int_{t_i}^{t_{i+1}} \int_{t_i}^{t_{i+1}}\int_{r}^{t_{i+1} } \int_{t_i}^{s } \int_{\RR^{2d}}D_{r,y}^W (\alpha_0^u-\alpha_0^{t_i}) \cdot D_{v,w}^W (\alpha_0^s-\alpha_0^{t_i}) \, \\
&\qquad\qquad\qquad\qquad\qquad \times |u-v|^{2H_0-2}|s-r|^{2H_0-2}q(B_u,w)q(B_s,y)dwdy dvdudrds\bigg]\notag\\
&:=R_{5,1,i}+ R_{5,2,i}.\notag
\end{align}
For $R_{5,2,i}$, recalling \eqref{e.3.43} we have
\begin{equation}\begin{split}
R_{5,2,i}  &=\EE \bigg[ \int_{t_i}^{t_{i+1}} \int_{t_i}^{t_{i+1}}\int_{r}^{t_{i+1} } \int_{t_i}^{s } \int_{\RR^{2d}} (\alpha_0^u-\alpha_0^{t_i}) \,\delta(B_r-y)\cdot  (\alpha_0^s-\alpha_0^{t_i})\,\delta(B_v-w)\,\\
 &\qquad\qquad\qquad \times |u-v|^{2H_0-2}|s-r|^{2H_0-2}  q(B_u,w)q(B_s,y)dwdy dvdudrds\bigg]\\
 &=\bigg[ \int_{t_i}^{t_{i+1}} \int_{t_i}^{t_{i+1}}\int_{r}^{t_{i+1} } \int_{t_i}^{s } (\alpha_0^u-\alpha_0^{t_i}) \, (\alpha_0^s-\alpha_0^{t_i})\,\\
 &\qquad\qquad\qquad \times |u-v|^{2H_0-2}|s-r|^{2H_0-2}  q(B_u,B_v)q(B_s,B_r) dvdudrds\bigg].
\end{split}\end{equation}
Taking this result back to \eqref{R5}, and with the help of \eqref{perI1}, \eqref{YI1}, we obtain
\begin{align}\label{G2split}
|R_{5,i} |^2  &\lesssim  \int^{t_{i+1}}_{t_{i}}\int^{t_{i+1}}_{t_{i}} \EE\left[(\alpha_0^r-\alpha_0^{t_i}) \,(\alpha_0^s-\alpha_0^{t_i}) \,  |s-r|^{2H_0-2} q(B_r,B_s)\right]ds dr\notag\\
&\qquad+\int^{t_{i+1}}_{t_{i}}\int^{t_{i+1}}_{t_{i}} \int^{t_{i+1}}_{r}\int_{t_i}^{s }\EE\left[ (\alpha_0^u-\alpha_0^{t_i})  \cdot  (\alpha_0^s -\alpha_0^{t_i})\,\right.\notag\\
 &\qquad\qquad\qquad\qquad\times|u-v|^{2H_0-2} |s-r|^{2H_0-2}q(B_u,B_v) \,  q(B_r,B_s)\Big]dudvds dr \notag\\
 &\leq C \int^{t_{i+1}}_{t_{i}}\int^{t_{i+1}}_{t_{i}} \EE \big[|\alpha_0^r -\alpha_0^{t_{i }}|^4\big]^{1/4}\EE \big[|\alpha_0^r -\alpha_0^{t_{i}}|^4\big]^{1/4}|r-s|^{2H_0-2}\EE \big[|q(B_r,B_s)|^2\big]^{1/2}dsdr\notag\\
 &\qquad +\int^{t_{i+1}}_{t_{i}}\int^{t_{i+1}}_{t_{i}} \int^{t_{i+1}}_{r}\int_{t_i}^{s } \EE \big[|\alpha_0^u -\alpha_0^{t_{i }}|^4\big]^{1/4}\EE \big[|\alpha_0^s -\alpha_0^{t_{i}}|^4\big]^{1/4}\notag\\
 &\qquad\qquad\qquad\qquad\times|r-s|^{2H_0-2}|u-v|^{2H_0-2}
 \EE\big[|q(B_u,B_v)q(B_r,B_s)|^2\big]^{1/2}dudvdsdr\\
 &\leq C \int^{t_{i+1}}_{t_{i}}\int^{t_{i+1}}_{t_{i}} (r-t_i)^{H_0}(s-t_i)^{H_0}(r-s)^{2H_0-2}dsdr\notag\\
 &\qquad + C \int^{t_{i+1}}_{t_{i}}\int^{t_{i+1}}_{t_{i}} \int^{t_{i+1}}_{r}\int_{t_i}^{s } (u-t_i)^{H_0}(s-t_i)^{H_0}|r-s|^{2H_0-2}|u-v|^{2H_0-2}dudvdsdr\notag\\
  &\leq C(t_{i+1}-t_i)^{2H_0}\int^{t_{i+1}}_{t_{i}}\int^{t_{i+1}}_{t_{i}} (r-s)^{2H_0-2}dsdr\notag\\
 &\qquad +  C (t_{i+1}-t_i)^{2H_0}\int^{t_{i+1}}_{t_{i}}\int^{t_{i+1}}_{t_{i}} \int^{t_{i+1}}_{r}\int_{t_i}^{s } |r-s|^{2H_0-2}|u-v|^{2H_0-2}dudvdsdr.\notag\\
&\leq C \big((t_{i+1}-t_i)^{4 H_0}+(t_{i+1}-t_i)^{6 H_0}\big).\notag
\end{align}
Thus
\[
\sum_{i=0}^{n-1}  |R_{5,i} |\lesssim  \sum_{i=0}^{n-1} (t_{i+1}-t_i)^{2H_0}
  \leq  \max_{0\le i\le n-1} (t_{i+1}-t_i)^{2H_0-1}\sum_{i=0}^{n-1}  (t_{i+1}-t_i) \to 0,\,\ n\to \infty.
\]
For $R_{3,i}$, from the orthogonality of the increments of standard Brownian motion and the fact that $\al_0^{t_i}=\exp\left\{\int_0^{t_i} {W} (dr,B_r)\right\}$ is $\cF_{t_i}$-adapted,  we have
\begin{equation}\label{Zconver}
\begin{split}
%\left|\sum_{i=1}^{n} R_{3,i}\right|^2&=
\EE\left[\left|\sum_{i=1}^{n} \alpha_0^{t_{i}} \int^{t_{i+1}}_{t_{i}}\left[  Z_r-Z_{t_i}\right]  dB_r\right|^2\right]\le &\sum_{i=1}^{n} \EE\left|\alpha_0^{t_{i }}\right|^2\EE \left|\int^{t_{i+1}}_{t_{i}}\left[  Z_r-Z_{t_i}\right]  dB_r\right|^2\\
%=\sum_{i=1}^{n}\EE\left[\left|   \red{\left( \alpha_0^{t_{i+1}}-\alpha_0^{t_{i }}+\alpha_0^{t_{i }} \right) }\int^{t_{i+1}}_{t_{i}}\left[  Z_r-Z_{t_i}\right]  dB_r\right|^2\right]
\leq& C\sum_{i=1}^{n}\int^{t_{i+1}}_{t_{i}} \EE\left|  Z_r-Z_{t_i}\right|^2 dr.
\end{split}
\end{equation}
%\begin{equation}\label{Zconver}
%\begin{split}
%|R_{3,i}|^2 \lesssim  \EE\left[\left|   \int^{t_{i+1}}_{t_{i}}\left[  Z_r-Z_{t_i}\right]  dB_r\right|^2\right] \leq
% \int^{t_{i+1}}_{t_{i}} \EE\left|  Z_r-Z_{t_i}\right|^2 dr.
%\end{split}
%\end{equation}
%Since $Z_t = D^B_t Y_t,\,\forall t\in[0,T]$, and from \eqref{e1} we have
%\begin{eqnarray}\label{DBY}
%  D_r ^B Y_t
%  &=& D_r^B \xi + \int_t^T D_r ^B Y_s {{W}}(ds,B_s) \\
%  &&\qquad  +\int_t^T Y_s  I_{[0, s]}(r)\nabla_x {{W}} (ds, B_s)
%  -\int_t^T D_r^B Z_s  dB_s,\ 0\leq t\leq r\leq T.\notag
%\end{eqnarray}
If $Z$ satisfies condition (3) in Theorem \ref{th1},   we have easily
\begin{align}
 \left|\sum^{n-1}_{i=0}R_{3,i}\right|^2 \lesssim C \sum^{n-1}_{i=0} |t_{i+1}-t_i|^{\kappa+1}.
\end{align}
Denote  $\tilde Y_t=D_r^BY_t ,\ \tilde Z_t=D_r^B Z_t $ (we fix $r$), and from \eqref{e1}  we   obtain
\begin{equation}\label{DBY1}
\begin{split}
 D_r^BY_t =\tilde Y_t
  =& D_r^B \xi + \int_t^T \tilde Y_s {{W}}(ds,B_s) \\
  &\qquad  +\int_r^T Y_s   \nabla_x {{W}} (ds, B_s)
  -\int_t^T \tilde Z_s  dB_s,\ 0\leq t\leq r\leq T.
  \end{split}
\end{equation}
Therefore, we first need to verify the square integrability of $\int_t^T Y_s \nabla_x {{W}} (ds, B_s)$, and then we can treat \eqref{DBY1} in a similar way to that for  \eqref{e1}.
%From the proof of Theorem \ref{z} we know that there exist a solution $(D_r ^B Y_t, D_r^B Z_t )$ of BSDE above is well-defined, i.e., \red{$\displaystyle \EE\Big[\int_0^T |D_r ^B Y_t|^2 + |D_r ^B Z_t|^2 dt\Big]<\infty$.} In addition we obtain
%We just need to verify
We can write
\begin{equation}
\begin{split}
\EE\Big|\int_t^T Y_s \nabla_x {{W}} (ds, B_s)\Big|^2 &= \EE\Big|\int_t^T Y_s \nabla_x \delta(B_s-x){{W}} (ds, x)dx\Big|^2 \\
%&=\EE\left[\int_{[a,b]^2}\int_{\RR^{2d}} D^W_{r,y}H\cdot Y_s \nabla_x \delta(B_s-z)|s-r|^{2H_0-2}q(y,z)dydzdrds  \right]\\
%&=\EE\left[\int_{[a,b]^2}\int_{\RR^{d}} D^W_{r,y}H\cdot Y_s   |s-r|^{2H_0-2}\nabla_z q(y,B_s)dy drds  \right],
\end{split}
\end{equation}
From \eqref{F2split} and by integration  by parts, for all $0\leq t\leq T$,
\begin{align}\label{YnabW}
&\EE\Big|\int_t^T Y_s \nabla_x {{W}} (ds, B_s)\Big|^2\notag\\
=\,&\EE\left[\int_{[t,T]^2}Y_r  Y_s  |s-r|^{2H_0-2} \nabla_{x,y} q(B_r,B_s)dy drds\right]\notag\\
&\quad+\EE\left[\int_{[a,b]^2}\int_{\RR^{d}}\int_{[r,b]}\int_{[s,b]}\int_{\RR^{d}}D_{r,y}^W Y_u \, D_{v,w}^W Y_s \right.\notag\\
 &\quad\qquad\qquad\times |u-v|^{2H_0-2}|s-r|^{2H_0-2} \nabla_x q(B_u, w)\nabla_x q(B_s, y)dw dudv dy drds\bigg] \\
\leq & \int_{[t,T]^2} \EE\Big[\big|Y_r\big|^2\Big]^{1/2}\EE\Big[\big|Y_s\big|^2\Big]^{1/2}|s-r|^{2H_0-2} \sup_{\om,r,s}\big|\nabla_{x,y} q(B_r,B_s)\big|dy drds\notag\\
&\quad+\EE\left[\int_{[t,T]^2}\int_{\RR^{d}}\int_{[r,T]}\int_{[s,T]}\int_{\RR^{d}}
\EE\Big[\big|D_{r,y}^W Y_u\big|^2\Big]^{1/2} \, \EE\Big[\big|D_{v,w}^W Y_s\big|^2\Big]^{1/2} \right.\notag\\
&\quad\qquad\qquad\times |u-v|^{2H_0-2}|s-r|^{2H_0-2}\sup_{\om,u,s,w,y}\big| \nabla_x q(B_u, w)\nabla_x q(B_s, y)\big|dw dudv dy drds\bigg]\notag\\
\leq& C\big(|T-t|^{2H_0} + |T-t|^{4H_0})\leq C T^{4H_0}.\notag
\end{align}
Thus we have $(\tilde Y_t, \tilde Z_t )$ of BSDE \eqref{DBY1} is well-defined, i.e., $\displaystyle \EE\Big[\int_0^T |\tilde Y_t|^2 + |\tilde Z_t|^2 dt\Big]<\infty$. Using the classical conclusion that $Z_t = D^B_t Y_t,\,\forall t\in[0,T]$ (see e.g. \cite{hu2011malliavin}), we can treat $Z_r-Z_{t_i}$ as
\begin{equation}\label{Zr-Zt}
\begin{split}
Z_r-Z_{t_i} &=  (D_r^B \xi -D_{t_i}^B \xi) + \int_r^{t_i} D_r ^B Y_s {{W}}(ds,B_s) \\
  &\qquad  +\int_r^{t_i} Y_s \nabla_x {{W}} (ds, B_s)
  -\int_r^{t_i} D_r^B Z_s  dB_s,\ 0\leq t_i\leq r\leq T\\
  &=\tilde Z_1+\tilde Z_2+\tilde Z_3+\tilde Z_4\,.
\end{split}
\end{equation}
For the above first term $\tilde Z_1$
we can  use  the assumption   $\EE|D_r^B \xi -D_{t_i}^B \xi|^2 \leq C |r-t_i|^{\kappa}$
for some $\kappa>0  $.
We can deal with the second term $\tilde Z_2$ in \eqref{Zr-Zt} in the similar way as in \eqref{R1}. In fact, with the help of \eqref{F2split} again, it has
\begin{equation}
\begin{split}
\EE&\left| \tilde Z_2\right|^2\\
&\leq \int_{r}^{t_{i+1}} \int_{r}^{t_{i+1}}\EE\Big[|D_r^BY_s|^2 \Big]^{1/2} \EE\Big[|D_r^BY_{s'}|^2 \Big]^{1/2} |s-s'|^{2H_0-2}\sup_{\om,s,s'}\Big|q(B_s,B_s')\Big| dsds'  \\
&\qquad\qquad + \int_{r}^{t_{i+1}} \int_{r}^{t_{i+1}} \int_{s'}^{t_{i+1}} \int_{r}^{s } \left( \EE\Big[|D_{s',y}^W(D_r^BY_{s'} )|^2 \Big]^{1/2} \EE\Big[D_{v,w}^W(|D_r^BY_s  )|^2 \Big]^{1/2}\right. \\
&\qquad\qquad \qquad\qquad\times \left.|u-v|^{2H_0-2}|s-s'|^{2H_0-2}\sup_{\om,u,s,w,y}\Big|q(B_u,w)q(B_s,y)\Big|\right) dudvdsds'\,.
\end{split}
\end{equation}
Since $Y, D^B Y\in \mathbb{D}^{1,2}$, we have the estimate
\begin{equation}
\begin{split}
\EE \left| \tilde Z_2\right|^2&\leq
C \int_{r}^{t_{i+1}} \int_{r}^{t_{i+1}}  |s-s'|^{2H_0-2} dsds'  \\
&\quad  +C \int_{r}^{t_{i+1}} \int_{r}^{t_{i+1}} \int_{s'}^{t_{i+1}} \int_{r}^{s }  |u-v|^{2H_0-2}|s-s'|^{2H_0-2}  dudvdsds'\\
&\leq  C\left(|t_{i+1}-r|^{2H_0}+|t_{i+1}-r|^{4H_0}\right).
\end{split}
\end{equation}
From \eqref{YnabW} we have
\begin{equation}
\EE\Big|\tilde Z_3\Big|^2\leq C\left(|t_{i+1}-r|^{2H_0}+|t_{i+1}-r|^{4H_0}\right).
\end{equation}
Finally it is easy to obtain
\begin{equation}
\displaystyle\EE \big|\tilde Z_4|^2 \leq \sup_{s\in [r,t_i]} \EE |D^B_r Z_s|^2 \,|t_i-r|\leq C\, |t_i-r|.
\end{equation}
Taking those estimates back to \eqref{Zconver} we have
\begin{equation}\label{Zficonv}
\begin{split}
%\EE\left[\left|\sum_{i=1}^{n}  \int^{t_{i+1}}_{t_{i}}\left[  Z_r - Z_{t_i}\right]  dB_r\right|^2\right] \leq
%\left|\sum_{i=1}^{n} R_{3,i}\right| \leq
 \left(\sum_{i=1}^{n} \int^{t_{i+1}}_{t_{i}} \EE\left|  Z_r-Z_{t_i}\right|^2 dr\right)^{1/2}\leq
C  &\left(\sum_{i=1}^{n}\big(|t_{i+1}-t_i|^{2H_0+1}  +  |t_{i+1}-t_i|^{4H_0+1}\right.\\
 &+ |t_{i+1}-t_i|^{\kappa+1} + |t_{i+1}-t_i|^{2}\bigg)^{1/2},
\end{split}
\end{equation}
which implies
\begin{equation}
\begin{split}
\left|\sum_{i=1}^{n} R_{3,i}\right|^2\lesssim & \,C \sum_{i=1}^{n}\big(|t_{i+1}-t_i|^{2H_0+1}  +  |t_{i+1}-t_i|^{4H_0+1} \\
 &\ + |t_{i+1}-t_i|^{\kappa+1} + |t_{i+1}-t_i|^{2} \big)\\
\leq& \max_{0\leq i\leq n-1}(t_{i+1}-t_i)^{\kappa\wedge1} \sum_{i=1}^{n}|t_{i+1}-t_i| \rightarrow 0,\ n\rightarrow \infty.
 \end{split}
\end{equation}
For $R_{2,i}$ and $R_{4,i}$, it is easy to deduce that
\begin{equation}
\begin{split}
|R_{2,i}|^2 &\lesssim C\,\left(\EE \Big| Y_{t_i}\Big|^2\,\EE\left[\left| \alpha_0^{t_{i+1}}- \alpha_0^{t_{i }}\right|^4\right]^{1/2} \EE\left[\left| \int^{t_{i+1}}_{t_{i}}   W(dr,B_r)\right|^4\right]^{1/2}\right) \\
&\leq C\,|{t_{i+1}}-{t_{i}}|^{4H_0},
 \end{split}
\end{equation}
and
\begin{equation}
\begin{split}
|R_{4,i}|^2 &\lesssim C\EE\left[\left| \alpha_0^{t_{i+1}}- \alpha_0^{t_{i }}\right|^4\right]^{1/2} \EE\left[\left|  \int^{t_{i+1}}_{t_{i}} Z_r  dB_r\right|^2\right]\\
&\leq C\,|{t_{i+1}}-{t_{i}}|^{2H_0+1}.
 \end{split}
\end{equation}
Thus we have
\[
\sum_{i=0}^{n-1}  |R_{2,i} |\lesssim  \sum_{i=0}^{n-1} (t_{i+1}-t_i)^{2H_0}
  \leq  \max_{0\le i\le n-1} (t_{i+1}-t_i)^{2H_0-1}\sum_{i=0}^{n-1}  (t_{i+1}-t_i) \to 0,\,\ n\to \infty.
\]
and
\[
\sum_{i=0}^{n-1}  |R_{4,i} |\lesssim  \sum_{i=0}^{n-1} (t_{i+1}-t_i)^{H_0+1/2}
  \leq  \max_{0\le i\le n-1} (t_{i+1}-t_i)^{H_0-1/2}\sum_{i=0}^{n-1}  (t_{i+1}-t_i) \to 0,\,\ n\to \infty.
\]
Hence, letting  the mesh size $|\pi_n|$ goes to zero yields
$\tilde R_t^n\rightarrow 0,\ P\mbox{-a.s.,}$
and the right side of \eqref{e.3.47} converges to $\displaystyle \int_t^T\alpha_0^r Z_r dB_r.$
This concludes the proof of the lemma.
\end{proof}
\begin{proof}[Proof of Proposition \ref{Uniq}]
In equation \eqref{e.5.5} we take the conditional expectation with respect to $\mathcal{F}_t^B$  we see  to obtain
\[
\alpha_0^t Y_t=\EE \left[ \alpha_0^T\xi |\cF_t\right]\,.
\]
Thus,
\begin{equation}
Y_t = \left(\alpha_0^t\right)^{-1}\EE^B \left[\alpha_0^T \xi\big|\cF_t^B \right] = \EE^B \left[\alpha_t^T \xi\big|\cF_t^B \right]=\EE ^B\left[\xi   \exp\big(\int_t^T   {W}(dr,B_r)\big)\,\big|\,\mathcal{F}_t^B\right]\,.
\label{e.3.48}
\end{equation}
From the general relationship between $Z$ and $Y$ (e.g. \cite{hu2011malliavin}) we have
 \begin{equation}
Z_t=D_t^BY_t=D_t^B\EE^B \left[\xi   \exp\big(\int_t^T  {W}(dr,B_r)\big)\,\big|\,\mathcal{F}_t^B\right] \,.
 \end{equation}
 This concludes the proof of the proposition.
\end{proof}
\section{BSDEs and semilinear SPDEs}\label{s.5}
In this section we   obtain the regularity of the solution to the BSDE, and then establish the
relationship between the SPDE
\begin{equation}\label{inveranderson}
- d u(t,x)=\frac12 \Delta u(t,x) dt+u(t,x) W(dt, x),\ u(T, x) =  \phi(x)\,.
\end{equation}
and our BSDE
\begin{equation}\label{marBSDE}
Y_s^{t,x}= \phi(B_{T-t}^{t,x})+\int_s^T
Y_r^{t,x} W(dr, B_{r-t}^{t,x})  -\int_s^t
Z^{t,x}_r dB_r\,,\ s\in[t,T].
\end{equation}
\begin{thm}\label{SisB}
Suppose $\phi\in C^2(\RR^d)$. Let $\{u(t,x):t\in[0,T], x\in\RR^d\}$ be a random field such that $u(t,x)$ is $\cF_t$-measurable for each $(t,x)$, $u\in C([0,T]\times \RR^d, \RR)\, a.s.$  and let  $u(t,x)$ satisfy \eqref{inveranderson}.  Then $u(t,x)=Y_t^{t,x},\ \nabla u(t,x)=Z_t^{t,x}$, where $(Y_t^{t,x},Z_t^{t,x})$ is the solution of \eqref{marBSDE}.
\end{thm}
\begin{proof}
It suffices to show that $(u(s,X_s^{t,x}), \nabla u(s,X_s^{t,x}))_{s\in[t,T]}$ solves BSDE \eqref{marBSDE}.   Since $W(t,x)$
is not differentiable in $t$ and $x$,  one could not apply It\^{o}'s formula to $u(s,X_s^{t,x})$. Let us consider
\begin{equation}\label{apSPDE}
-d u^{\ep,\eta}(t,x)=\frac12 \Delta u^{\ep,\eta}(t,x) dt+u^{\ep,\eta}(t,x) \dot{W}_{\ep,\eta}(dt, x),\ u(T, x) =  \phi(x)\,.
\end{equation}
Recall \eqref{nsp} we have
$${\dot{W}}_{\varepsilon,\eta}(s,x) = \int_0^s \int_{R^d} \varphi_\eta (s-r) p_\varepsilon (x -y) W(dr,y)dy\,. $$
We see that
$u(t,x)$ is differentiable with respect to $x$. Now  we can use It\^{o}'s formula to $u^{\ep,\eta}(s,X_s^{t,x})$ to  deduce
\begin{equation}\label{inverapSPDE}
u^{\ep,\eta}(t,X_s^{t,x})=\phi(X_T^{t,x}) + \int_s^T u^{\ep,\eta}(r,X_r^{t,x}){\dot{W}}_{\varepsilon,\eta}(r,X_r^{t,x})dr - \int_s^T \nabla u^{\ep,\eta}(r,X_r^{t,x}) dB_r.
\end{equation}
Note that $X_r^{t,x} = x +B_r-B_t = B^x_{r-t}$, and by the uniqueness of BSDE %\eqref{e appro1}
we know
 $  Y_s^{t,x,\ep,\eta}=u^{\ep,\eta}(s,X_s^{t,x}),\  Z_s^{t,x,\ep,\eta}
 =\nabla u^{\ep,\eta}(s,X_s^{t,x})$
 satisfy
 \begin{equation}\label{apmarBSDE}
Y_s^{t,x,\ep,\eta}= \phi(B_{T-t}^{t,x})+\int_s^T
Y_r^{t,x,\ep,\eta} \dot{W}(r, B_{r-t}^{t,x})dr  -\int_s^t
Z^{t,x,\ep,\eta}_r dB_r\,,\ s\in[t,T].
\end{equation}
%is same as \eqref{inverapSPDE}, i.e.,

Theorem \ref{th1}  yields
$$Y_s=\lim_{\ep,\eta\rightarrow 0}u^{\ep,\eta}(s,X_s^{t,x})= \lim_{\ep,\eta\rightarrow 0} Y_s^{t,x,\ep,\eta} ,\qquad
Z_s=\lim_{\ep,\eta\rightarrow 0}\nabla u^{\ep,\eta}(s,X_s^{t,x})= \lim_{\ep,\eta\rightarrow 0} Z_s^{t,x,\ep,\eta}  $$
  is a solution pair of BSDE \eqref{marBSDE}.
It remains  to show SPDE \eqref{apSPDE} converges to \eqref{inveranderson}.
  From the classical Feynman-Kac formula
  it follows
  \begin{equation}\label{uep}
u^{\ep,\eta}(t,x) = \EE^B\left[\phi(X_T^{t,x})\exp\Big\{\int_t^T \dot{W}^{\ep,\eta}(r,X_r^{t,x})dr\Big\}\right]= \EE^B\left[\phi(B_{T-t}^{x})\exp\Big\{\int_t^T \dot{W}^{\ep,\eta}(r,(B_{T-t}^{x}))dr\Big\}\right].
\end{equation}
Define
\begin{equation}\label{FKu}
u(t,x)=\EE\left[\phi(B_{T-t}^{x})\exp\Big\{\int_t^T {W}(dr,(B_{T-t}^{x}))\Big\}\right]\,.
\end{equation}
Similar to  the proof of    Lemma \ref{Y converge},  we  can deduce
\begin{equation}\label{uconv}
 \lim_{\ep,\eta\rightarrow 0}\EE^W \big|u^{\ep,\eta}(t,x)-u(t,x)\big|^p =0 \ \ \mbox{for all}\ p\geq 2.
\end{equation}
Since $u^{\ep,\eta}$ satisfies \eqref{apSPDE}  for any $C^{\infty}$ function $\psi$ with compact support, we have
\begin{align}
\int_{\mathbb{R}^{d}} u^{\ep,\eta}(t, x) \psi(x) d x=& \int_{\mathbb{R}^{d}} \phi(x) \psi(x) d x+\frac{1}{2} \int_{t}^{T} \int_{\mathbb{R}^{d}} u^{\ep,\eta}(s, x) \Delta \psi(x) d x d s \notag\\
&+\int_{t}^{T} \int_{\mathbb{R}^{d}} u^{\ep,\eta}(s, x) \psi(x) \dot{W}_{\ep,\eta}(s,  x)dsdx.\label{weakep}
\end{align}
Letting $\varepsilon, \eta\to 0$ will yield
\begin{align}
\int_{\mathbb{R}^{d}} u(t, x) \psi(x) d x=& \int_{\mathbb{R}^{d}} \phi(x) \psi(x) d x+\frac{1}{2} \int_{t}^{T} \int_{\mathbb{R}^{d}} u(s, x) \Delta \psi(x) d x d s\notag \\
&+\int_{t}^{T} \int_{\mathbb{R}^{d}} u(s, x) \psi(x) {W}(ds,  x)dx,\label{weak}
\end{align}
since
\begin{equation} \lim_{\ep,\eta\rightarrow 0} \int_{t}^{T} \int_{\mathbb{R}^{d}} u^{\ep,\eta}(s, x) \psi(x) \dot{W}_{\ep,\eta}(s,  x)dsdx = \int_{t}^{T} \int_{\mathbb{R}^{d}} u(s, x) \psi(x) {W}(ds,  x)dx.
\label{e.5.11}
\end{equation}
In fact,    \eqref{e.5.11} can be deduced in  a similar way
 to   that of  Theorem \ref{yw converge}. This proves the conclusion.
\end{proof}
\begin{thm}
Suppose the same conditions as in Theorem \ref{th1} and let $(Y_s^{t,x},Z_s^{t,x})$ be  the solution pair of BSDE \eqref{marBSDE}.
Then  $u(t,x) :=Y_t^{t,x}, t\in[0,T],x\in\RR^d$  is in  $  C([0,T]\times\RR^d,\RR)$ and   is the solution of SPDE \eqref{inveranderson}.
\end{thm}
\begin{proof}
Notice that $u\left(t+h, X_{t+h}^{t, x}\right)=Y_{t+h}^{t+h, X_{t+h}^{t,x}}=Y_{t+h}^{t, x}$. We still use the approximated   BSDE \eqref{apmarBSDE}. Define $u^{\ep,\eta}(t,x) :=Y_t^{\ep,\eta,t,x}, t\in[0,T],x\in\RR^d$.  We want to show that
$u^{\ep,\eta}(t,x)  $  satisfies \eqref{inveranderson}.
An application of   It\^o's formula yields  that for $h>0$
\begin{equation}\label{part}
\begin{aligned}
  u^{\ep,\eta}\left(t+h, X_{t}^{t, x}\right)-u^{\ep,\eta}\left(t+h, X_{t+h}^{t, x}\right)
=& -\int_{t}^{t+h} \frac{1}{2}\Delta u^{\ep,\eta}\left(t+h, X_{s}^{t, x}\right) d s\\
&\quad -\int_{t}^{t+h} \nabla u^{\ep,\eta}\left(t+h, X_{s}^{t, x}\right)  d B_{s}  \,.
\end{aligned}
\end{equation}
Combining this with the backward SDE satisfied by $u^{\ep,\eta}(t,x) :=Y_t^{\ep,\eta,t,x}, t\in[0,T],x\in\RR^d$ we have
\begin{equation}\label{part}
\begin{aligned}
u^{\ep,\eta}(t+h, x)-&u^{\ep,\eta}(t, x)=  u^{\ep,\eta}\left(t+h, X_{t}^{t, x}\right)-u^{\ep,\eta}\left(t+h, X_{t+h}^{t, x}\right)+u^{\ep,\eta}\left(t+h, X_{t+h}^{t, x}\right)-u^{\ep,\eta}(t, x) \\
=&-\int_{t}^{t+h} \frac{1}{2}\Delta u^{\ep,\eta}\left(t+h, X_{s}^{t, x}\right) d s-\int_{t}^{t+h} \nabla u^{\ep,\eta}\left(t+h, X_{s}^{t, x}\right)  d B_{s} \\
&-\int_{t}^{t+h}  Y_{s}^{\ep,\eta,t, x} \dot{W}_{\ep,\eta}(s,X_{s}^{t, x})ds+\int_{t}^{t+h} Z_{s}^{\ep,\eta ,t, x} d B_{s}.
\end{aligned}
\end{equation}
Thus, let $\pi_{n}$ be a partition $t=t_{0}<t_{1}<\cdots<t_{n}=T .$ By \eqref{part}, we have
\begin{equation}
\begin{aligned}
\phi(x)-u^{\ep,\eta}(t, x)=&-\sum_{i=0}^{n-1} \int_{t_{i}}^{t_{i+1}}\frac{1}{2}\Delta u^{\ep,\eta}\left(t_{i}, X_{s}^{t_i, x}\right) d s-\sum_{i=0}^{n-1} \int_{t_{i}}^{t_{i+1}} u^{\ep,\eta}(s, X_{s}^{t_i, x}) \dot{W}_{\ep,\eta}\left( s, X_{s}^{t_i, x}\right)ds\\
&+\sum_{i=0}^{n-1} \int_{t_{i}}^{t_{i+1}}\left[Z_{s}^{\ep,\eta ,t_i, x}-\nabla u^{\ep,\eta}\left(t_{i}, X_{s}^{t_i, x}\right) \right] d B_{s}. \,.
\end{aligned}
\end{equation}
On the other hand, it is elementary to show that random field $\{ Z_{s}^{\ep,\eta,t, x}, t\leq s\leq T\}$ has a continuous version (e.g., \cite[Proposition 5.2]{SSZ2019})   such that
\begin{equation}
Z_s^{\ep,\eta,t, x} = D_s^BY_s^{\ep,\eta,t, x} =\nabla Y_s^{\ep,\eta,t, x}(\nabla X_s^{t, x})^{-1},
\end{equation}
and in particular, $Z_t^{\ep,\eta,t, x} = \nabla Y_t^{\ep,\eta,t, x}$.
%\begin{equation}
%Z_t^{\ep,\eta,t, x} = DY_t^{\ep,\eta,t, x} =\nabla Y_t^{\ep,\eta,t, x} \,.
%\end{equation}
Thus, if we let mesh sizes of the partitions $\pi_{n}$ go to zero, then it yields
\begin{equation}
\begin{aligned}
\phi(x)-u^{\ep,\eta}(t, x)=&-   \int_{0}^{t }\frac{1}{2}\Delta u^{\ep,\eta}\left(s, x\right) d s-  \int_{0}^{t } u^{\ep,\eta}(s, x) \dot{W}_{\ep,\eta}\left(s, x\right)ds \,,
\end{aligned}
\end{equation}
or by Duhamel's principle we obtain
\begin{equation}
\begin{aligned}
\phi(x)-u^{\ep,\eta}(t, x)=&-     \int_{0}^{t }\int_{\RR^d}
p_{t-s}(x-y)  u^{\ep,\eta}(s, y)\dot{W}_{\ep,\eta}\left(s, x\right)dyds \,.
\end{aligned}
\end{equation}
 From Theorem \ref{yw converge},  letting $\ep, \eta\to 0$   we get
\begin{equation}
\begin{aligned}
\phi(x)-u^{ }(t, x)=&-     \int_{0}^{t }\int_{\RR^d}
p_{t-s}(x-y)  u^{ }(s, y) W\left(d s, y \right) dy\,.
\end{aligned}
\end{equation}
%
%If we let mesh sizes of the partitions $\pi_{n}$ go to zero, we directly deduce \eqref{inverapSPDE} is valid and $u^{\ep,\eta}(t,x):= Y_t^{\ep,\eta,t,x}$ is the solution of approximation SPDE \eqref{apSPDE}.\\
%Hence, we shall use the Feynman-Kac formula to deduce $u^{\ep,\eta}(t,x) $ have a express as in \eqref{uep} and define $u(t,x)$ as in \eqref{FKu}. As the same method of Theorem \ref{SisB}, one can show $u^{\ep,\eta}(t,x) $ converges to $u(t,x)$, and then weak solution \eqref{weakep}   converges to \eqref{weak} in $L^2$ sense, which means $u^{\ep,\eta}(t,x)$ approximates the weak solution of \eqref{inveranderson}.  Finally, from Lemma \ref{Y converge}, we have the solution
The above formula means that $u(t,x):= Y_t^{t,x}$ of BSDE \eqref{marBSDE} is a mild solution of  SPDE \eqref{inveranderson}.
\end{proof}

\section{Appendix}
\begin{prop}Let $Y$ be a process such that its Malliavin derivative exists and
 assume that $D_{s,y} ^WY$ is integrable with respect to $s$.  Then
\begin{equation}\begin{split}\label{F2split}
\EE\left(\int_a^b Y_s W(ds,B_s)\right)^2 &= \EE\left[\int_{[a,b]^2}Y_r  Y_s  |s-r|^{2H_0-2} q(B_r,B_s)  drds\right]\\
&\quad+  \EE\Bigg[\int_{[a,b]^2}\int_{[r,b]}\int_{[a,s]}\int_{\RR^{2d}}D_{r,y}^W Y_u \, D_{v,w}^W Y_s |u-v|^{2H_0-2} \\
&\qquad \qquad   \qquad  |s-r|^{2H_0-2} q(B_u,w) q(B_s ,y)dw dydv du drds\Bigg]\,.
\end{split}\end{equation}
\end{prop}
\begin{proof}
%First we deal with $\EE\left[\left|\sum_{i=1}^{n}\int^{t_{i+1}}_{t_{i}} \left[ Y_r -Y_{t_i}\right] W(dr,B_r)\right|^2\right]$.
Recalling $W(\phi)=\int_{\RR_+\times \RR^d} \phi(t, x) W(dt, x)dx$. We have
\begin{equation}\begin{split}
\EE\left(\int_a^b Y_s W(ds,B_s)\right)^2 = \EE\left(\int_a^b\int_{\RR^d} Y_s\delta(B_s-x) W(ds,x)dx\right)^2.
\end{split}\end{equation}
Denote by $F:= \int_a^b Y_s W(ds,B_s)$ and  we shall use  $F \cdot {W}(\phi)=\delta(F\phi)+\langle D^W F, \phi\rangle_{\mathcal{H}} $. From the definition of spatial covariance
\eqref{e.1.7}, it follows
\begin{equation}\begin{split}\label{F2split1}
\EE\left(\int_a^b Y_s W(ds,B_s)\right)^2 &= \EE\left(F\cdot\int_a^b\int_{\RR^d} Y_s\delta(B_s-x) W(ds,x)dx\right)=\EE\left(\big\langle D^W F,\, Y_\cdot\,\delta(B_\cdot-\cdot) \big\rangle_\mathcal{H}\right)\\
&=\EE\left[\int_{[a,b]^2}\int_{\RR^{2d}} D_{r,y}^W F\cdot Y_s \delta(B_s-z) |s-r|^{2H_0-2} q(y,z)dydz drds\right]\\
&=\EE\left[\int_{[a,b]^2}\int_{\RR^{d}} D_{r,y}^W F\cdot Y_s  |s-r|^{2H_0-2} q(y,B_s)dy drds\right].
\end{split}\end{equation}
Note that,
\begin{equation}\begin{split}
D_{r,y}^W F&= D_{r,y}^W \left(\int_a^b\int_{\RR^d} Y_s\delta(B_s-x) W(ds,x)dx \right)\\
&=\int_r^b\int_{\RR^d} D_{r,y}^WY_s \delta(B_s-x) W(ds,x)dx + Y_r\delta(B_r-y).
\end{split}\end{equation}
Substituting this computation into\eqref{F2split}  we have
\begin{equation}
\EE\left(\int_a^b Y_s W(ds,B_s)\right)^2=I_1+I_2\,.
\label{e.i12}
\end{equation}
For $I_2$, it is easy to deduce
\begin{equation}\begin{split}
I_2=\, &\EE\left[\int_{[a,b]^2}\int_{\RR^{d}}Y_r\delta(B_r-y)\cdot Y_s  |s-r|^{2H_0-2} q(y,B_s)dy drds\right]\\
 =\,&\EE\left[\int_{[a,b]^2}Y_r  Y_s  |s-r|^{2H_0-2} q(B_r,B_s)dy drds\right]\,.
\end{split}\end{equation}
 $I_1$ has the following expression
\begin{equation}\begin{split}\label{DWyse}
I_1=\,&\EE\left[\int_{[a,b]^2}\int_{\RR^{d}}\int_r^b\int_{\RR^d} D_{r,y}^WY_s \delta(B_s-x) W(ds,x)dx\cdot Y_s  |s-r|^{2H_0-2} q(y,B_s)drdsdy\right]\\
=\,&\EE\left[\int_{[a,b]^2}\int_{\RR^{d}} I_3  |s-r|^{2H_0-2} q(y,B_s)drdsdy\right]
\end{split}\end{equation}
where
\[
I_3=\displaystyle\int_r^b\int_{\RR^d} D_{r,y}^WY_s \delta(B_s-x) W(ds,x)dx\cdot Y_s\,.
\]
Using $F \cdot {W}(\phi)=\delta(F\phi)+\langle D^W F, \phi\rangle_{\mathcal{H}} $ again we have
\begin{equation}\begin{split}
I_3=\, &\EE^W\left[\int_r^b\int_{\RR^d} D_{r,y}^WY_s \delta(B_s-x) W(ds,x)dx\cdot Y_s\right] = \EE^W\big\langle D_{r,y}^W Y_\cdot \delta(B_\cdot-\cdot),D^W Y_s \big\rangle_\mathcal{H}\\
=\,&\EE^W\left[ \int_{[r,b]}\int_{[a,s]}\int_{\RR^{2d}}D_{r,y}^W Y_u \delta(B_u-x) D_{v,w}^W Y_s |u-v|^{2H_0-2} q(x,w)dxdwdv du\right]\\
=\,&\EE^W\left[ \int_{[r,b]}\int_{[a,s]}\int_{\RR^{d}}D_{r,y}^W Y_u \, D_{v,w}^W Y_s |u-v|^{2H_0-2} q(B_u,w)dwdv du\right]\,.
\end{split}\end{equation}
Substituting this back  to \eqref{DWyse} we obtain
\begin{equation}\label{eq5.13}
\begin{split}
I_2=\,&\EE\bigg[\int_{[a,b]^2}\int_{[r,b]}\int_{[a,s]}\int_{\RR^{2d}}D_{r,y}^W Y_u \, D_{v,w}^W Y_s |u-v|^{2H_0-2}\\
&\qquad\quad |s-r|^{2H_0-2} q(B_u,w) q(y,B_s)dw dy dudv drds\bigg]\,.
 \end{split}
\end{equation}
Inserting the expressions for $I_1$ and $I_2$ into \eqref{e.i12} yields the proposition.
%\begin{align}
%&\EE\left[\left|\sum_{i=1}^{n}\int^{t_{i+1}}_{t_{i}} \left[ Y_r -Y_{t_i}\right] W(dr,B_r)\right|^2\right]\notag\\
%\leq &\sum_{i,j=1}^{n}\int^{t_{i+1}}_{t_{i}}\int^{t_{j+1}}_{t_{j}}(r-s)^{2H_0-2} \EE \left[(Y_r -Y_{t_i})(Y_s -Y_{t_j})\rho(B_r)\rho(B_s)\prod_{i=1}^d R_{H_i}(B_r^i,B_s^i) \right]drds\notag\\
%\leq &\sum_{i,j=1}^{n}\int^{t_{i+1}}_{t_{i}}\int^{t_{j+1}}_{t_{j}} \EE \big[|Y_r -Y_{t_i}|^4\big]^{1/4}\EE \big[|Y_r -Y_{t_j}|^4\big]^{1/4}(r-s)^{2H_0-2}\\
%&\qquad\qquad\qquad\qquad\qquad\qquad\qquad\times\EE \big[|\rho(B_r)\rho(B_s)\prod_{i=1}^d R_{H_i}(B_r^i,B_s^i)|^2 \big]^{1/2}drds\notag\\
%\leq &C \sum_{i,j=1}^{n}\int^{t_{i+1}}_{t_{i}}\int^{t_{j+1}}_{t_{j}} (r-t_i)^{1/2}(s-t_j)^{1/2}(r-s)^{2H_0-2}dsdr\\
%\leq & C \sum_{i,j=1}^{n} |t_{i+1}-t_i|^{H_0+1/2}|t_{j+1}-t_j|^{H_0+1/2}.\notag
%\end{align}
%\red{I don't know what conditions can we add to deduce the integrability of \ref{eq5.13} and the convergence of the first term of \eqref{tilR}. But the following estimate can be obtained since $\alpha_0^t$ is explicitly defined.}\\
%Moreover, to deduce the convergence of the second term in \eqref{tilR}, we first consider that, if we denote by
%$$G:=\int_a^b \alpha_0^s W(ds,B_s) =  \int_a^b \int_{\RR^d} \alpha_0^s \delta(B_s -x) W(ds,x)dx.$$
\end{proof}

\end{document}